\documentclass[11pt]{amsart}

\usepackage{graphicx}
\usepackage[labelsep=space]{caption}
\usepackage{amsfonts}
\usepackage{tikz-cd}
\usepackage{amsthm}
\usepackage{nccmath}
\usepackage{etex}
\usepackage{amssymb,mathtools}
\usepackage{mathrsfs}
\usepackage{dsfont}
\usepackage[all]{xy}
\usepackage{microtype}

\usepackage[utf8]{inputenc}
\usepackage[T1]{fontenc}
\usepackage{lmodern} 
\usepackage{latexsym}
\usepackage{MnSymbol}
\usepackage{nicefrac}
\usepackage{microtype}
\usepackage{color}
\usepackage{tikz-cd}
\usepackage{blindtext}
\usepackage{enumitem}
\usepackage{old-arrows}
\usepackage{empheq}
\usepackage{extpfeil}
\usepackage{hyperref}
\usepackage{theoremref}
\usepackage{float}
\usepackage{caption}
\usepackage{subcaption}

\makeatletter
\renewenvironment{proof}[1][\proofname]{%
	\par\pushQED{\qed}\normalfont%
	\topsep6\p@\@plus6\p@\relax
	\trivlist\item[\hskip\labelsep\bfseries#1\@addpunct{.}]%
	\ignorespaces
}{%
	\popQED\endtrivlist\@endpefalse
}
\makeatother

\usepackage{geometry}
\usepackage{secdot}
\usetikzlibrary{decorations.markings}
\tikzset{double line with arrow/.style args={#1,#2}{decorate,decoration={markings,%
			mark=at position 0 with {\coordinate (ta-base-1) at (0,1pt);
				\coordinate (ta-base-2) at (0,-1pt);},
			mark=at position 1 with {\d[#1] (ta-base-1) -- (0,1pt);
				\d[#2] (ta-base-2) -- (0,-1pt);
}}}}
\usepackage{ifxetex}
\usepackage{etoolbox}
\ifxetex
\usepackage{unicode-math}
\makeatletter
\patchcmd{\arrowfill@}{-7mu}{-14mu}{}{}
\patchcmd{\arrowfill@}{-7mu}{-14mu}{}{}
\patchcmd{\arrowfill@}{-2mu}{-4mu}{}{}
\patchcmd{\arrowfill@}{-2mu}{-4mu}{}{}
\makeatother
\fi

\theoremstyle{plain}
\newtheorem{theorem}{Theorem}[section]

\newtheorem{corollary}[theorem]{Corollary}
\newtheorem{proposition}[theorem]{Proposition}
\theoremstyle{remark}
\newtheorem{remark}[theorem]{Remark}
\theoremstyle{definition}

\usetikzlibrary{decorations.markings}
\tikzset{double line with arrow/.style args={#1,#2}{decorate,decoration={markings,%
			mark=at position 0 with {\coordinate (ta-base-1) at (0,1pt);
				\coordinate (ta-base-2) at (0,-1pt);},
			mark=at position 1 with {\draw[#1] (ta-base-1) -- (0,1pt);
				\draw[#2] (ta-base-2) -- (0,-1pt);
}}}}

\newtheorem{theoremx}{Theorem}

\begin{document}

\title[$B_{n,m}(\mathbb{R}P^2)$ AND GENERALISED FADELL--NEUWIRTH SHORT EXACT SEQUENCE]{THE BRAID GROUP $B_{n,m}(\mathbb{R}P^2)$ AND THE SPLITTING PROBLEM OF
	THE GENERALISED FADELL--NEUWIRTH SHORT EXACT SEQUENCE}
\author{Stavroula Makri}
\address{Normandie Univ., UNICAEN, CNRS, LMNO, 14000 Caen, France}
\email{stavroula.makri@unicaen.fr}

\begin{abstract}
	Let $n,m\in \mathbb{N}$, and let $B_{n,m}(\mathbb{R}P^2)$ be the set of $(n + m)$-braids of the projective plane whose associated permutation lies in the subgroup $S_n\times S_m$ of the symmetric group $S_{n+m}$.
	We study the splitting problem of the following generalisation of the Fadell--Neuwirth short exact sequence:
	$$1\rightarrow B_m(\mathbb{R}P^2 \setminus \{x_1,\dots,x_n\})\rightarrow 
	B_{n,m}(\mathbb{R}P^2)\xrightarrow{\bar{q}} B_n(\mathbb{R}P^2)\rightarrow 1,$$
	where the map $\bar{q}$ can be considered geometrically as the epimorphism that forgets the last $m$ strands, as well as the existence of a section of the corresponding fibration $q:F_{n+m}(\mathbb{R}P^2)/S_n\times S_m\to F_{n}(\mathbb{R}P^2)/S_n$, where we denote by $F_n(\mathbb{R}P^2)$ the $n^{th}$ ordered configuration space of the projective plane $\mathbb{R}P^2$.
	
	Our main results are the following: if $n=1$ the homomorphism $\bar{q}$ and the corresponding fibration $q$ admits no section, while if $n=2$, then $\bar{q}$ and $q$ admit a section. For $n\geq 3$, we show that if $\bar{q}$ and $q$ admit a section then $m\equiv 0, (n-1)^2\ \textrm{mod}\ n(n-1)$.  Moreover, using geometric constructions, we show that the homomorphism $\bar{q}$ and the fibration $q$ admit a section for $m=kn(2n-1)(2n-2)$, where $ k\geq1$, and for $m=2n(n-1)$.
	In addition, we show that for $m\geq3$, $B_m(\mathbb{R}P^2\setminus\{x_1,\dots,x_n\})$ is not residually nilpotent and for $m\geq 5$, it is not residually solvable.
\end{abstract}
\maketitle

\section{Introduction}
	Surface braid groups are both a generalisation, to any connected surface, of the fundamental group of the surface and of the braid groups of the plane, known as Artin braid groups and defined by Artin in 1925 in \cite{artin1925theorie}. They were initially introduced by Zariski in \cite{zariski1936poincare}, \cite{zariski1937topological}. During the 1960's, Fox gave an equivalent definition of surface braid groups from a topological point of view in terms of configuration spaces. Let $\Sigma$ be a connected surface and $n, m\in \mathbb{N}$. We denote the $n^{th}$ ordered configuration space of the surface $\Sigma$ by: 
	$$F_n(\Sigma)=\{(p_1,\dots,p_n)\in\Sigma^n\ | \ p_i\neq p_j \ \text{for all}\ i,j\in\{1,\dots,n\},\ i\neq j\}.$$
	Considering the natural action of the symmetric group $S_n$ on $F_n(\Sigma)$, defined by permuting coordinates, the $n^{th}$ unordered configuration space is the orbit space $F_n(\Sigma)/S_n$, and will be denoted by $UF_n(\Sigma)$. Similarly, we can consider the space obtained by quotienting the
	$(n+m)^{th}$ configuration space of $\Sigma$ by the subgroup $S_n\times S_m$ of $S_{n+m}$, that is  $UF_{n,m}(\Sigma)=F_{n+m}(\Sigma)/S_n\times S_m$. Fox--Neuwirth \cite{fox1962braid} proved that:
	$$\pi_1(F_n(\Sigma))\cong P_n(\Sigma)\ \text{and}\ \pi_1(UF_n(\Sigma))\cong B_n(\Sigma).$$
	We call $P_n(\Sigma)$ the pure braid groups of $\Sigma$ and $B_n(\Sigma)$ the full braid groups of $\Sigma$.
	Similarly, $\pi_1(UF_{n,m}(\Sigma))\cong B_{n,m}(\Sigma)$, where $B_{n,m}(\Sigma)$ is a subgroup of $B_{n+m}(\Sigma)$ known as a mixed braid group of $\Sigma$. 
	
	For $n>m$ we may define the projection $$p_{n,m}: F_n(\Sigma)\to F_m(\Sigma)\ \text{given by}\ p_{n,m}(x_1,\dots,x_m,x_{m+1},\dots,x_n) = (x_1,\dots,x_m).$$ 
	In \cite{fadell1962configuration}, Fadell--Neuwirth proved that for $n,m\in\mathbb{N}$, where $1\leq m<n$, and for any connected surface $\Sigma$ with empty boundary, the map $p_{n,m}$ 
	 is a locally-trivial fibration. The fibre over a point $(x_1,\dots,x_m)$ of the base space $F_m(\Sigma)$ may be identified with the $(n-m)^{th}$ configuration space of $\Sigma$ with $m$ punctures, $F_{n-m}(\Sigma\setminus\{x_1,\dots,x_m\})$, which we interpret as a subspace of the total space $F_n(\Sigma)$ via the injective map $i: F_{n-m}(\Sigma\setminus\{x_1,\dots,x_m\})\to F_n(\Sigma)$, defined by $i(y_1,\dots,y_{n-m})= (x_1,\dots,x_m,y_1,\dots,y_{n-m})$. This fibration gives rise to a long exact sequences of homotopy groups and in turn to the Fadell--Neuwirth pure braid group short exact sequence: 
	\begin{equation}\label{s1}
	\begin{tikzcd}
	1 \ar[r]& P_{n-m}(\Sigma\setminus\{x_1,\dots,x_m\}) \ar[r, "\bar{i}"] & P_n(\Sigma) \ar[r, "\bar{p}_{n,m}"] & P_m(\Sigma)\ar[r] & 1,
	\end{tikzcd}
	\end{equation}
	where $m\geq3$ when $\Sigma$ is the $2-$sphere $\mathbb{S}^2$, $m\geq2$ when $\Sigma$ is the projective plane $\mathbb{R}P^2$, and where $\bar{i}$ and $\bar{p}_{n,m}$ are the homomorphisms induced by the maps $i$ and $p_{n,m}$ respectively. The homomorphism $\bar{p}_{n,m}:P_n(\Sigma)\to P_m(\Sigma)$ can be considered geometrically as the map that forgets the last $n-m$ strands of the braid. Note that \eqref{s1} exists for any connected surface $\Sigma$ with or without boundary, see [\cite{guaschi2013survey}, Remark 8(d)]. The splitting problem refers to the question of whether or not the homomorphism $\bar{p}_{n,m}$ admits a section and it was a central question during the foundation and the development of the theory of surface braid groups during the 1960’s. It was
	studied notably by Fadell \cite{fadellHomotopy}, Fadell--Neuwirth \cite{fadell1962configuration}, Fadell--Van Buskirk \cite{fadell1962braid}, Van Buskirk \cite{van1966braid} and Birman \cite{https://doi.org/10.1002/cpa.3160220104}, approaching it either geometrically or algebraically. The splitting of \eqref{s1} in the case of the plane, $\Sigma=\mathbb{R}^2$, is an important result in classical braid theory. Based on this result, we have the Artin normal form for pure Artin groups $P_n$, that is $P_n\cong F_{n-1}\rtimes F_{n-2}\rtimes\dots\rtimes F_2\rtimes F_1$, where $F_n$ is the free group of rank $n$, and is one of the main tools in the study of $P_n$. The Artin normal form of $P_n$ is used, for instance, for the uniqueness of roots in $P_n$, by Bardakov \cite{BardakovOn}, for the study of the lower central series and the residual nilpotence of $P_n$, by Falk--Randell \cite{falk}, as well as for the proof of the bi-orderability of $P_n$, by Kim--Rolfsen \cite{RolfsenKim}. A complete solution to the splitting problem of \eqref{s1} was given by Gon\c{c}alves--Guaschi in \cite{gonccalves2010braid}. 
	
	In order to generalise \eqref{s1} to non-pure subgroups, we consider the map $q: UF_{n,m}(\Sigma)\to UF_n(\Sigma)$, defined by forgetting the last $m$ coordinates. 
For any connected surface $\Sigma$ without boundary, the map $q$ is a locally-trivial fibration, whose fibre can be identified with the unordered configuration space $UF_m(\Sigma\setminus\{x_1,\dots,x_n\})$. From the long exact sequence in
	homotopy of the fibration $q$ we obtain the so-called generalised Fadell--Neuwirth short exact sequence:
	\begin{equation}\label{s2}
		\begin{tikzcd}
		1 \ar[r]& B_m(\Sigma\setminus\{x_1,\dots,x_n\}) \ar[r] & B_{n,m}(\Sigma) \ar[r, "\bar{q}"] & B_n(\Sigma)\ar[r] & 1,
		\end{tikzcd}
	\end{equation}
		where $n\geq3$ when $\Sigma=\mathbb{S}^2$ and $n\geq2$ when $\Sigma=\mathbb{R}P^2$, and where $\bar{q}$ is the homomorphism induced by the map $q$. The homomorphism $\bar{q}:B_{n,m}(\Sigma)\to B_n(\Sigma)$ can be considered geometrically as the epimorphism that forgets the last $m$ strands of the braid.  Note that, as in the pure case, \eqref{s2} exists for any connected surface $\Sigma$ with or without boundary. The question of whether or not the generalised Fadell--Neuwirth short exact sequence \eqref{s2} splits does not yet have a complete solution. To the best of our knowledge, the only surface, besides the plane $\mathbb{R}^2$, for which this problem has been studied is $\Sigma=\mathbb{S}^2$, by Gon\c{c}alves--Guaschi in \cite{gonccalves2005braid} and more recently by Chen--Salter in \cite{MR4185935}. Along with the braid groups of the $2-$sphere, the braid groups of the projective plane are particularly interesting since they are the only surface braid groups that contain torsion elements \cite{fadell1962configuration}, \cite{fadell1962braid}, \cite{van1966braid}. 
		
In this work we will mainly be concerned with the case $\Sigma=\mathbb{R}P^2$:
	\begin{equation}\label{ms}
\begin{tikzcd}
1 \ar[r]& B_m(\mathbb{R}P^2\setminus\{x_1,\dots,x_n\}) \ar[r] & B_{n,m}(\mathbb{R}P^2) \ar[r, "\bar{q}"] & B_n(\mathbb{R}P^2)\ar[r] & 1,
\end{tikzcd}
\end{equation}
and the corresponding fibration $q: UF_{n,m}(\mathbb{R}P^2)\to UF_n(\mathbb{R}P^2)$, and in particular the question
of whether or not they admit a section, for $n,m\in\mathbb{N}$. 
Our two main approaches are the following: one using group-theoretic aspects via the presentations of the braid groups of \eqref{ms}, given in Section \ref{S2}, Proposition \ref{pres3} and Theorem \ref{pre4}, which will allow us to give necessary conditions for a splitting, given in Section \ref{S5}, and the other geometrical, topological, which give sufficient conditions by constructing, for certain values of $m$, two cross-sections for the fibration $q: UF_{n,m}(\Sigma)\to UF_n(\Sigma)$, given in Section \ref{S6}.

Our main results are the following:
\begin{theoremx}\label{th1}
Let $m\in\mathbb{N}$. If $n=2$, the short exact sequence \eqref{ms} splits.
\end{theoremx}
\noindent
We prove this result in Section \ref{S3}, by giving an explicit algebraic and geometric section. In the same section, we will also show that for $n=1$ the homomorphism $\bar{q}:B_{n,m}(\mathbb{R}^2)\to B_n(\mathbb{R}^2)$ admits no section. 

For $n\geq3$ we have:
\begin{theoremx}\label{th2}
Let $n\geq 3$.
If the short exact sequence \eqref{ms} splits, then $m\equiv 0, (n-1)^2\ \textrm{mod}\ n(n-1)$.
\end{theoremx}
\noindent
This theorem will be proved in Section \ref{S5}. In order to prove this result, we study the torsion elements of $B_n(\mathbb{R}P^2)$ as well as the splitting problem for the quotient group $B_{n,m}(\mathbb{R}P^2)/\Gamma_2(B_m(\mathbb{R}P^2\setminus\{x_1,\dots, x_n\}))$. This method has been used in the study of the splitting of the Fadell--Neuwirth pure braid group short exact sequence and for the generalised one for $\Sigma=\mathbb{S}^2$ in \cite{gonccalves2010braid} and \cite{gonccalves2005braid} respectively, and is based on the following observation: let $G$ be a group,
and let $K, H$ be normal subgroups of $G$ such that $H$ is contained in $K$. If the extension $1\rightarrow K \rightarrow G\rightarrow Q\rightarrow 1$ splits then so does the extension $1\rightarrow K/H \rightarrow G/H\rightarrow Q\rightarrow 1$. The study of the possible splitting of the second extension is sometimes easier than that of the first extension, and thus, supposing that this second extension splits, we may obtain some restrictions on the values of $m$. This method of studying the quotient $B_{n,m}(\mathbb{R}P^2)/\Gamma_2(B_m(\mathbb{R}P^2\setminus\{x_1,\dots, x_n\}))$ gives restrictions on the values of $m$. One could ask whether the study of the quotient of $B_{n,m}(\mathbb{R}P^2)$ by higher elements of the lower central or derived series give more information. For most $m$, the answer is no because the lower central and derived series become constant, due to the following theorem, proven in Section \ref{S4}. 

For $n, m\geq 1$, let
$\beta_{n,m}=B_m(\mathbb{R}P^2\setminus \{x_1,\dots,x_n\})$.
\begin{theoremx}\label{th_l}
	Let $n\geq 1$. Then:
	\begin{itemize}
		\item If $m\geq3$, then $\Gamma_2\big(\beta_{n,m}\big)=\Gamma_3\big(\beta_{n,m}\big).$ 
		\item If $m\geq5$, then $\big(\beta_{n,m}\big)^{(1)}=\big(\beta_{n,m}\big)^{(2)}.$
	\end{itemize}
	In particular, for $m\geq3$, $\beta_{n,m}$ is not residually nilpotent and for $m\geq 5$, it is not residually solvable.
\end{theoremx}
In Section \ref{S6}, we conclude this work by constructing two different cross-sections $s:UF_n(\mathbb{R}P^2)\to UF_{n,m}(\mathbb{R}P^2)$, for certain values of $m$. 
\begin{theoremx}\label{th3}
Let $n\geq 3$. The fibration $$q: UF_{n,m}(\mathbb{R}P^2)\to UF_n(\mathbb{R}P^2)$$ admits a cross-section for $m=kn(2n-1)(2n-2),\ \text{where}\ k\geq1$, and for $m=2n(n-1)$.
\end{theoremx}
\noindent
Applying a standard argument given in [\cite{gonccalves2004braid}, Proposition 7], it follows that the fibration $q: UF_{n,m}(\mathbb{R}P^2)\to UF_n(\mathbb{R}P^2)$ admits a cross-section if and only if the short exact sequence \eqref{ms} splits.
This justifies Corollary \ref{imply}, which states that for $n\geq 3$, the short exact sequence \eqref{ms} splits for $m=2n(n-1)$ and $m=kn(2n-1)(2n-2)$, where $k\geq1$.
Note that these values are compatible with the statement of Theorem \ref{th2}.

\section{Presentation of $B_{n,m}(\mathbb{R}P^2)$ and of certain subgroups}\label{S2}
Our convention is that, throughout this text, we read the elements of the braid groups from left to right.
We consider the real projective plane as the quotient space obtained by identifying the antipodal points on the boundary circle of the closed $2$-disc $\mathbb{D}^2$, as in Figure \ref{proj_gen}.
A presentation of the braid groups $B_n(\mathbb{R}P^2)$, of the projective plane was given by Van Buskirk:

\begin{theorem}[Van Buskirk, \cite{van1966braid}]\label{vb}
	Let $n\in \mathbb{N}$. The following constitutes a presentation of $B_n(\mathbb{R}P^2)$.
	\\
	\\\textbf{Generators:} $\sigma_1, \dots, \sigma_{n-1}, \rho_1,\dots, \rho_n$.
	\\\textbf{Relations:} \begin{enumerate}[label=(\roman*)]
		\item $\sigma_i\sigma_j=\sigma_j\sigma_i$, \ for \ $|i-j|>1$,
		\item $\sigma_i\sigma_{i+1}\sigma_i = \sigma_{i+1}\sigma_i\sigma_{i+1}$,\ for \ $1\leq i\leq n-2,$
		\item $\sigma_i\rho_j = \rho_j\sigma_i$, \ for \ $j\neq i, i+1$,
		\item $\rho_i = \sigma_i\rho_{i+1}\sigma_i $,\ for \ $1\leq i\leq n-1,$
		\item $\sigma_i^2 = \rho_{i+1}^{-1}\rho_i^{-1}\rho_{i+1}\rho_i$,\ for \ $1\leq i\leq n-1,$
		\item $\rho_1^2 =\sigma_1\sigma_2\dots\sigma_{n-2}\sigma_{n-1}^2\sigma_{n-2}\dots\sigma_2\sigma_1$.
	\end{enumerate}	
\end{theorem}

The generators $\sigma_1, \dots, \sigma_{n-1}$ are the standard Artin generators, and for $i=1,\dots,n$, $\rho_i$ corresponds geometrically to elements of the fundamental
group of the projective plane based at the $i^{th}$ basepoint, as we see in Figure \ref{proj_gen}. Note that the elements $\sigma_i^{-1}$ and $\rho_i^{-1}$ correspond geometrically to the elements presented in Figure \ref{proj_gen} but with the arrows reversed. 
From this presentation, we observe that $B_1(\mathbb{R}P^2)=\langle\rho_1\ | \ \rho_1^2=1\rangle\cong\mathbb{Z}_2$, and from [\cite{van1966braid}, p. 87], we know that $B_2(\mathbb{R}P^2)$ is isomorphic to the dicyclic group of order 16 and in particular $B_2(\mathbb{R}P^2)=\langle \sigma_1,\rho_1\ |\ (\sigma_1\rho_1^{-1})^4=\rho_1^2=\sigma_1^2\rangle$. If $n\geq 3$, the group $B_n(\mathbb{R}P^2)$ is infinite, [\cite{van1966braid}, p. 86].

\begin{figure}[H]
	\centering
	\includegraphics[width=.5\textwidth]{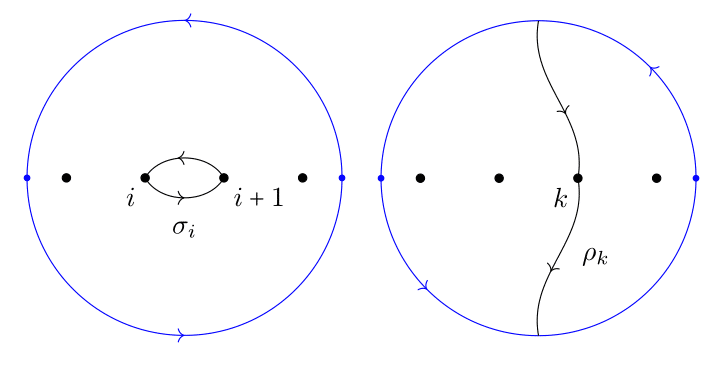}
	\caption{Generators $\sigma_i$ and $\rho_k$ of $B_n(\mathbb{R}P^2)$.}
	\label{proj_gen}
\end{figure}

The aim of this section is to give a presentation of the group $B_{n,m}(\mathbb{R}P^2)$.
To do so, we will use the short exact sequence \eqref{ms} and standard results about presentations of group extensions. Thus, we first need a presentation of $B_m(\mathbb{R}P^2\setminus\{x_1,\dots,x_n\})$, which we can obtain making use of the short exact sequence 
$$1\rightarrow P_m(\mathbb{R}P^2 \setminus \{x_1,\dots,x_n\})\rightarrow B_m(\mathbb{R}P^2\setminus \{x_1,\dots,x_n\})\rightarrow S_m\rightarrow 1,$$
but then, we first have to give a presentation of $P_m(\mathbb{R}P^2 \setminus \{x_1,\dots,x_n\})$. 

We start with the following proposition that provides a presentation of the pure braid groups $P_n(\mathbb{R}P^2)$. Note that there is already one presentation of the group $P_n(\mathbb{R}P^2)$ in \cite{gonccalves2007braid}, but we give a different one here, where our generators $\rho_i$ differ from those in \cite{gonccalves2007braid}.
\begin{proposition}\label{pre1}
	The following constitutes a presentation of $P_n(\mathbb{R}P^2)$, for $n\in\mathbb{N}$.
	\\
	\\\textbf{Generators:} $B_{i, j}, \ \text{for} \ 1\leq i<j\leq n \ \text{and} \ \rho_k, \ \text{for} \ 1\leq k\leq n$.
	\\\textbf{Relations:} \begin{enumerate}[label=(\roman*)]
		\item The Artin relations between the generators $B_{i,j}$ coming from those of $P_n$, namely for \ $1\leq i<j \leq n$ and $1\leq r<s \leq n$,
		\begin{equation*}
		B_{r,s}B_{i,j}B_{r,s}^{-1}=
		\begin{cases}
		B_{i,j}, &i<r<s<j \ \text{or} \ r<s<i<j. \\
		B_{i,j}^{-1}B_{r,j}^{-1}B_{i,j}B_{r,j}B_{i,j}, &r<i=s<j.\\
		B_{s,j}^{-1}B_{i,j}B_{s,j}, &i=r<s<j.\\
		B_{s,j}^{-1}B_{r,j}^{-1}B_{s,j}B_{r,j}B_{i,j}B_{r,j}^{-1}B_{s,j}^{-1}B_{r,j}B_{s,j}, & r<i<s<j.
		\end{cases}
		\end{equation*}

		\item $ \rho_i (B_{1,i}\dots B_{i-1,i}) = (B_{i,i+1}\dots B_{i,n})\rho_i^{-1}, \ \text{for} \ 1\leq i \leq n$, called surface relations.
		\item $ \rho_i\rho_j\rho_i^{-1} = \Big({\displaystyle \prod_{l=i+1}^{j-1} B_{l,j}}\Big)^{-1}B_{i,j}\Big({\displaystyle \prod_{l=i+1}^{j-1} B_{l,j}}\Big)\rho_j,\ \text{for} \ 1\leq i   < j \leq n.$
		\item For $1\leq i<j\leq n$,  $ 1\leq k\leq n$ and $k\neq j$,\\
		\begin{equation*}
		\rho_kB_{i, j}\rho_k^{-1}=
		\begin{cases}
		B_{i,j}, & \text{for}\ j<k \ \text{or} \ k<i .\\
		\rho_j^{-1} \Big({\displaystyle \prod_{l=k+1}^{j-1} B_{l,j}}\Big)^{-1}B_{i,j}^{-1}\Big({\displaystyle \prod_{l=k+1}^{j-1} B_{l,j}}\Big)\rho_j   & \text{for}\ k=i .\\
		A_{j,k}B_{i,j}A_{j,k}^{-1}      & \text{for}\ i<k<j,
		\end{cases}
		\end{equation*}
		where $A_{j,k}:=\rho_j^{-1} \Big({\displaystyle \prod_{l=k+1}^{j-1} B_{l,j}}\Big)^{-1}B_{k,j}^{-1}\Big({\displaystyle \prod_{l=k+1}^{j-1} B_{l,j}}\Big)\rho_j$.
	\end{enumerate}
\end{proposition}
\begin{remark}\label{bgenerators}
	The generator $B_{i,j}$ is given by $B_{i,j} = \sigma_i^{-1}\cdots\sigma_{j-2}^{-1}\sigma_{j-1}^2\sigma_{j-2}\cdots\sigma_i$, or equivalently, by $B_{i,j} = \sigma_{j-1}\cdots\sigma_{i+1}\sigma_i^2\sigma_{i+1}^{-1}\cdots\sigma_{j-1}^{-1}$ (see Figure \ref{b}).
\end{remark}

\begin{figure}[H]
	\centering
	\includegraphics[width=.5\textwidth]{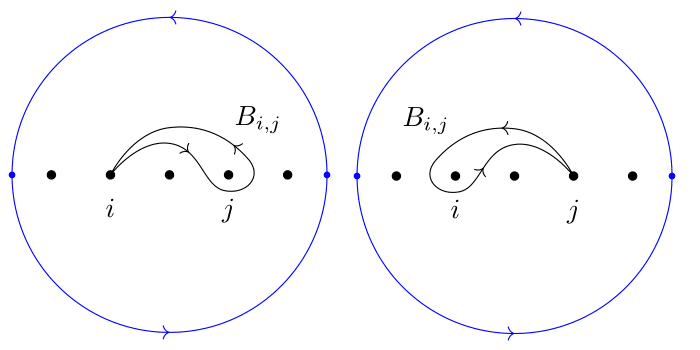}
	\caption{Equivalent illustrations of the generator $B_{i,j}$.}
	\label{b}
\end{figure}

	The products $\big({\displaystyle \prod_{l=i+1}^{j-1} B_{l,j}}\big)^{-1}B_{i,j}\big({\displaystyle \prod_{l=i+1}^{j-1} B_{l,j}}\big)$ and $\displaystyle\prod_{l = i+1}^{j}B_{i,l}$, for $i\leq j$, that appear frequently in the presentation of $P_n(\mathbb{R}P^2)$, can be seen geometrically as in Figures \ref{p1} and \ref{p2}.

\begin{figure}[H]
 	\centering
    \includegraphics[width=.6\textwidth]{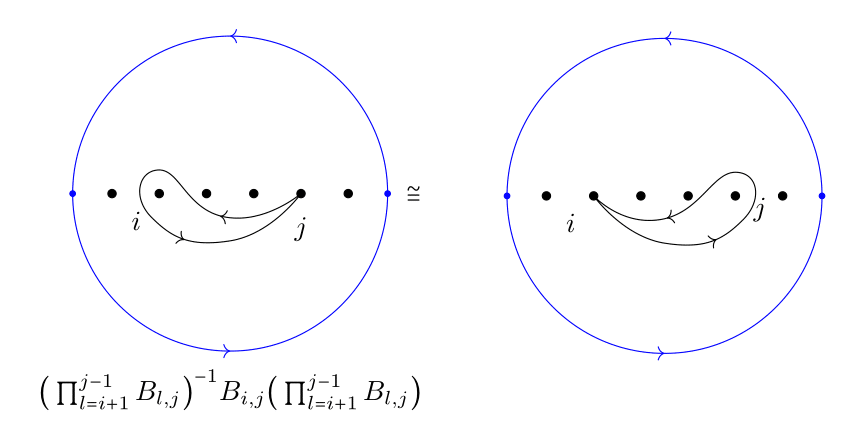}
    \caption{ A geometric presentation of the product $\big({\displaystyle \prod_{l=i+1}^{j-1} B_{l,j}}\big)^{-1}B_{i,j}\big({\displaystyle \prod_{l=i+1}^{j-1} B_{l,j}}\big)$.}
     \label{p1}
 \end{figure}
   \begin{figure}[H]
   	\centering
	\includegraphics[width=.2\textwidth]{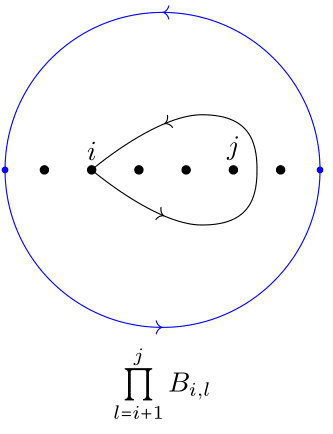}
	\caption{A geometric presentation of the product $\displaystyle\prod_{l = i+1}^{j}B_{i,l}$.}
	\label{p2}
   \end{figure}

\begin{proof}[Proof of Proposition \ref{pre1}]
	We will prove the statement by applying induction on $n\in\mathbb{N}$.
	
	For $n=1$ the given presentation yields $P_1(\mathbb{R}P^2)=\langle\rho_1\ | \ \rho_1^2=1\rangle\cong\mathbb{Z}_2$, which holds, since  $P_1(\mathbb{R}P^2)=\pi_1(\mathbb{R}P^2)\cong\mathbb{Z}_2$. Moreover, for $n=2$ we obtain that $P_2(\mathbb{R}P^2)=\langle\rho_1, \rho_2 \ |  \ \rho_1^2=\rho_2^2=(\rho_1\rho_2)^2\rangle\cong Q_8$, the quaternion group of order 8, which is indeed the case by Van Buskirk, \cite{van1966braid}. 
	
	Suppose that, for some $n\geq 2$, the group $P_n(\mathbb{R}P^2)$ has the given presentation. 
	We consider the corresponding Fadell--Neuwirth short exact sequence, where the map $\bar{p}_{n+1,n}$ can be considered geometrically as the epimorphism that forgets the last strand:	
	$$1\longrightarrow \pi_1(\mathbb{R}P^2 \setminus \{x_1,\dots, x_n\})\longrightarrow P_{n+1}(\mathbb{R}P^2)\overset{\bar{p}_{n+1,n}}{\longrightarrow}P_n(\mathbb{R}P^2) \longrightarrow 1.$$
	To retain the symmetry among relations we take the free group $\ker(\bar{p}_{n+1,n})\cong \pi_1(\mathbb{R}P^2 \setminus \{x_1,\dots, x_n\})$ to have the following presentation with $n+1$ generators and a single relation: 	
	$$\langle \rho_{n+1}, B_{1,n+1}, B_{2,n+1},\dots, B_{n,n+1} \ | \ \rho_{n+1} ( B_{1,n+1}B_{2,n+1}\cdots B_{n,n+1} ) \rho_{n+1} = 1 \rangle.$$
	The single relation of $\ker(\bar{p}_{n+1,n})\cong \pi_1(\mathbb{R}P^2 \setminus \{x_1,\dots, x_n\})$ is obtained geometrically, see Figure \ref{fn1}.
	
	In order to obtain a presentation of $P_{n+1}(\mathbb{R}P^2)$, we will apply standard techniques for obtaining presentations of group extensions as described in [\cite{johnsonpresentation}, p. 139].
	Based on these techniques, the required generating set of $P_{n+1}(\mathbb{R}P^2)$ is obtained by taking the union of the generators of $\pi_1(\mathbb{R}P^2 \setminus \{x_1,\dots, x_n\})$ together with coset representatives of the generators of $P_n(\mathbb{R}P^2)$, namely the generators $\rho_k$, for $1\leq k \leq n$ and $B_{i,j}$, for $1\leq i < j \leq n$. Thus the resulting generating set of $P_{n+1}(\mathbb{R}P^2)$ is $$\langle B_{i,j}, \rho_k \ | \ \text{for} \ 1\leq i < j \leq n+1,\ 1\leq k \leq n+1 \rangle.$$
	The relations that we obtain for the presentation of $P_{n+1}(\mathbb{R}P^2)$ arise from the following three classes of relations.
	
	The first class consists of the relations of $\ker(\bar{p}_{n+1,n})$, which in this case, it is the following single relation: $\rho_{n+1} ( B_{1,n+1}B_{2,n+1}\cdots B_{n,n+1} ) \rho_{n+1} = 1$. 
	
	The second class of relations is obtained by rewriting the relations of $P_n(\mathbb{R}P^2)$ in terms of the corresponding coset representatives in $P_{n+1}(\mathbb{R}P^2)$, and then expressing the resulting elements as a word in the generators of $\ker(\bar{p}_{n+1,n})$. In this way, one may see that all of the relations of $P_n(\mathbb{R}P^2)$ are lifted directly to relations of $P_{n+1}(\mathbb{R}P^2)$, except for the relation $$\rho_i (B_{1,i}\cdots B_{i-1,i}) = (B_{i,i+1}\cdots B_{i,n})\rho_i^{-1}, \ \text{for} \ 1\leq i \leq n,$$ which may be rewritten as $\rho_i (B_{1,i}\cdots B_{i-1,i})\rho_i(B_{i,i+1}\cdots B_{i,n})^{-1}= 1, \ \text{for} \ 1\leq i \leq n$. For the same values of $i$ we have: $$\rho_i (B_{1,i}\cdots B_{i-1,i})\rho_i(B_{i,i+1}\cdots B_{i,n})^{-1}={\displaystyle \Big(\prod_{l = i+1}^{n} B_{l,n+1}}\Big)^{-1}B_{i,n+1}\Big({\displaystyle \prod_{l=i+1}^{n} B_{l,n+1}}\Big),$$ see Figure \ref{iia}. It follows that $$\rho_i (B_{1,i}\cdots B_{i-1,i})\rho_i={\displaystyle \Big(\prod_{l = i+1}^{n} B_{l,n+1}}\Big)^{-1}B_{i,n+1}\Big({\displaystyle \prod_{l=i+1}^{n} B_{l,n+1}}\Big)(B_{i,i+1}\cdots B_{i,n}),$$ for $1\leq i \leq n$. Moreover, $${\displaystyle \Big(\prod_{l = i+1}^{n} B_{l,n+1}}\Big)^{-1}B_{i,n+1}\Big({\displaystyle \prod_{l=i+1}^{n} B_{l,n+1}}\Big)(B_{i,i+1}\cdots B_{i,n})=(B_{i,i+1}\cdots B_{i,n})B_{i,n+1},$$ see Figure \ref{iib}. Thus, $\rho_i (B_{1,i}\cdots B_{i-1,i})\rho_i=(B_{i,i+1}\cdots B_{i,n})B_{i,n+1}$ and finally we obtain the relation $\rho_i (B_{1,i}\cdots B_{i-1,i})=(B_{i,i+1}\cdots B_{i,n})B_{i,n+1}\rho_i^{-1}, \ \text{for} \ 1\leq i \leq n$. Combining this relation with the single relation of the presentation of $\ker(\bar{p}_{n+1,n})$, we obtain the following relation: $$\rho_i (B_{1,i}\cdots B_{i-1,i}) = (B_{i,i+1}\cdots B_{i,n}B_{i,n+1})\rho_i^{-1},\ \text{for}\ 1\leq i \leq n+1.$$ This is relation ($ii$) of the given presentation of $P_{n+1}(\mathbb{R}P^2)$. 
	
	The third class of relations is obtained by rewriting the conjugates of the generators of $\ker(\bar{p}_{n+1,n})$ by the coset representatives of the generators of $P_n(\mathbb{R}P^2)$ in $P_{n+1}(\mathbb{R}P^2)$ in terms of the generators of $\ker(\bar{p}_{n+1,n})$. Thus, we obtain:
	\begin{enumerate}[label=(\alph*)]
		\item The Artin relations for $B_{i,j}$. For all $1\leq i< j\leq n $ \ and \ $1\leq l \leq n$\ :
		\begin{equation*}
		B_{i,j}B_{l,n+1}B_{i,j}^{-1}=
		\begin{cases}
		B_{l,n+1}, & l<i. \\
		B_{l,n+1}, & j<l. \\
		B_{l,n+1}^{-1}B_{i,n+1}^{-1}B_{l,n+1}B_{i,n+1}B_{l,n+1}, & l=j.\\
		B_{j,n+1}^{-1}B_{l,n+1}B_{j,n+1}, & l=i.\\
		B_{j,n+1}^{-1}B_{i,n+1}^{-1}B_{j,n+1}B_{i,n+1}B_{l,n+1}B_{i,n+1}^{-1}B_{j,n+1}^{-1}B_{i,n+1}B_{j,n+1}, & i<l<j.
		\end{cases}
		\end{equation*} 
		
		We can obtain these relations geometrically or by using the presentation of $P_n$ given in \cite{hansen}, but with conjugation on the right. One can consider the generators $B_{i,j}$ to be words
		in $P_n$ and using the homomorphism $P_n\to P_n(\mathbb{R}P^2)$ induced by an inclusion
		$\mathbb{D}^2\subseteq \mathbb{R}P^2$ , these relations also hold in $P_n(\mathbb{R}P^2)$.
		\item For all $1\leq i < j \leq n$: $$B_{i,j}\rho_{n+1}B_{i,j}^{-1} = \rho_{n+1}.$$
		\item We obtain relation ($iii$) of the presentation in the case $n+1$, with $j=n+1$ in those relations.
		\item  We obtain relation ($iv$) of the presentation in the case $n+1$, with $j=n+1$ in those relations.
	\end{enumerate}
Relation (b) is clear and relations (c), (d) are obtained geometrically, as shown, for the general case, in Figure \ref{iii} and Figures \ref{ivb}, \ref{ivc}, respectively.
	
	We now see that relation ($i$) of the presentation of $P_{n+1}(\mathbb{R}P^2)$ is obtained from the lifted relation ($i$) of $P_n(\mathbb{R}P^2)$, together with relation (a) from the third class of relations. Relation ($ii$), namely the surface relations of the presentation of $P_{n+1}(\mathbb{R}P^2)$ has already been obtained by combining a relation from the second class with the single relation from the first class of relations. Relation ($iii$) is obtained from the lifted relation ($iii$) of $P_n(\mathbb{R}P^2)$, together with relation (c) from the third class of relations. Finally, relation ($iv$) of the presentation of $P_{n+1}(\mathbb{R}P^2)$ can be obtained from the lifted relation ($iv$) of $P_n(\mathbb{R}P^2)$, together with the relations (b) and (d) of the third class. 
	
	To sum up, the set of generators and the set of relations, that we have obtained, coincide with the given presentation for $P_{n+1}(\mathbb{R}P^2)$ and it follows by induction that the given presentation is valid for $P_n({\mathbb{R}P^2})$, for every $n\in \mathbb{N}$. 
\end{proof}
\begin{figure}[H]
	\centering
	\includegraphics[width=.6\textwidth]{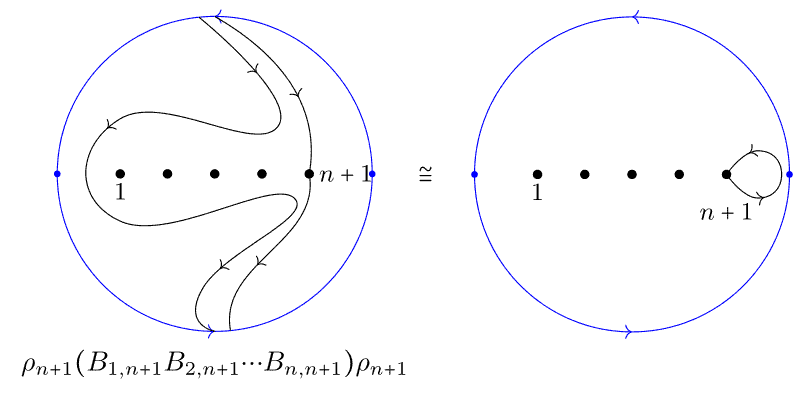}
	\caption{Relation $\rho_{n+1} ( B_{1,n+1}B_{2,n+1}\cdots B_{n,n+1} ) \rho_{n+1} = 1$.}
	\label{fn1}
\end{figure}

\begin{figure}[H]
	\centering
	\includegraphics[width=.6\textwidth]{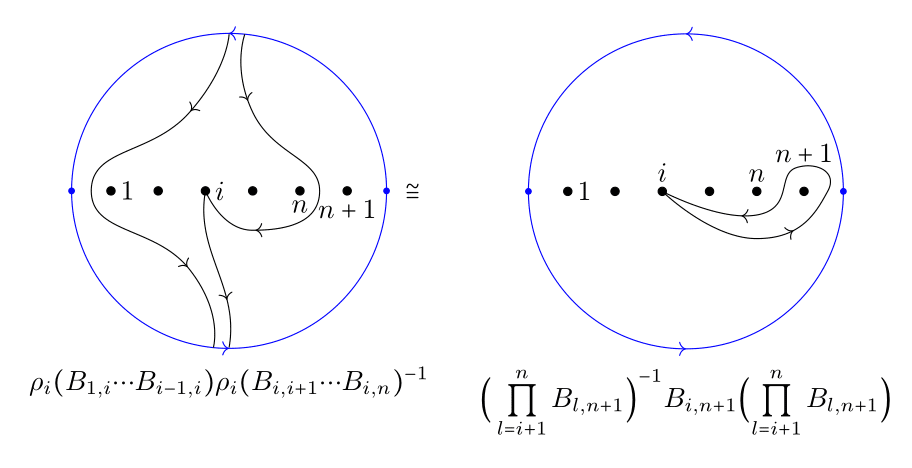}
	\caption{Relation ($iia$).}
	\label{iia}
	\end{figure}
	\begin{figure}[H]
	\centering
	\includegraphics[width=.6\textwidth]{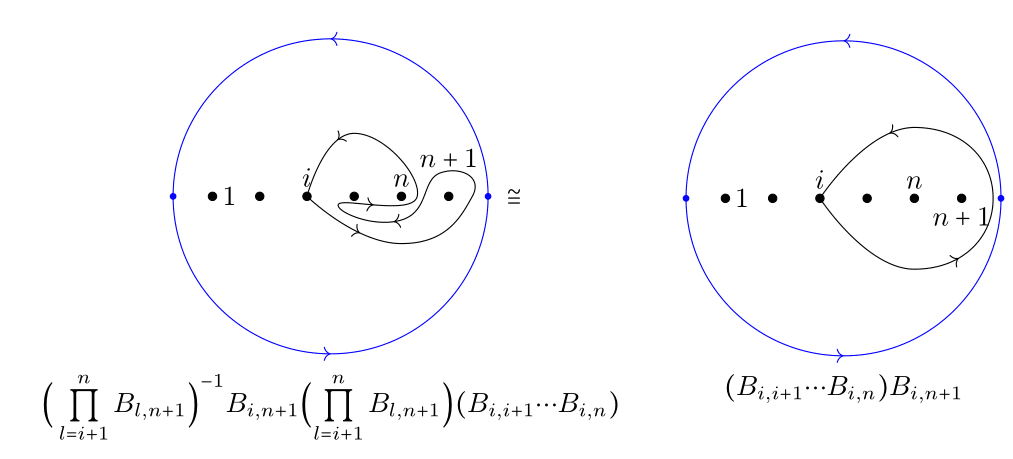}
	\caption{Relation ($iib$).}
	\label{iib}
	\end{figure}
\begin{figure}[H]
	\centering
	\includegraphics[width=1\textwidth]{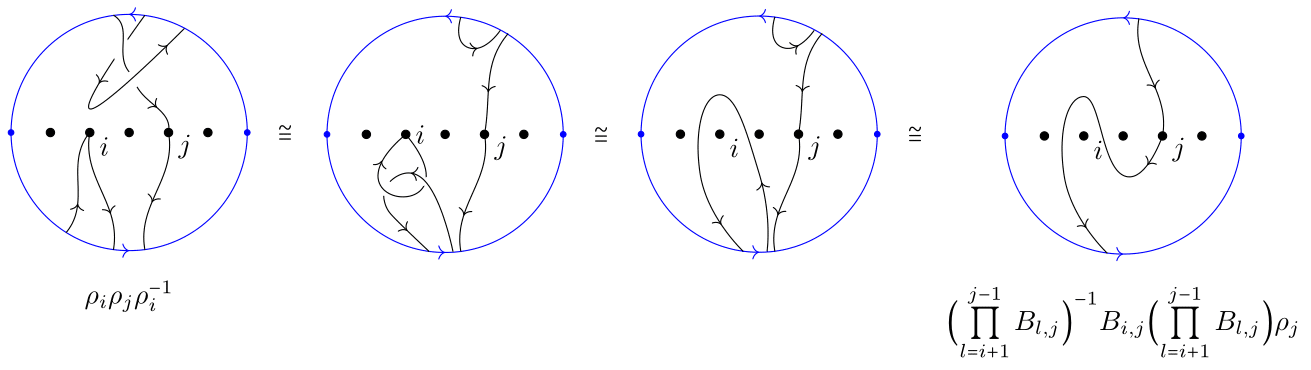}
	\caption{Relation ($iii$).}
	\label{iii}
\end{figure}

\begin{figure}[H]
	\centering
	\includegraphics[width=.7\textwidth]{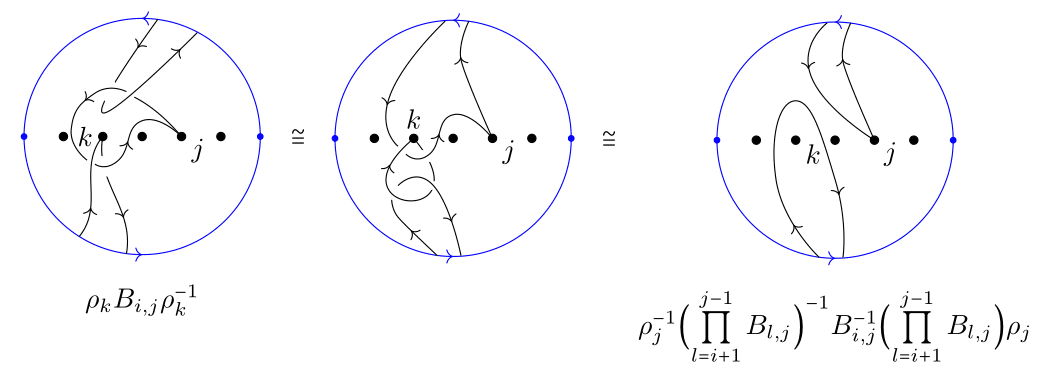}
	\caption{Relation ($iv$) for $k=i$.}
	\label{ivb}
\end{figure}

\begin{figure}[H]
	\centering
	\includegraphics[width=.7\textwidth]{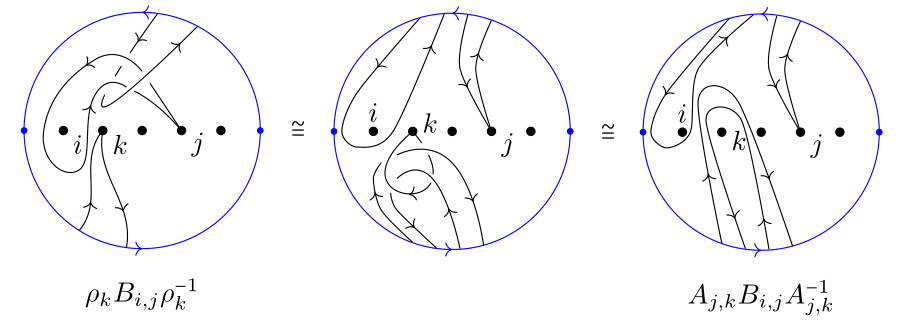}
	\caption{Relation ($iv$) for $i<k<j$.}
	\label{ivc}
\end{figure}

Using Proposition \ref{pre1} we shall obtain a presentation of the group $P_m(\mathbb{R}P^2\setminus \{x_1,\dots, x_n\})$. Note that there is also a presentation of this group in [\cite{guaschiIncl}, Proposition 11].

\begin{proposition}\label{presentationpunctures}
	For $n, m\geq 1$, the following constitutes a presentation of\\ $\Pi_{n,m}=P_m(\mathbb{R}P^2\setminus \{x_1,\dots, x_n\})$.
	\\
	\\\textbf{Generators:} $B_{i, j}, \ \text{for} \ 1\leq i< j $,  $n+1\leq j\leq n+m $ 
	and $ \rho_k, \ \text{for} \ n+1\leq k\leq n+m$.
	\\\textbf{Relations:} \begin{enumerate}[label=(\roman*)]
		\item For $1\leq i <j$ and $1\leq k<l,$ where $n+1\leq j < l \leq n+m,$
		\begin{equation*}
		B_{i,j}B_{k,l}B_{i,j}^{-1}=
		\begin{cases}
		B_{k,l}, & \text{for}\ k<i \ \text{or} \ j<k. \\
		B_{k,l}^{-1}B_{i,l}^{-1}B_{k,l}B_{i,l}B_{k,l}, & \text{for}\ i<k=j<l.\\
		B_{j,l}^{-1}B_{k,l}B_{j,l}, & \text{for} \ k=i<j<l.\\
		B_{j,l}^{-1}B_{i,l}^{-1}B_{j,l}B_{i,l}B_{k,l}B_{i,l}^{-1}B_{j,l}^{-1}B_{i,l}B_{j,l}, & \text{for}\ i<k<j<l.
		\end{cases}
		\end{equation*}

		\item For \ $n+1\leq k \leq n+m$,  \\
		$\rho_k (B_{1,k}\cdots B_{k-1,k}) = (B_{k,k+1}\cdots B_{k,n+m})\rho_k^{-1}$, called surface relations.
		\item For \ $n+1\leq k < l \leq n+m$,  \\
		$ \rho_k\rho_l\rho_k^{-1} = \Big({\displaystyle \prod_{i=k+1}^{l-1} B_{i,l}}\Big)^{-1}B_{k,l}\Big({\displaystyle \prod_{i=k+1}^{l-1} B_{i,l}}\Big)\rho_l.$
		\item For $1\leq i<j$,\  $ n+1<j\leq n+m$,\ $n+1\leq k \leq n+m$ and $k\neq j$,
		\begin{equation*}
		\rho_kB_{i, j}\rho_k^{-1}=
		\begin{cases}
		B_{i,j}, & \text{for}\ j<k \ \text{or} \ k<i .\\
		\rho_j^{-1} \Big({\displaystyle \prod_{l=k+1}^{j-1} B_{l,j}}\Big)^{-1}B_{i,j}^{-1}\Big({\displaystyle \prod_{l=k+1}^{j-1} B_{l,j}}\Big)\rho_j   & \text{for}\ k=i .\\
		A_{j,k}B_{i,j}A_{j,k}^{-1}      & \text{for}\ i < k<j,
		\end{cases}
		\end{equation*}		
		where $A_{j,k}:=\rho_j^{-1}\Big({\displaystyle \prod_{l=k+1}^{j-1} B_{l,j}}\Big)^{-1}B_{k,j}^{-1}\Big({\displaystyle \prod_{l=k+1}^{j-1} B_{l,j}}\Big)\rho_j$.
	\end{enumerate}

\end{proposition}

\begin{proof}
	We proceed using the same techniques as in the proof of Proposition \ref{pre1}; applying induction on $m$, where $m\in\mathbb{N}$ and using standard results concerning the presentation of group extensions [\cite{johnsonpresentation}, p. 139], on the short exact sequence
		\begin{equation*}
	\begin{tikzcd}
	1 \arrow[r]& \Pi_{n+m,1} \arrow[r] & \Pi_{n,m+1} \arrow[r, "\bar{p}_{m+1,m}"] & \Pi_{n,m}\ar[r] & 1,
	\end{tikzcd}
	\end{equation*}
	where $\bar{p}_{m+1,m}$ can be considered geometrically as the epimorphism that forgets the last strand.
	We leave the details to the reader.
\end{proof}	

In order to give a presentation of $B_{n,m}(\mathbb{R}P^2)$, we first determine a presentation of the group $B_m(\mathbb{R}P^2\setminus \{x_1,\dots,x_n\})$.

\begin{proposition}\label{pres3}
	For $n, m\geq 1$, the following constitutes a presentation of\\
 $\beta_{n,m}=B_m(\mathbb{R}P^2\setminus \{x_1,\dots,x_n\})$.
	\\
	\\\textbf{Generators:} $B_{i, j}, \ \text{for} \ 1\leq i\leq n$, $n+1\leq j\leq n+m, $
	$ \rho_k, \ \text{for} \ n+1\leq k\leq n+m$, and $\sigma_l,$ for $1\leq l\leq m-1.$
	\\\textbf{Relations:} \begin{enumerate}[label=(\roman*)]
		\item For $1\leq i,k \leq n$ and $n+1\leq j < l \leq n+m,$ 
		\begin{equation*}
		B_{i,j}B_{k,l}B_{i,j}^{-1}=
		\begin{cases}
		B_{k,l}, & \text{for}\ k<i. \\
		B_{j,l}^{-1}B_{k,l}B_{j,l}, & \text{for} \ k=i.\\
		B_{j,l}^{-1}B_{i,l}^{-1}B_{j,l}B_{i,l}B_{k,l}B_{i,l}^{-1}B_{j,l}^{-1}B_{i,l}B_{j,l}, & \text{for}\ i<k.
		\end{cases}
		\end{equation*} 
		For $k=j$, 
	$$B_{i,j}B_{k,l}B_{i,j}^{-1}=
		\Lambda_{j,l}B_{i,l}^{-1}
		\Lambda_{j,l}^{-1}B_{i,l}\Lambda_{j,l}^{-1},$$
		where $\Lambda_{j,l}=\sigma_{l-1-n}\cdots\sigma_{j+1-n}\sigma_{j-n}^{-2}\sigma_{j+1-n}^{-1}\cdots\sigma_{l-1-n}^{-1}$.
		
		\item For $n+1\leq k \leq n+m$,  \\
		$\rho_k\Big({\displaystyle \prod_{i=1}^{n} B_{i,k}}\Big)(\sigma_{k-1-n}\cdots\sigma_2\sigma_1^2\sigma_2\cdots\sigma_{k-1-n})= (\sigma_{k-n}\cdots\sigma_{m-2}\sigma_{m-1}^2\sigma_{m-2}\cdots\sigma_{k-n}) \rho_k^{-1}.$
		\item For $n+1\leq k < j \leq n+m$,  \\
		$ \rho_k\rho_j\rho_k^{-1} = (\sigma_{j-1-n}^{-1}\cdots\sigma_{k+1-n}^{-1}\sigma_{k-n}^{2}\sigma_{k+1-n}\cdots \sigma_{j-1-n})\rho_j,$\\
		$\rho_k^{-1}\rho_j\rho_k=\rho_j(\sigma_{j-1-n}\cdots\sigma_{k+1-n}\sigma_{k-n}^{2}\sigma_{k+1-n}^{-1}\cdots \sigma_{j-1-n}^{-1}).$ 
		\item For $1\leq i\leq n$, $ n+1\leq j\leq n+m$, $n+1\leq k \leq n+m$ and $k\neq j$\\
		\begin{equation*}
		\rho_kB_{i, j}\rho_k^{-1}=
		\begin{cases}
		B_{i,j}, & \text{for}\ j<k.\\
		T_{j,k}^{-1} B_{i,j}T_{j,k}, & \text{for}\ n+1\leq k<j,
		\end{cases}
		\end{equation*}
		where $T_{j,k}:=\rho_j^{-1} (\sigma_{j-1-n}^{-1}\cdots\sigma_{k+1-n}^{-1}\sigma_{k-n}^{2}\sigma_{k+1-n}\cdots \sigma_{j-1-n})\rho_j$.
		\item For $1\leq r, s \leq m-1,$\\
		$\sigma_r\sigma_s=\sigma_s\sigma_r,$ for $|r-s|>1$.\\
		$\sigma_r\sigma_{r+1}\sigma_r=\sigma_{r+1}\sigma_r\sigma_{r+1},$ for $1\leq r< m-1.$\\
		\item For $1\leq i \leq n, \ n+1\leq j \leq n+m, \ 1\leq r \leq m-1, $ \\
		\begin{equation*}
		\sigma_rB_{i, j}\sigma_r^{-1}=
		\begin{cases}
		B_{i,j}, & \text{for}\ r\neq j-n-1, j-n.\\
		\sigma_{j-n-1}^2 B_{i,j-1}\sigma_{j-n-1}^{-2}, & \text{for} \ r=j-n-1, \ n+2\leq j\leq n+m.\\
		B_{i,j+1}, & \text{for} \ r=j-n, \ n+1\leq j \leq n+m-1.
		\end{cases}
		\end{equation*}
		\item For $1\leq r \leq m-1,$ $n+1\leq k \leq n+m,$\\
		\begin{equation*}
		\sigma_r\rho_k\sigma_r^{-1}=
		\begin{cases}
		\rho_k, & \text{for}\ r\neq k-n, k-1-n.\\
		\sigma_{k-n}^2\rho_{k+1}, & \text{for}\ r=k-n.\\
		\rho_{k-1}	\sigma_{k-1-n}^{-2}, & \text{for} \ r=k-1-n.
		\end{cases}
		\end{equation*}
		\end{enumerate}
\end{proposition}

\begin{remark}\label{Bgenerators}
	In the above relations, as a consequence of Remark \ref{bgenerators}, for $n+1\leq i< j \leq n+m$, the element $B_{i,j}$, which does not appear
	in the list of generators, should be rewritten as:
	$$B_{i,j}=\sigma_{j-1-n}\cdots\sigma_{i+1-n}\sigma_{i-n}^{2}\sigma_{i+1-n}^{-1}\cdots \sigma_{j-1-n}^{-1},$$ and in particular $B_{n+r,n+r+1}=\sigma_r^2$, for $1\leq r\leq m-1$. Moreover, for $n+1\leq k< j \leq n+m$ then ${\displaystyle \prod_{i=k}^{j-1} B_{i,j}}=\sigma_{j-1-n}\cdots\sigma_{k+1-n}\sigma_{k-n}^2\sigma_{k+1-n}\cdots\sigma_{j-1-n}.$ 
\end{remark}
\begin{proof}
	
	Once again, we proceed by applying standard results concerning the presentation of an extension, [\cite{johnsonpresentation}, p. 139], based on the following short exact sequence:	
	$$1\rightarrow \Pi_{n,m}\rightarrow \beta_{n,m}\rightarrow S_m\rightarrow 1,$$
    where $S_n$ is the symmetric group. The details are left to the reader.
\end{proof}
\begin{remark}
	Relation ($vii$) of Proposition \ref{pres3} is obtained by rewriting the conjugates of $\rho_k$, for $n+1\leq k\leq n+m$, by the coset representatives of the generators of $S_m$ in $\beta_{n,m}$, $\sigma_l$, for $1\leq l\leq m-1$, in terms of the generators of $\Pi_{n,m}$. Figure \ref{fn4} presents both the element $\sigma_{k-n}\rho_k\sigma_{k-n}^{-1}$ and $\sigma_{k-n}^2\rho_{k+1}$, and Figure \ref{fn3} presents both the element $\sigma_{k-1-n}\rho_k\sigma_{k-1-n}^{-1}$ and $\rho_{k-1}\sigma_{k-1-n}^{-2}$.
\end{remark}
\begin{figure}[H]
	\begin{subfigure}{3cm}
		\includegraphics[width=3cm]{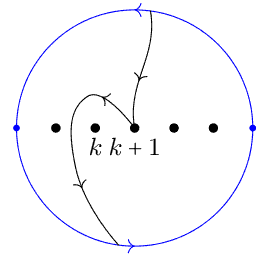}
		\caption{\quad}
		\label{fn4}
	\end{subfigure}
	\qquad
	\begin{subfigure}{3cm}
		\includegraphics[width=3cm]{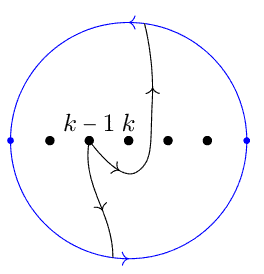}
		\caption{\quad}
		\label{fn3}
	\end{subfigure}
	\caption{From left to right illustration of the elements $\sigma_{k-n}\rho_k\sigma_{k-n}^{-1},\ \sigma_{k-n}^2\rho_{k+1}$ and $\sigma_{k-1-n}\rho_k\sigma_{k-1-n}^{-1},\ \rho_{k-1}\sigma_{k-1-n}^{-2}$.}
	\label{A}
\end{figure}

From Proposition \ref{pres3}  we are able to obtain a presentation of the Abelianisation of  $\beta_{n,m}$, which we will use in Section \ref{S5}
\begin{corollary}\label{presofquotient}
	For $n,m\geq 1$, the following constitutes a presentation of $\beta_{n,m}/\Gamma_2(\beta_{n,m})$.
	\textbf{Generators}: $\rho$, $\sigma$, $\beta_1,\dots ,\beta_n$.
	\\ \textbf{Relations}:\begin{enumerate}[label=(\roman*)]
		\item All generators commute pairwise,
		\item $\sigma^2 = 1$,
		\item $\rho^2\prod_{i=1}^{n}{\beta_i} = 1$,
	\end{enumerate}
	In particular, $$\beta_{n,m}/\Gamma_2(\beta_{n,m})\cong\mathbb{Z}^n \times \mathbb{Z}_2,$$ where $\rho$, $\beta_1,\dots ,\beta_{n-1}$ generate the $\mathbb{Z}^n$-component and $\sigma$ the $\mathbb{Z}_2$-component.
\end{corollary}

\begin{remark}\label{exp}
Note that the generators, $\sigma_1,\dots,\sigma_{m-1}$ of $\beta_{n,m}$, due to relations ($iii$) and ($v$) of Proposition \ref{pres3} are sent to the same element $\sigma$ in $\beta_{n,m}/\Gamma_2(\beta_{n,m})$. In addition, relation ($iii$) of Proposition \ref{pres3} implies that $\sigma^2=1$ in $\beta_{n,m}/\Gamma_2(\beta_{n,m})$. The generators $\rho_k$, for $n+1\leq k\leq n+m$, of $\beta_{n,m}$, due to relation ($vii$) of Proposition \ref{pres3}, are sent to the same element $\rho$ in $\beta_{n,m}/\Gamma_2(\beta_{n,m})$. Moreover, the generators $B_{1, j},\dots,B_{n, j}$, for $n+1\leq j\leq n+m$, of $\beta_{n,m}$, due to relation ($vi$) of Proposition \ref{pres3}, are sent to the elements $\beta_1,\dots ,\beta_n$, respectively, in $\beta_{n,m}/\Gamma_2(\beta_{n,m})$.
We observe that the elements $\rho$, $\sigma$, $\beta_1,\dots ,\beta_{n-1}$ form a generating set for $\beta_{n,m}/\Gamma_2(\beta_{n,m})$, since we can obtain the element $\beta_n$ from relation ($ii$) of Proposition \ref{pres3}.
\end{remark}

Finally, we are ready to give a presentation of the group $B_{n,m}(\mathbb{R}P^2)$.

\begin{theorem}\label{pre4}
	For $n\geq 2$ and  $m\geq 1$, the following constitutes a presentation of $B_{n,m}(\mathbb{R}P^2)$.
	\\
	\\\textbf{Generators:} $B_{i, j}, \ \text{for} \ 1\leq i\leq n $ and $n+1\leq j\leq n+m, $\\ 
	$ \rho_k, \ \text{for} \  n+1\leq k\leq n+m$, \\
	$\sigma_l,$\ for \ $1\leq l\leq m-1,$\\
	$\tau_s,$\ for \ $1\leq s\leq n-1$ \ and \\
	$q_t,$\ for \ $1\leq t\leq n.$
	\\
	\\\textbf{Relations:} \begin{enumerate}[label=(\Roman*)]	
		\item The relations ($i$)-($vii$) of Proposition \ref{pres3}.
		\item For $1\leq i \leq n-1,$ $1\leq j \leq n,$\\
		$\tau_i\tau_j=\tau_i\tau_j,$ for $|i-j|>1$, \\
		$\tau_i\tau_{i+1}\tau_i=\tau_{i+1}\tau_i\tau_{i+1},$ for $1\leq i < n-1,$\\
		$\tau_iq_j=q_j\tau_i,$ for $j \neq i, i+1,$\\
		$q_i=\tau_iq_{i+1}\tau_i,$\\
		$\tau_i^2=q_{i+1}^{-1}q_{i}^{-1}q_{i+1}q_i,$\\
		$q_1^2=(\tau_1\tau_2\cdots\tau_{n-2}\tau_{n-1})(B_{n,n+1}B_{n,n+2}\cdots B_{n,n+m-1}B_{n,n+m})(\tau_{n-1}\tau_{n-2}\cdots\tau_2\tau_1).$
		\item \begin{enumerate}
			\item For $1\leq l\leq m-1$, $1\leq s\leq n-1$ and $1\leq t \leq n$,\\
			$\sigma_l\tau_s=\tau_s\sigma_l,$\\
			$\sigma_lq_t=q_t\sigma_l$.
			\item For $n+1\leq k\leq n+m$, $1\leq s\leq n-1$ and $1\leq t \leq n$,\\
			$\rho_k\tau_s=\tau_s\rho_k,$\\
			$q_t\rho_kq_t^{-1}=E_{t,k}\rho_k,$\\
			where $E_{t,k}:=N_{t,k}^{-1}B_{t,k}N_{t,k}$ and $N_{t,k}:=
			\Big({\displaystyle \prod_{l=t+1}^{n}B_{l,k}}\Big)(\sigma_{k-1-n}\cdots\sigma_1^2\cdots\sigma_{k-1-n})$.
			\item For $1\leq s\leq n-1$ and $1\leq i\leq n < j \leq n+m,$\\
			\begin{equation*}
			\tau_sB_{i,j}\tau_s^{-1}=
			\begin{cases}
			B_{i,j}, & \text{for}\ s\neq i-1,i.\\
			B_{i, j}^{-1}B_{i-1, j}B_{i, j}, & \text{for} \ s=i-1.\\
			B_{i+1, j}, & \text{for} \ s=i.
			\end{cases}
			\end{equation*}
			\item For $1\leq t\leq n$ and $1\leq i\leq n < k \leq n+m,$\\
			\begin{equation*}
			q_tB_{i,k}q_t^{-1}=
			\begin{cases}
			B_{i,k}, & \text{for}\ t<i.\\
			(N_{t,k}\rho_k)^{-1}B_{i, k}^{-1}(N_{t,k}\rho_k), & \text{for} \ t=i.\\
			(\rho_k^{-1}E_{t,k}\rho_k)^{-1}B_{i,k}(\rho_k^{-1}E_{t,k}\rho_k), & \text{for} \ i<t.
			\end{cases}
			\end{equation*}
		\end{enumerate}
	\end{enumerate}
	
\end{theorem}
Note that the generators $q_1,\dots,q_n$\ and \ $\tau_1,\dots, \tau_{n-1}$ of $B_{n,m}(\mathbb{R}P^2)$ correspond to the standard generators, $p_1,\dots,p_n$\ and \ $\sigma_1,\dots, \sigma_{n-1}$ respectively, of $B_n(\mathbb{R}P^2)$, as described in Theorem \ref{vb}.
\begin{proof}
	Once more we will apply the same methods concerning the presentation of an extension, [\cite{johnsonpresentation}, p. 139] based on the following short exact sequence:
	
	$$1\rightarrow \beta_{n,m}\rightarrow B_{n,m}(\mathbb{R}P^2)\xrightarrow{\bar{q}} B_n(\mathbb{R}P^2)\rightarrow 1,$$
	where the map $\bar{q}$ can be considered geometrically as the epimorphism that forgets the last $m$ strands.
	 We keep the same notation, and for simplicity, we denote the corresponding coset representatives in $B_{n,m}(\mathbb{R}P^2)$ by $q_1,\dots,q_n$ and $\tau_1,\dots, \tau_{n-1}$. The union of these elements together with the generators of $\beta_{n,m}$ of Proposition \ref{pres3} gives us the set of generators of $B_{n,m}(\mathbb{R}P^2)$.
	
	As before we will obtain three classes of relations in $B_{n,m}(\mathbb{R}P^2)$. The first class of relations is the set of relations of $\ker(\bar{q})$, meaning of $\beta_{n,m}$, which are relations ($I$) of the statement. 
	
	The second class of relations is obtained by rewriting the relations of $B_n(\mathbb{R}P^2)$, given in Theorem \ref{vb}, in terms of the chosen coset representatives in $B_{n,m}(\mathbb{R}P^2)$ and expressing the corresponding elements as a word in $\beta_{n,m}$. Therefore,
	for $1\leq i \leq n-1$, $1\leq j \leq n$ and $|i-j|>1$, we have
	$\tau_i\tau_j=\tau_i\tau_j,$  for $1\leq i < n-1$, we have
	$\tau_i\tau_{i+1}\tau_i=\tau_{i+1}\tau_i\tau_{i+1},$ and for $j \neq i, i+1$ we have $\tau_iq_j=q_j\tau_i.$  Moreover, we obtain
	$q_i=\tau_iq_{i+1}\tau_i$ and 
	$\tau_i^2=q_{i+1}^{-1}q_{i}^{-1}q_{i+1}q_i.$ The relation $q_1^2(\tau_{1}\cdots\tau_{n-2}\tau_{n-1}^2\tau_{n-2}\cdots\tau_1)^{-1}=1$ of $B_n(\mathbb{R}P^2)$ can be expressed as a word in $B_m(\mathbb{R}P^2 \setminus \{x_1,\dots,x_n\})$ as follows:
	\begin{equation}\label{lastrelation}
	q_1^2(\tau_{1}\cdots\tau_{n-2}\tau_{n-1}^2\tau_{n-2}\cdots\tau_1)^{-1}=\Big(\prod_{j=n+1}^{n+m}(B_{n,j}^{-1}
	\cdots B_{2,j}^{-1})B_{1,j}(B_{2,j}\cdots B_{n,j})\Big).
	\end{equation}
	Relation \eqref{lastrelation} is obtained geometrically, as shown, for $m=2$, in Figure \ref{fn2}.
	
	We will come back to this relation later, as we will use some relations that we will derive in the third class of relations, to simplify it.
	
	The third class of relations is derived from conjugating the generators of $\beta_{n,m}$ by the coset representatives $q_1,\dots,q_n$ and $\tau_1,\dots, \tau_{n-1}$. The relations of this class are obtained geometrically. 
	\begin{itemize}
		\item By conjugating the generators $\sigma_l$, for $1\leq l\leq m-1$ and by using the standard Artin relations in $B_{n,m}(\mathbb{R}P^2)$, we obtain the following relations:\\
		for all $1\leq l \leq m-1$ and $1\leq s \leq n-1$, we have $\sigma_l\tau_s=\tau_s\sigma_l.$ Moreover, we have $\sigma_lq_t=q_t\sigma_l$, for all $1\leq l \leq m-1$ and $1\leq t \leq n$. These relations correspond to relations $(III)(a)$ of the statement.
		\item By conjugating the generators $\rho_k, \ \text{for} \  n+1\leq k\leq n+m$, we have the following:\\
		for $n+1\leq k\leq n+m$, $1\leq s\leq n-1$ and $1\leq t \leq n$,\\
		$\rho_k\tau_s=\tau_s\rho_k$ and\\
		$q_t\rho_kq_t^{-1}=
		\Big(\big({\displaystyle \prod_{l=t+1}^{n}B_{l,k}}\big)(\sigma_{k-1}\cdots\sigma_1^2\cdots\sigma_{k-1})\Big)^{-1}B_{t, k}\Big(\big({\displaystyle \prod_{l=t+1}^{n}B_{l,k}}\big)(\sigma_{k-1}\cdots\sigma_1^2\cdots\sigma_{k-1})\Big)\rho_k,$
		\\ where the term ${\displaystyle \prod_{l=t+1}^{n}B_{l,k}}$ corresponds to the first $n$ points and $(\sigma_{k-1-n}\cdots\sigma_1^2\cdots\sigma_{k-1-n})$ corresponds to the following $k-n$ points.
		These relations correspond to relations $(III)(b)$ of the statement.
		\item By conjugating the generators $B_{i, j}, \ \text{for} \ 1\leq i\leq n $ and $n+1\leq j\leq n+m$ we obtain Relations $(III)(c)$ and $(III)(d)$ of the statement.
	\end{itemize}
	
	We come back to relation \eqref{lastrelation}. From Relation $(III)(c)$ and applying iterated conjugate, we obtain, for any $n+1\leq j \leq n+m$, $$(B_{n,j}^{-1}\cdots B_{2,j}^{-1})B_{1,j}(B_{2,j}\cdots B_{n,j})=(\tau_{1}\tau_2\cdots\tau_{n-2}\tau_{n-1})B_{n,j}(\tau_{1}\tau_2\cdots\tau_{n-2}\tau_{n-1})^{-1},$$which gives \begin{align*}
	&\Big(\prod_{j=n+1}^{n+m}(B_{n,j}^{-1}
	\cdots B_{2,j}^{-1})B_{1,j}(B_{2,j}\cdots B_{n,j})\Big)
	\\&=(\tau_{1}\tau_2\cdots\tau_{n-2}\tau_{n-1})(B_{n,n+1}B_{n,n+2}\cdots B_{n,n+m-1}B_{n,n+m})(\tau_{1}\tau_2\cdots\tau_{n-2}\tau_{n-1})^{-1}.
	\end{align*} 
	Thus, we obtain:
	\begin{align*}
		\Big(\prod_{j=n+1}^{n+m}(B_{n,j}^{-1}
		\cdots B_{2,j}^{-1})B_{1,j}(B_{2,j}\cdots B_{n,j})\Big)(\tau_{1}\tau_2\cdots\tau_{n-2}\tau_{n-1})&
		\\=(\tau_{1}\tau_2\cdots\tau_{n-2}\tau_{n-1})(B_{n,n+1}B_{n,n+2}\cdots B_{n,n+m-1}B_{n,n+m})
	\end{align*}
	Therefore, relation \eqref{lastrelation} becomes:
	\begin{align*}
	q_1^2&=\Big(\prod_{j=n+1}^{n+m}(B_{n,j}^{-1}
	\cdots B_{2,j}^{-1})B_{1,j}(B_{2,j}\cdots B_{n,j})\Big)(\tau_{1}\tau_2\cdots\tau_{n-2}\tau_{n-1}^2\tau_{n-2}\cdots\tau_2\tau_1)\\
	&=(\tau_{1}\tau_2\cdots\tau_{n-2}\tau_{n-1})(B_{n,n+1}B_{n,n+2}\cdots B_{n,n+m-1}B_{n,n+m})(\tau_{n-1}\tau_{n-2}\cdots\tau_2\tau_1),
	\end{align*}
	which is the last relation of relations ($II$) of the statement.
	
	As a result, we have obtained a generating set and a complete set of relations, which coincide with those given in the statement, and this completes the proof.	
\end{proof}
	\begin{figure}[H]
	\centering
	\includegraphics[width=.9\textwidth]{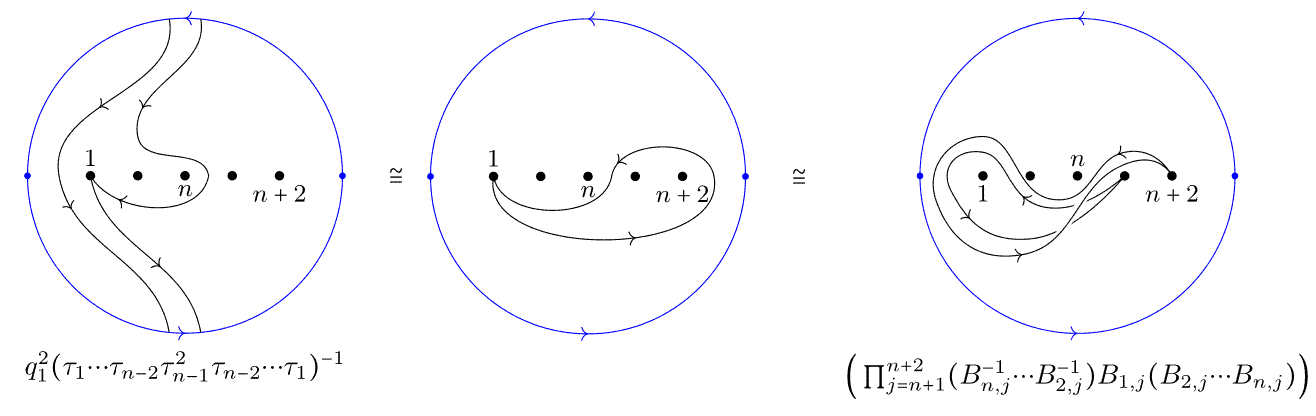}
	\caption{Relation \eqref{lastrelation} for $m=2$.}
	\label{fn2}
\end{figure}
\section{The problem of the existence of a section for the cases $n=1,2$}\label{S3}
In this section we will prove the non-existence of an algebraic section for the homomorphism $B_{1,m}(\mathbb{R}P^2) \xrightarrow{\bar{q}} B_1(\mathbb{R}P^2)$. Moreover, we will prove Theorem \ref{th1}.

\begin{proposition}\label{sectionfor1}
	Let $m\in \mathbb{N}$. The homomorphism $\bar{q}:B_{1,m}(\mathbb{R}P^2) \to B_1(\mathbb{R}P^2)$ admits no section.
\end{proposition}

\begin{proof}
	Let $m\in \mathbb{N}$. We have $B_1(\mathbb{R}P^2)=\langle\rho_1\ | \ \rho_1^2=1\rangle\cong\mathbb{Z}_2$. Under a possible section $s:B_1(\mathbb{R}P^2)\rightarrow B_{1,m}(\mathbb{R}P^2)$, the only non-trivial element of $B_1(\mathbb{R}P^2)$, $\rho_1$, is of order two and thus it has to be mapped to the full twist $\Delta_{1+m}^2$ in $B_{1,m}(\mathbb{R}P^2)$, since $\Delta_{1+m}^2$ is the unique element of order two in $B_{1,m}(\mathbb{R}P^2)\subset B_{1+m}(\mathbb{R}P^2)$, see [\cite{gonccalves2004braid}, Proposition 23]. Note that the full twist $\Delta_{1+m}^2$ is a pure braid and therefore it belongs indeed to $B_{1,m}(\mathbb{R}P^2)$. The full twist $\Delta_{1+m}^2=(\sigma_1\cdots\sigma_{m})^{1+m}$ is mapped by $\bar{q}$ to the trivial element of
	$B_1(\mathbb{R}P^2)$. Thus, $(\bar{q}\circ s)(\rho_1)=\bar{q}(\Delta_{1+m}^2)=1\neq\rho_1$ and as a result, such a section $s:B_1(\mathbb{R}P^2)\rightarrow B_{1,m}(\mathbb{R}P^2)$ cannot exist.
\end{proof}

Note that Proposition \ref{sectionfor1} implies that there is no geometric section either.

\begin{proof}[Proof of Theorem \ref{th1}]
	We will provide both a geometric and an algebraic section.
	
$(i)$ Let $x,y$ be two distinct unordered points in $\mathbb{R}P^2$. We consider a two-sheeted covering map from $\mathbb{S}^2$ to $\mathbb{R}P^2$. These two points, $x,y$, lift to two pairs of distinct antipodal points $\tilde{x},-\tilde{x}$ and $\tilde{y},-\tilde{y}$ respectively on the sphere $\mathbb{S}^2$. 

 We consider the unique plane $\Pi$ that passes through $\tilde{x},-\tilde{x}, \tilde{y},-\tilde{y}$. The plane $\Pi$ intersects the sphere in a great circle $C$, which clearly contains the points $\tilde{x},-\tilde{x}, \tilde{y},-\tilde{y}$. We define $\pm A$ to be the intersection of $\mathbb{S}^2$ with the normal to the plane $\Pi$ passing through the centre of the circle $C$. Consider the geodesics $A\tilde{x}, A(-\tilde{x}), A\tilde{y}, A(-\tilde{y})$, and $-A\tilde{x}, -A(-\tilde{x}), -A\tilde{y}, -A(-\tilde{y})$. On each of these geodesics, take $n$ equally-spaced distinct new points on $\mathbb{S}^2$. We observe that these $8n$ new distinct points on $\mathbb{S}^2$ project to $4n$ new distinct points on $\mathbb{R}P^2$, since the $4n$ equally spaced distinct new points on $A\tilde{x}, A(-\tilde{x}), A\tilde{y}, A(-\tilde{y})$ are antipodal to the $4n$ equally spaced distinct new points on $-A(-\tilde{x}), -A\tilde{x}, -A(-\tilde{y}), -A\tilde{y}$. Thus based on this construction, from two distinct unordered points $x,y$ on $\mathbb{R}P^2$ we obtain $4n$ distinct unordered points on $\mathbb{R}P^2$, different from $x$ and $y$. We will denote these $4n$ distinct unordered new points on $\mathbb{R}P^2$ by $N$. This is an explicit cross-section $s:UF_2(\mathbb{R}P^2) \to UF_{2,4n}(\mathbb{R}P^2)$, for $n\in \mathbb{N}$. 
	
	Moreover, if we project the points $A, -A$ from $\mathbb{S}^2$ into $\mathbb{R}P^2$ we obtain one new point on $\mathbb{R}P^2$, which we denote by $\bar{A}$, different from $x, y$ and from the points in $N$. This yields an explicit cross-section $s:UF_2(\mathbb{R}P^2) \to UF_{2,4n+1}(\mathbb{R}P^2)$, for $n\in \mathbb{N}$. 
	
	We recall that the points $\tilde{x},-\tilde{x}, \tilde{y}, -\tilde{y}$ lie on the great circle $C$ on $\mathbb{S}^2$. These four points give rise to four geodesics; $\tilde{x}\tilde{y}$, $\tilde{y}(-\tilde{x}), -\tilde{x}(-\tilde{y}),-\tilde{y}\tilde{x}$  on $C$. Taking the midpoints of each of these four geodesics, we observe that the midpoint $M_1$ (resp. $M_2$) of the geodesic $\tilde{x}\tilde{y}$ (resp. $\tilde{y}-(\tilde{x})$) is the antipodal point of the midpoint of the geodesic $-\tilde{x}(-\tilde{y})$ (resp. $-(\tilde{y})\tilde{x}$), as we see in Figure \ref{sphere}. Projecting $M_1, -M_1, M_2, -M_2$ from $\mathbb{S}^2$ into $\mathbb{R}P^2$ we obtain two new distinct unordered points, which we denote by $M_1$ and $M_2$ for simplicity. Note that $M_1$ and $M_2$ are different from $x,y$ and from any other new point that we have already constructed. Thus the $N$ points that we have already explicitly constructed together with these two new distinct unordered points $M_1$ and $M_2$, provide an explicit cross-section $s:UF_2(\mathbb{R}P^2) \to UF_{2,4n+2}(\mathbb{R}P^2)$, for $n\in \mathbb{N}$.
	\begin{figure}[H]
		\centering
		\includegraphics[width=.3\textwidth]{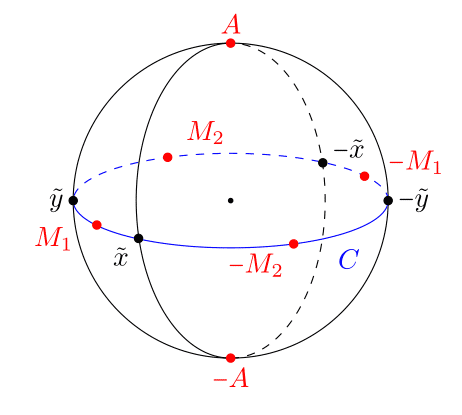}
		\caption{A geometric construction of new distinct unordered points on $\mathbb{S}^2$.}
		\label{sphere}
	\end{figure}
	Finally, if we consider the $N$ points together with the two unordered distinct points, $M_1, M_2$ and the point $\bar{A}$ on $\mathbb{R}P^2$, we obtain $4n+3$ new unordered distinct points different from $x, y$ on the projective plane. This leads to an explicit cross-section $s:UF_2(\mathbb{R}P^2) \to UF_{2,4n+3}(\mathbb{R}P^2)$, for $n\in \mathbb{N}$. 
	
	Summing up, there exists a cross-section, $s:UF_2(\mathbb{R}P^2)\to UF_{2,m}(\mathbb{R}P^2)$, for the fibration $q: UF_{2,m}(\mathbb{R}P^2)\to UF_2(\mathbb{R}P^2)$, for any $m\in \mathbb{N}$. We underline that the new points that we have obtained are independent of any choice of the representatives on $\mathbb{S}^2$ of $x, y$ and that this construction depends continuously on $x$ and $y$. 
	
$(ii)$ Let $m\in \mathbb{N}$. We will define explicitly a section $\bar{s}: B_2(\mathbb{R}P^2)\to B_{2,m}(\mathbb{R}P^2)$ for the homomorphism $\bar{q}:B_{2,m}(\mathbb{R}P^2) \to B_2(\mathbb{R}P^2)$. We will do so by considering two different cases, when $m=2k$ and when $m=2k+1$, where $k\in\mathbb{N}$. From [\cite{van1966braid}, p. 87], we know that $B_2(\mathbb{R}P^2)$ is isomorphic to the dicyclic group of order 16 and thus
	$$B_2(\mathbb{R}P^2)=\langle a_2,\
	\Delta_2\ |\ a_2^8=1,\ \Delta_2^2=a_2^4,\ 
	\Delta_2a_2\Delta_2^{-1}=a_2^{-1} \rangle\cong Dic_{16},$$
	where $a_2=\sigma_1^{-1}\rho_1$ and $\Delta_2=\sigma_1$. Note that this isomorphism follows from [\cite{gonccalves2012surface}, Proposition 15]. Before examining the two different cases, we will prove a general statement which we will use later. Let $$Dic_{8(2+2k)}=\langle x, y\ |\ x^{4(2k+2)}=1,\ y^2=x^{2(2k+2)},\ yxy^{-1}=x^{-1} \rangle,$$ then
	\begin{equation}\label{combine}
	Dic_{16}\cong\langle x^{k+1}, x^ky\ |\ x^{8(k+1)}=1,\ (x^ky)^2=x^{4(k+1)},\ (x^ky)x^{k+1}(x^ky)^{-1}=x^{-(k+1)} \rangle.
	\end{equation}
	Note that from the given presentation of $Dic_{8(2+2k)}$ it follows that $x^{k+1}$ is of order $8$. Moreover, using the relation $yxy^{-1}=x^{-1}$ we have $(x^ky)^2=x^kyx^ky^{-1}y^2=y^2$.
	From the given presentation of $Dic_{8(2+2k)}$ we also have	$x^{4(k+1)}=y^{2}$.
	Thus, $(x^ky)^2=x^{4(k+1)}$.
	Lastly, from the relation $yxy^{-1}=x^{-1}$, we deduce that	$(x^ky)x^{k+1}(x^ky)^{-1}=x^{-(k+1)}$.
	From these relations we conclude that  $\langle x^{k+1}, x^ky\rangle$ is isomorphic to a quotient of the dicyclic group of order 16. But $Dic_{8(2+2k)}=\langle x \rangle \sqcup \langle x \rangle y,$ where $x^{k+1}\in \langle x \rangle$ while $x^{k}y\in \langle x \rangle y$. Thus, $\langle x^{k+1}\rangle\cap\{x^{k}y \}=\emptyset$, which implies that $\langle x^{k+1}, x^{k}y \rangle$ contains at least $16$ distinct elements, and so $\langle x^{k+1},x^{k}y \rangle\cong Dic_{16}$, as claimed.
		
	\begin{enumerate}[leftmargin=*,label=(\alph*)]
		\item First, let $m=2k$, for $k\in \mathbb{N}$. 
		From [\cite{gonccalves2012surface}, Theorem 5], for $n\geq 2$, the group $B_n(\mathbb{R}P^2)$, contains the dicyclic group of order $8n$, $Dic_{8n}=\langle x,y\ | \ x^{4n}=1,\  x^{2n}=y^2,\ yxy^{-1}=x^{-1} \rangle$, as a subgroup. Thus, $Dic_{8(2+2k)}$ is a subgroup of $B_{2+2k}(\mathbb{R}P^2)$. From [\cite{gonccalves2012surface}, Proposition 15] we have $$Dic_{8(2+2k)}=\langle a_{2+2k},\ \Delta_{2+2k}\ |\ a_{2+2k}^{4(2k+2)}=1,\ \Delta_{2+2k}^2=a_{2+2k}^{2(2k+2)},\ 
		\Delta_{2+2k}a_{2+2k}\Delta_{2+2k}^{-1}=a_{2+2k}^{-1} \rangle,$$
		where $a_{2+2k}=\sigma_{1+2k}^{-1}\cdots\sigma_1^{-1}\rho_1$ and $\Delta_{2+2k}=(\sigma_1\sigma_2\cdots\sigma_{1+2k})(\sigma_1\sigma_2\cdots\sigma_{2k})\cdots(\sigma_1\sigma_2)\sigma_1$. 
		We consider the subgroup of $Dic_{8(2+2k)}$ generated by the elements $a_{2+2k}^{k+1}$ and $a_{2+2k}^{k}\Delta_{2+2k}$ and it follows from \eqref{combine} that $\langle a_{2+2k}^{k+1}, a_{2+2k}^{k}\Delta_{2+2k} \rangle\cong Dic_{16}$. 
		Therefore, the homomorphism 
		$$i_1:  B_2(\mathbb{R}P^2)\to B_{2+2k}(\mathbb{R}P^2)$$ defined by $$i_1(a_2)=a_{2+2k}^{k+1}\ \text{and}\ i_1(\Delta_2)=a_{2+2k}^{k}\Delta_{2+2k}$$
		is an injective homomorphism of $B_2(\mathbb{R}P^2)$ in $B_{2+2k}(\mathbb{R}P^2)$. Note that the permutation type of the elements $a_{2+2k}^{k+1}$ and $a_{2+2k}^{k}\Delta_{2+2k}$ has no transposition that corresponds to the first and second strand. 
		In order to obtain an embedding of $B_2(\mathbb{R}P^2)$ in $B_{2,2k}(\mathbb{R}P^2)$, we need to conjugate $a_{2+2k}^{k+1}$ and $a_{2+2k}^{k}\Delta_{2+2k}$ by a suitable element, which is $c:=(\sigma_2\sigma_3\cdots\sigma_{k+1})$. Thus, we define the injective homomorphism 	$$\bar{i}_1:  B_2(\mathbb{R}P^2)\to B_{2,2k}(\mathbb{R}P^2)$$ defined by $$\bar{i}_1(a_2)=c\cdot a_{2+2k}^{k+1}\cdot c^{-1}\ \text{and}\ \bar{i}_1(\Delta_2)=c\cdot a_{2+2k}^{k}\Delta_{2+2k}\cdot c^{-1}.$$
		One can easily verify geometrically that $(\bar{q}\circ\bar{i}_1)(a_2)=a_2\in B_2(\mathbb{R}P^2)$ and $(\bar{q}\circ\bar{i}_1)(\Delta_2)=\Delta_2\in B_2(\mathbb{R}P^2)$, where $\bar{q}$ forgets the last $2k$ strands, see Figures \ref{i1a} and \ref{i1d}. Summing up, the homomorphism $\bar{i}_1$ is a section for the homomorphism $\bar{q}: B_{2,2k}(\mathbb{R}P^2)\to B_2(\mathbb{R}P^2)$. 
		
		\begin{figure}[H]
			\centering
			\includegraphics[width=.8\textwidth]{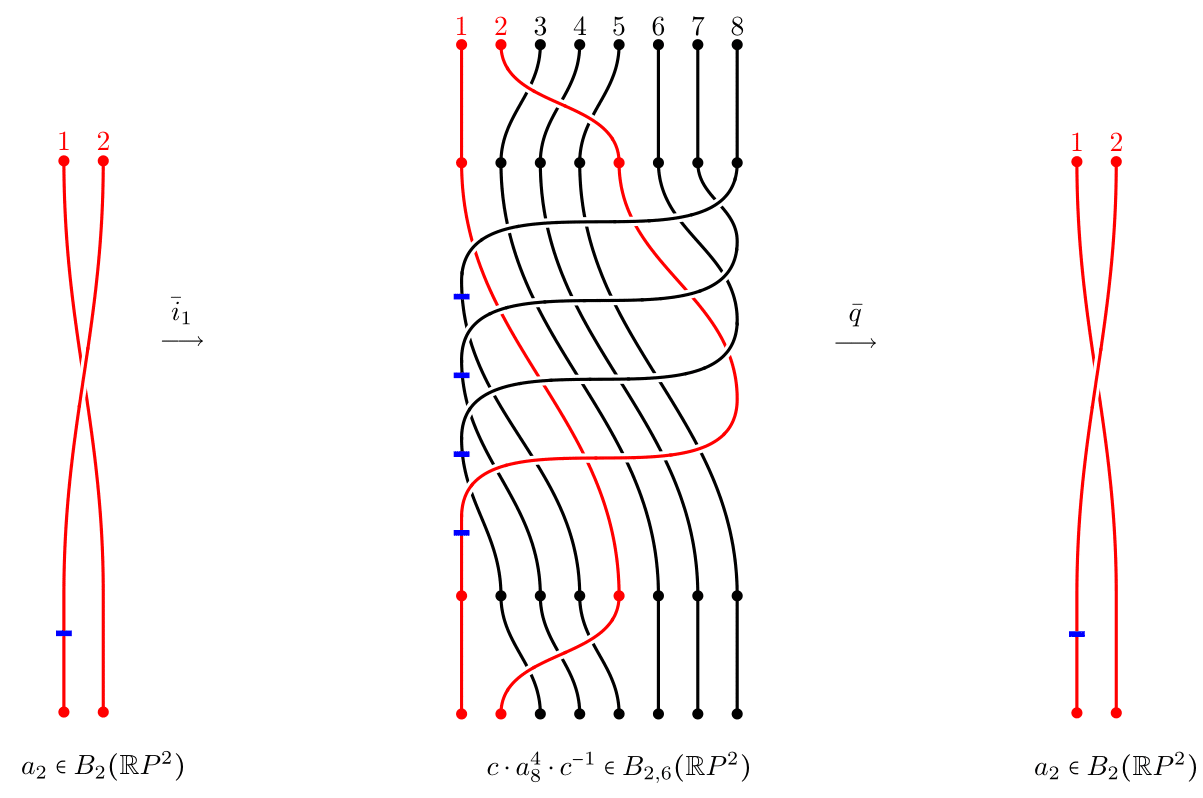}
			\caption{The image of $a_2=\sigma^{-1}_1\rho_1$ under the map $\bar{i}_1$, for $k=3$. Note that with the blue dash we illustrate the element $\rho_1$.}
			\label{i1a}
		\end{figure}
		\begin{figure}[H]
			\centering
			\includegraphics[width=.8\textwidth]{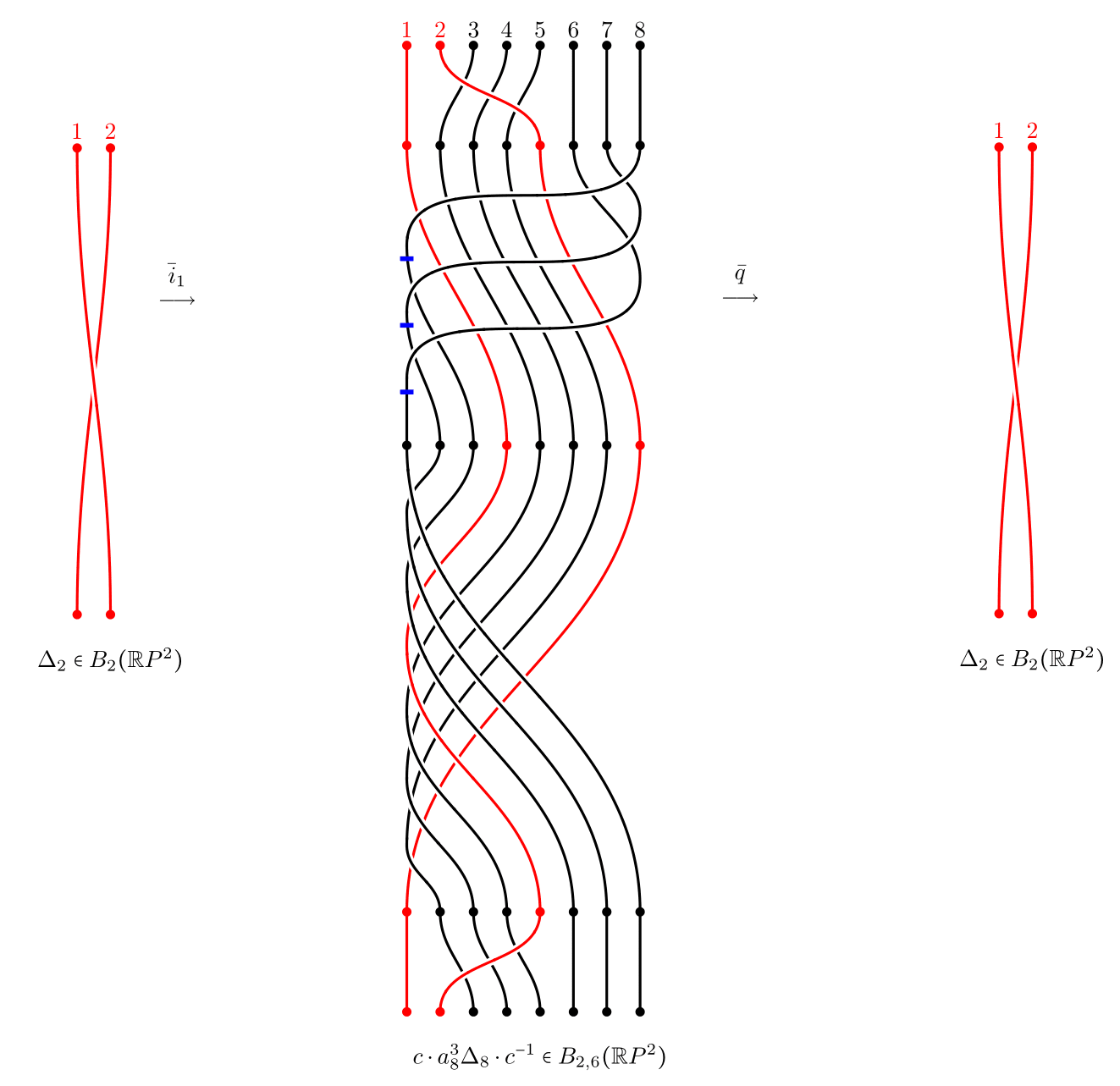}
			\caption{The image of $\Delta_2=\sigma_1$ under the map $\bar{i}_1$, for $k=3$. Note that with the blue dash we illustrate the element $\rho_1$.}
			\label{i1d}
		\end{figure}
		
		\item Now let $m=2k+1$, for $k\in \mathbb{N}$.
		From [\cite{gonccalves2012surface}, Theorem 5], for $n\geq 3$, the group $B_n(\mathbb{R}P^2)$, contains, $Dic_{8(n-1)}=\langle x,y\ | \ x^{4(n-1)}=1,\ x^{2(n-1)}=y^2,\ yxy^{-1}=x^{-1} \rangle$, the dicyclic group of order $8(n-1)$ as a subgroup. 
		So, $Dic_{8(2+2k)}$ is a subgroup of $B_{2+(2k+1)}(\mathbb{R}P^2)$, and from [\cite{gonccalves2012surface}, Proposition 15] we obtain
		\begin{align*}
		Dic_{8(2+2k)}=\langle b, \Delta a^{-1}\ |\ b^{4(2k+2)}=1,\ (\Delta a^{-1})^2=b^{2(2k+2)},
		(\Delta a^{-1})b(\Delta a^{-1})^{-1}=b^{-1} \rangle,
		\end{align*}
		where $b:=b_{2+(2k+1)}=\sigma_{2k+1}^{-1}\cdots\sigma_1^{-1}\rho_1$, $a:=a_{2+(2k+1)}=\sigma_{2k+2}^{-1}\cdots\sigma_1^{-1}\rho_1$ and $\Delta:=\Delta_{2+(2k+1)}=(\sigma_1\sigma_2\cdots\sigma_{2k+2})(\sigma_1\sigma_2\cdots\sigma_{2k+1})\cdots(\sigma_1\sigma_2)\sigma_1$. 
		We consider the subgroup of $Dic_{8(2+2k)}$ generated by the elements $b^{k+1}$ and $b^{k}\Delta a^{-1}$. From \eqref{combine} we get that $\langle b^{k+1}, b^{k}\Delta a^{-1} \rangle\cong Dic_{16}$. Therefore, the homomorphism 
		$$i_2:  B_2(\mathbb{R}P^2)\to B_{2+(2k+1)}(\mathbb{R}P^2)$$ defined by $$i_2(a_2)=b^{k+1}\ \text{and}\ i_2(\Delta_2)=b^{k}\Delta a^{-1}$$
		is an injective homomorphism of $B_2(\mathbb{R}P^2)$ in $B_{2+(2k+1)}(\mathbb{R}P^2)$. In order to obtain an embedding of $B_2(\mathbb{R}P^2)$ in $B_{2, (2k+1)}(\mathbb{R}P^2)$, we need to conjugate $b^{k+1}$ and $b^{k}\Delta a^{-1}$ by a suitable element, which is, as in the case (a), $c:=(\sigma_2\sigma_3\cdots\sigma_{k+1})$. Thus, we define the injective homomorphism 	$$\bar{i}_2:  B_2(\mathbb{R}P^2)\to B_{2,2k+1}(\mathbb{R}P^2)$$ defined by $$\bar{i}_2(a_2)=c\cdot b^{k+1}\cdot c^{-1}\ \text{and}\ \bar{i}_2(\Delta_2)=c\cdot b^{k}\Delta a^{-1}\cdot c^{-1}.$$
		Once again, as in the case (a), one can verify geometrically that $(\bar{q}\circ\bar{i}_2)(a_2)=a_2\in B_2(\mathbb{R}P^2)$ and $(\bar{q}\circ\bar{i}_2)(\Delta_2)=\Delta_2\in B_2(\mathbb{R}P^2)$, where $\bar{q}$ forgets the last $2k+1$ strands. Summing up, the homomorphism $\bar{i}_2$ is a section for the homomorphism $\bar{q}: B_{2,2k+1}(\mathbb{R}P^2)\to B_2(\mathbb{R}P^2)$.
	\end{enumerate}
	
	We conclude that, for every $m\in \mathbb{N}$, we have an explicit algebraic section $$\bar{s}: B_2(\mathbb{R}P^2)\to B_{2,m}(\mathbb{R}P^2)$$ for the homomorphism $\bar{q}:B_{2,m}(\mathbb{R}P^2) \to B_2(\mathbb{R}P^2)$ defined as follows:
	$$\bar{s}= \begin{cases} \bar{i}_1, & \text{for}\ m=2k, \\ \bar{i}_2, & \text{for}\ m=2k+1.\end{cases}$$
\end{proof}

\section{Lower central and derived series of the group $\beta_{n,m}$}\label{S4}
In this section, we determine, for $n\geq 2$ and for most values of $m$, the lower central and derived series of $\beta_{n,m}=B_m(\mathbb{R}P^2\setminus\{x_1,\dots, x_n\})$.

Let $i\in \mathbb{N}$. We recall that the lower central series of a group $G$ is defined as the descending normal series $\Gamma_1(G)=G\supseteq\Gamma_2(G)\supseteq\dots \supseteq\Gamma_n(G)\supseteq\dots$, where, for $i>1$, $\Gamma_i(G)=[\Gamma_{i-1}(G), G]$ is the subgroup of $G$ generated by the commutators $[x,y]=xyx^{-1}y^{-1}$, for $x\in \Gamma_{i-1}(G)$ and $y\in G$. Moreover, the derived series of a group $G$ is defined as the descending normal series $G^{(0)}=G\supseteq G^{(1)}\supseteq\dots \supseteq G^{(n)}\supseteq\dots$, where, for $i\geq1$, $G^{(i)}=[G^{(i-1)}, G^{(i-1)}]$ is the subgroup of $G$ generated by the commutators $[x,y]$, for $x,y\in G^{(i-1)}$. Note that $\Gamma_2(G)=G^{(1)}$.

For $m=1$ and $n\geq 1$, $\beta_{n,1}=\pi_1(\mathbb{R}P^2\setminus\{x_1,\dots,x_n\})=F_{n}$, where $F_n$ denotes the free group on $n$ generators. So the lower central and derived series of $\beta_{n,1}$ are those of free groups of finite rank.

\begin{proof}[Proof of Theorem \ref{th_l}]
	Let $n\geq 1$ and let $\Gamma_2:=\Gamma_2(\beta_{n,m})$. Suppose that $K\subset\Gamma_2$ is a normal subgroup of $\beta_{n,m}$. We consider the following commutative diagram:
	\begin{equation}\label{com}
	\begin{tikzcd}
	1 \ar[r] & \Gamma_2 \ar[d, "pr|_{\Gamma_2}"] \ar[r] & \beta_{n,m} \ar[d,"pr"] \ar[r, "ab"] & \beta_{n,m}/\Gamma_2 \ar[-,double line with arrow={-,-}]{d} \ar[r] & 1 \\
	1 \ar[r] & \Gamma_2/K \ar[r] &\beta_{n,m}/K \ar[r, "p"] & \beta_{n,m}/\Gamma_2 \ar[r] & 1,
	\end{tikzcd}
	\end{equation}
	where the map $ab$ is Abelianisation of $\beta_{n,m}$, the map $pr$ is the canonical projection from 
	$\beta_{n,m}$ to $\beta_{n,m}/K$ and the map $pr|_{\Gamma_2}$ is the restriction of $pr$ to $\Gamma_2$. The map $ab$ that factors through $\beta_{n,m}/K$, and thus the map $p$ satisfies $ab=p\circ pr$. 
	
	First we consider the case where $K=\Gamma_3$, for $\Gamma_3:=\Gamma_3(\beta_{n,m})$, and we will prove that, for $m\geq 3$, $\Gamma_2\big(\beta_{n,m}\big)=\Gamma_3\big(\beta_{n,m}\big)$. Let $m\geq 3$. Based on the commutative diagram \eqref{com}
	 we will obtain relations in $\beta_{n,m}/\Gamma_3$ by projecting the relations of $\beta_{n,m}$ given in Proposition \ref{pres3}. From Corollary \ref{presofquotient}, we know that $\beta_{n,m}/\Gamma_2$ is isomorphic to $\mathbb{Z}^n\times\mathbb{Z}_2$. In particular, from Remark \ref{exp}, we know that under the map $ab$, the generators  $\sigma_1,\dots,\sigma_{m-1}$ are sent to the same element $\sigma\in \beta_{n,m}/\Gamma_2$. Thus, without loss of generality let $\sigma=\bar{\sigma}_1\in \beta_{n,m}/\Gamma_2$, where $\bar{\sigma}_1=\sigma_1\Gamma_2$. It follows that $$pr(\sigma_i)=t_i\cdot{\sigma}_1\Gamma_3\in \beta_{n,m}/\Gamma_3,$$	
	for $1\leq i\leq m-1$, where $t_i\in \ker(p)=\Gamma_2/\Gamma_3$ and $t_1$ may be taken to be 1.
	
	For simplicity, we will denote ${\sigma}_1\Gamma_3$ by $\tilde{\sigma}_1$. Projecting now relation $(v)$ of Proposition \ref{pres3}, $\sigma_r\sigma_{r+1}\sigma_r=\sigma_{r+1}\sigma_r\sigma_{r+1}$, into $\beta_{n,m}/\Gamma_3$ we obtain $t_r\tilde{\sigma}_1t_{r+1}\tilde{\sigma}_1t_r\tilde{\sigma}_1=t_{r+1}\tilde{\sigma}_1t_r\tilde{\sigma}_1t_{r+1}\tilde{\sigma}_1$, and since $t_i\in\Gamma_2/\Gamma_3$, where $\Gamma_2/\Gamma_3$ is a central subgroup of $\beta_{n,m}/\Gamma_3$, it follows that $t_r=t_{r+1},$ for $1\leq r\leq m-1$, meaning that $1=t_1=\dots=t_{m-1}$. We conclude that $pr(\sigma_1)=\dots=pr(\sigma_{m-1})=\tilde{\sigma}_1.$
	
	Similarly, from Remark \ref{exp}, we know that under the map $ab$ the generators $\rho_{n+1},\dots,\rho_{n+m}$ projects to the same element $\rho$ in the Abelianisation $\beta_{n,m}/\Gamma_2$. Thus, without loss of generality let $\rho=\bar{\rho}_{n+1}$, where $\bar{\rho}_{n+1}=\rho_{n+1}\Gamma_2$. It follows that $$pr(\rho_i)=s_i\cdot\rho_{n+1}\Gamma_3\in \beta_{n,m}/\Gamma_3,$$	
	for $n+1\leq i\leq n+m$, where $s_i\in \ker(p)=\Gamma_2/\Gamma_3$ and $s_{n+1}$ may be taken to be $1$. 
	
	We will denote $\rho_{n+1}\Gamma_3$ by $\tilde{\rho}_{n+1}$. Projecting relation $(iii)$ of Proposition \ref{pres3}, $$ \rho_k\rho_j\rho_k^{-1} = (\sigma_{j-1-n}^{-1}\dots\sigma_{k+1-n}^{-1}\sigma_{k-n}^{2}\sigma_{k+1-n}\dots \sigma_{j-1-n})\rho_j,$$ for $n+1\leq k < j \leq n+m$, under $pr$, gives $s_k\tilde{\rho}_{n+1}s_j\tilde{\rho}_{n+1}\tilde{\rho}_{n+1}^{-1}s_k^{-1}=\tilde{\sigma}_1^2s_j\tilde{\rho}_{n+1}$, and since 
	$s_i\in\Gamma_2/\Gamma_3$, for $n+1\leq i\leq n+m$, where $\Gamma_2/\Gamma_3$ is central in $\beta_{n,m}/\Gamma_3$, it follows that
	$\tilde{\sigma}_1^2=1$.
	
	Since $m\geq 3$, the relation $\sigma_1\rho_{n+3}\sigma_1^{-1}=\rho_{n+3}$ ($(vii)$ of Proposition \ref{pres3}) exists in $\beta_{n,m}$, and projecting it into $\beta_{n,m}/\Gamma_3$, it follows that $\tilde{\sigma}_1s_{n+3}\tilde{\rho}_{n+1}\tilde{\sigma}_1^{-1}=s_{n+3}\tilde{\rho}_{n+1}$. Since $s_{n+3}\in\Gamma_2/\Gamma_3$ we obtain $$\tilde{\sigma}_1\tilde{\rho}_{n+1}=\tilde{\rho}_{n+1}\tilde{\sigma}_1.$$
	Furthermore, the projection of the relation $\sigma_{k-n}\rho_k\sigma_{k-n}^{-1}=\sigma_{k-n}^2\rho_{k+1}$ ($(vii)$ of Proposition \ref{pres3}), for $n+1\leq k \leq n+m-1$, yields $\tilde{\sigma}_1s_k\tilde{\rho}_{n+1}\tilde{\sigma}_1^{-1}=1\cdot s_{k+1}\tilde{\rho}_{n+1}$. We know that $\tilde{\sigma}_1\tilde{\rho}_{n+1}=\tilde{\rho}_{n+1}\tilde{\sigma}_1, s_{n+3}\in\Gamma_2/\Gamma_3$ and that $s_k\tilde{\rho}_{n+1}=s_{k+1}\tilde{\rho}_{n+1}$. Thus, $1=s_{n+1}=\dots =s_{n+m-1}=s_{n+m}$, and it follows that $$pr(\rho_{n+1})=\dots=pr(\rho_{n+m})=\tilde{\rho}_{n+1}.$$ 
	
	From Remark \ref{exp}, we have that, under the map $ab$, the generators $B_{1, j},\dots,B_{n, j}$ are sent to the elements $\beta_1,\dots ,\beta_n$, respectively, in $\beta_{n,m}/\Gamma_2(\beta_{n,m})$. Thus, without loss of generality it follows that $ab(B_{i,j})=\bar{B}_{i,n+1}\in \beta_{n,m}/\Gamma_2$, $1\leq i\leq n$ and $n+1\leq j\leq n+m$, where $\bar{B}_{i,n+1}={B}_{i,n+1}\Gamma_2$. Note that in Corollary \ref{presofquotient} the element $ab(B_{i,j})=\bar{B}_{i,n+1}$ is denoted by $\beta_i$. Thus, $pr(B_{i,j})=a_{i,j}\cdot B_{i,n+1}\Gamma_3\in \beta_{n,m}$, for $1\leq i\leq n, \ n+1\leq j\leq n+m$, where $a_{i,j}\in \ker(p)=\Gamma_2/\Gamma_3$ and $a_{i,n+1}$ may be taken to be $1$.
	
	We will denote $B_{i,n+1}\Gamma_3$ by $\tilde{B}_{i,n+1}$. Since $m\geq 3$, the relation $\sigma_1B_{i,n+3}\sigma_1^{-1}=B_{i,n+3}$ ($(vi)$ of Proposition \ref{pres3}), for $1\leq i\leq n$, exists in $\beta_{n,m}$, and projecting it into $\beta_{n,m}/\Gamma_3$, it follows that $\tilde{\sigma}_1a_{i,n+3}\tilde{B}_{i,n+1}\tilde{\sigma}_1^{-1}=a_{i,n+3}\tilde{B}_{i,n+1}$. Since $a_{i,n+3}\in\Gamma_2/\Gamma_3$ we obtain $$\tilde{\sigma}_1\tilde{B}_{i,n+1}=\tilde{B}_{i,n+1}\tilde{\sigma}_1,\ \text{for}\ 1\leq i\leq n.$$ Moreover, $pr(\sigma_{j-n}B_{i,j}\sigma_{j-n}^{-1})=pr(B_{i,j+1})$, for $1\leq i\leq n, \ n+1\leq j\leq n+m-1$, yields $\tilde{\sigma}_1a_{i,j}\tilde{B}_{i,n+1}\tilde{\sigma}_1^{-1}=a_{i,j+1}\tilde{B}_{i,n+1}$. We know that
	$\tilde{\sigma}_1\tilde{B}_{i,n+1}=\tilde{B}_{i,n+1}\tilde{\sigma}_1$ and that $a_{i,j}\in\Gamma_2/\Gamma_3$, and thus
	$1=a_{i,n+1}=\dots =a_{i,n+m-1}=a_{i,n+m}$, and $$pr(B_{i,j})=\tilde{B}_{i,n+1},\ \text{for} \ 1\leq i\leq n, \ n+1\leq j\leq n+m.$$
	
	We conclude that $\{\tilde{\sigma}_1, \tilde{\rho}_{n+1}, \tilde{B}_{1,n+1},\dots, \tilde{B}_{n,n+1}\}$ is a generating set for $\beta_{n,m}/\Gamma_3$, subject to the relations $\tilde{\sigma}_1^2=1,\ \tilde{\sigma}_1\tilde{\rho}_{n+1}=\tilde{\rho}_{n+1}\tilde{\sigma}_1\ \text{and}\ \tilde{\sigma}_1\tilde{B}_{i,n+1}=\tilde{B}_{1,n+1}\tilde{\sigma}_1$, for $1\leq i\leq n$.
	
	In order to determine further relations among these generators, we will now project some other relations of $\beta_{n,m}$ into $\beta_{n,m}/\Gamma_3$. Projecting relation ($iv$) of Proposition \ref{pres3}, we obtain $\tilde{\rho}_{n+1}\tilde{B}_{i,n+1}\tilde{\rho}_{n+1}^{-1}=\tilde{B}_{i,n+1}$, for $1\leq i\leq n$, which implies that, $\tilde{\rho}_{n+1}\tilde{B}_{i,n+1}=\tilde{B}_{i,n+1}\tilde{\rho}_{n+1}$, for $1\leq i\leq n$.
	Similarly, projecting relation ($i$) of Proposition \ref{pres3}, we obtain  $\tilde{B}_{i,n+1}\tilde{B}_{k,n+1}=\tilde{B}_{k,n+1}\tilde{B}_{i,n+1},$ for $1\leq i,k\leq n$.
	Finally, the projection of relation $\rho_k {\displaystyle \prod_{i=1}^{n} B_{i,k}}(\sigma_{k-1-n}\dots\sigma_1^2\dots\sigma_{k-1-n})= (\sigma_{k-n}\dots\sigma_{m-1}^2\dots\sigma_{k-n}) \rho_k^{-1}$ ($(ii)$ of Proposition \ref{pres3}) yields $\tilde{\rho}_{n+1}^2{\displaystyle \prod_{i=1}^{n} \tilde{B}_{i,n+1}}=1.$
	
	Summing up, we have shown that the generators $\{\tilde{\sigma}_1,\tilde{\rho}_{n+1}, \tilde{B}_{1,n+1},\dots, \tilde{B}_{n,n+1}\}$ of $\beta_{n,m}/\Gamma_3$ commute pairwise, so the quotient group $\beta_{n,m}/\Gamma_3$ is Abelian. By the universal property of the Abelianisation, it follows that $\Gamma_2$ is a subgroup of $\Gamma_3$, since any homomorphism of $\beta_{n,m}$ onto an Abelian group factors through $\beta_{n,m}/\Gamma_2$, which implies that $\Gamma_2\subseteq\Gamma_3$. But by the definition of the lower central series, we know that $\Gamma_3\subseteq\Gamma_2$, and therefore we conclude that $\Gamma_3=\Gamma_2$. As a result, $\Gamma_2(\beta_{n,m})=\Gamma_3(\beta_{n,m})$, for $m\geq3$.  
	
	We now prove the second statement, that is, for $m\geq 5$, $\beta_{n,m}^{(1)}=\beta_{n,m}^{(2)}$. Let $m\geq 5$ and let $K=\beta_{n,m}^{(2)}$. Based on the commutative diagram \eqref{com} and using the same argument as before, we have: $$pr(\sigma_i)=t_i\cdot{\sigma}_1\beta_{n,m}^{(2)}\in \beta_{n,m}/\beta_{n,m}^{(2)},$$	
	for $1\leq i\leq m-1$, where $t_i\in \ker(p)=\beta_{n,m}^{(1)}/\beta_{n,m}^{(2)}$ and $t_1$ may be taken to be $1$.
	Let ${\sigma}_1\beta_{n,m}^{(2)}$ be denoted by $\bar{\sigma}_1$. For $m\geq 5$, the projection by $pr$ of the relation $\sigma_r\sigma_s=\sigma_s\sigma_r$ ($(v)$ of Proposition \ref{pres3}), for $|r-s|>1$ implies that $\bar{\sigma}_1t_k\bar{\sigma}_1=t_k\bar{\sigma}_1\bar{\sigma}_1$, for $3\leq k\leq m-1$, and so $\bar{\sigma}_1$ commutes with $t_k$ for $3\leq k\leq m-1$. Moreover, for $4\leq l\leq m-1$, the projection of the same relation yields $t_l\bar{\sigma}_1t_2\bar{\sigma}_1=t_2\bar{\sigma}_1t_l\bar{\sigma}_1$. Since $\bar{\sigma}_1$ commutes with $t_l$ and $\beta_{n,m}^{(1)}/\beta_{n,m}^{(2)}$ is an Abelian group, it follows that $\bar{\sigma}_1t_2=t_2\bar{\sigma}_1$. Hence $t_k$ commutes with $\bar{\sigma}_1$, for $2\leq k \leq m-1$ and $m\geq 5$. Now, projecting the relation $\sigma_r\sigma_{r+1}\sigma_r=\sigma_{r+1}\sigma_r\sigma_{r+1}$ ($(v)$ of Proposition \ref{pres3}) into $\beta_{n,m}/\beta_{n,m}^{(2)}$ it follows that 
	$t_r\bar{\sigma}_1t_{r+1}\bar{\sigma}_1t_r\bar{\sigma}_1=t_{r+1}\bar{\sigma}_1t_r\bar{\sigma}_1t_{r+1}\bar{\sigma}_1$, and so for $2\leq r < m-1$ we have $t_2=\dots=t_{m-1}$, since $t_k$ commutes with $\bar{\sigma}_1$, for $2\leq k \leq m-1$. If $r=1$, the projection of the same relation yields $\bar{\sigma}_1t_2\bar{\sigma}_1\bar{\sigma}_1=t_2\bar{\sigma}_1\bar{\sigma}_1t_2\bar{\sigma}_1$, thus $t_2=1$ and $t_1=\dots=t_{m-1}=1$. We conclude that $$pr(\sigma_1)=\dots=pr(\sigma_{m-1})=\bar{\sigma}_1.$$
	
	As in the previous case we have $$pr(\rho_i)=s_i\cdot\rho_{n+1}\beta_{n,m}^{(2)}\in \beta_{n,m}/\beta_{n,m}^{(2)},$$	
	for $n+1\leq i\leq n+m$, where $s_i\in \ker(p)=\beta_{n,m}^{(1)}/\beta_{n,m}^{(2)}$ and $s_{n+1}$ may be taken to be $1$.
	
	We will denote $\rho_{n+1}\beta_{n,m}^{(2)}$ by $\bar{\rho}_{n+1}$. The proof that $\bar{\sigma}_1^2=1$ differs here, since we can no longer use the argument of central elements. Projecting the relation   $\rho_{n+1}\rho_{n+2}\rho_{n+1}^{-1} = \sigma_1^{2}\rho_{n+2}$, which is relation $(iii)$ of Proposition \ref{pres3} for $k=n+1$ and $j=n+2$, into $\beta_{n,m}/\beta_{n,m}^{(2)}$ we obtain
	$\bar{\rho}_{n+1}s_{n+2}\bar{\rho}_{n+1}\bar{\rho}_{n+1}^{-1}=\bar{\sigma}_1^2s_{n+2}\bar{\rho}_{n+1}$.
	It follows that 
	\begin{equation}\label{cor}
	\bar{\rho}_{n+1}s_{n+2}\bar{\rho}_{n+1}^{-1}s_{n+2}^{-1}=\bar{\sigma}_1^2.
	\end{equation}
	 Moreover, projecting the relation $\sigma_2\rho_{n+1}\sigma_2^{-1}=\rho_{n+1}$ ($(vii)$ of Proposition \ref{pres3}) into $\beta_{n,m}/\beta_{n,m}^{(2)}$ we obtain $\bar{\sigma}_1\bar{\rho}_{n+1}\bar{\sigma}_1^{-1}=\bar{\rho}_{n+1}$, which implies that $\bar{\rho}_{n+1}$ commutes with $\bar{\sigma}_1$. The projection of the relation $\sigma_1\rho_{n+1}\sigma_1^{-1}=\sigma_1^2\rho_{n+2}$ ($(vii)$ of Proposition \ref{pres3}) into $\beta_{n,m}/\beta_{n,m}^{(2)}$ gives $\bar{\sigma}_1\bar{\rho}_{n+1}\bar{\sigma}_1^{-1}=\bar{\sigma}_1^2s_{n+2}\bar{\rho}_{n+1}$, which, using the fact that $\bar{\rho}_{n+1}$ commutes with $\bar{\sigma}_1$, yields $s_{n+2}=\bar{\sigma}_1^{-2}$. Combining the relation $s_{n+2}=\bar{\sigma}_1^{-2}$ with relation \eqref{cor} and the fact that $\bar{\rho}_{n+1}$ commutes with $\bar{\sigma}_1$, we conclude that $\bar{\sigma}_1^2=1$ and $s_{n+2}=1$.
	
	We now prove that the generators $\rho_{n+1}, \dots, \rho_{n+m}$ all have the same image in $\beta_{n,m}/\beta_{n,m}^{(2)}$ under $pr$. 
	Since $m\geq 5$, the relation $\sigma_l\rho_k\sigma_l^{-1}=\rho_k$ ($(vii)$ of Proposition \ref{pres3}), which holds for $l\neq k-n, k-n-1$ and $n+1\leq k\leq n+m$, exists in $\beta_{n,m}$. Let $n+1\leq k\leq n+m-1$. Thus, there exists some $l$, where $1\leq l\leq m-1$, such that $\sigma_l\rho_k\sigma_l^{-1}=\rho_k$. So, projecting this relation into $\beta_{n,m}/\beta_{n,m}^{(2)}$ it follows that $\bar{\sigma}_1s_k\bar{\rho}_{n+1}\bar{\sigma}_1^{-1}=s_k\bar{\rho}_{n+1}$. Thus, $\bar{\sigma_1}$ commutes with $s_k\bar{\rho}_{n+1}$, for $n+1\leq k\leq n+m-1$. Furthermore, the projection of the relation $\sigma_{k-n}\rho_k\sigma_{k-n}^{-1}=\sigma_{k-n}^2\rho_{k+1}$ ($(vii)$ of Proposition \ref{pres3}) in $\beta_{n,m}/\beta_{n,m}^{(2)}$, for $n+1\leq k \leq n+m-1$, implies that $\bar{\sigma}_1s_k\bar{\rho}_{n+1}\bar{\sigma}_1^{-1}=1\cdot s_{k+1}\bar{\rho}_{n+1}$. Since $\bar{\sigma}_1$ commutes with $s_k\bar{\rho}_{n+1}$, we obtain $s_k\bar{\rho}_{n+1}=s_{k+1}\bar{\rho}_{n+1}$ for $n+1\leq k\leq n+m-1$ and therefore $s_k=s_{k+1}$ for $n+1\leq k\leq n+m-1$. Hence, $pr(\rho_{n+1})=\dots=pr(\rho_{n+m})=\bar{\rho}_{n+1}$, and $\bar{\sigma}_1\bar{\rho}_{n+1}=\bar{\rho}_{n+1}\bar{\sigma}_1$.
	
	As before, we have $pr(B_{i,j})=a_{i,j}\cdot B_{i,n+1}\beta_{n,m}^{(2)}\in \beta_{n,m}/\beta_{n,m}^{(2)}$, for $1\leq i\leq n, \ n+1\leq j\leq n+m$, where $a_{i,j}\in \ker(p)=\beta_{n,m}^{(1)}/\beta_{n,m}^{(2)}$ and $a_{i,n+1}$ may be taken to be $1$. We will denote $B_{i,n+1}\beta_{n,m}^{(2)}$ by $\bar{B}_{i,n+1}$. Since $m\geq 5$, the relation $\sigma_rB_{i,j}\sigma_r^{-1}=B_{i,j}$ ($(vi)$ of Proposition \ref{pres3}), for $r\neq j-n-1, j-n$, where $n+1\leq j\leq n+m$, $1\leq r\leq m-1$ and $1\leq i\leq n$, exists in $\beta_{n,m}$. Let $n+1\leq j\leq n+m-1$. Thus, there exists some $l$, where $1\leq l\leq m-1$, such that $\sigma_lB_{i,j}\sigma_l^{-1}=B_{i,j}$. So, projecting this relation into $\beta_{n,m}/\beta_{n,m}^{(2)}$ it follows that $\bar{\sigma}_1a_{i,j}\bar{B}_{i,n+1}\bar{\sigma}_1^{-1}=a_{i,j}\bar{B}_{i,n+1}$. Hence, $\bar{\sigma}_1$ commutes with $a_{i,j}\bar{B}_{i,n+1}$ for $1\leq i\leq n,\ n+1\leq j\leq n+m-1$. Furthermore, the projection of the relation $\sigma_{j-n}B_{i,j}\sigma_{j-n}^{-1}=B_{i,j+1}$ ($(vi)$ of Proposition \ref{pres3}) in $\beta_{n,m}/\beta_{n,m}^{(2)}$, for $1\leq i\leq n$ and $n+1\leq j\leq n+m-1$, yields $\bar{\sigma}_1a_{i,j}\bar{B}_{i,n+1}\bar{\sigma}_1^{-1}=a_{i,j+1}\bar{B}_{i,n+1}$. Hence, $a_{i,j}=a_{i,j+1}$, for $1\leq i\leq n$ and $n+1\leq j\leq n+m-1$. As a result, $pr(B_{i,j})=\bar{B}_{i,n+1}$, for $1\leq i\leq n, \ n+1\leq j\leq n+m$, and moreover $\bar{\sigma}_1\bar{B}_{i,n+1}=\bar{B}_{i,n+1}\bar{\sigma}_1$, for $1\leq i\leq n$.
	
	We conclude that the projection of the generating set of $\beta_{n,m}$ given in Proposition \ref{pres3} gives rise to the following generating set for $\beta_{n,m}/\beta_{n,m}^{(2)}$: $$\{\bar{\sigma}_1, \bar{\rho}_{n+1}, \bar{B}_{1,n+1},\dots, \bar{B}_{n,n+1}\}.$$
	
	Once again, projecting relations ($i$) and ($iv$) of Proposition \ref{pres3}, that hold in $\beta_{n,m}$, into $\beta_{n,m}/\beta_{n,m}^{(2)}$, we obtain $\bar{B}_{i,n+1}\bar{B}_{k,n+1}=\bar{B}_{k,n+1}\bar{B}_{i,n+1}$, for $1\leq i,k\leq n$ and
	$\bar{\rho}_{n+1}\bar{B}_{i,n+1}=\bar{B}_{i,n+1}\bar{\rho}_{n+1}$, for $1\leq i\leq n$, respectively.
	
	Finally, the projection of relation ($ii$) of Proposition \ref{pres3} in $\beta_{n,m}/\beta_{n,m}^{(2)}$ implies that $\bar{\rho}_{n+1}^2{\displaystyle \prod_{i=1}^{n} \bar{B}_{i,n+1}}=1$.
	
	We observe that the generators $\{\bar{\sigma}_1, \bar{\rho}_{n+1}, \bar{B}_{1,n+1},\dots, \bar{B}_{n,n+1}\}$ of $\beta_{n,m}/\beta_{n,m}^{(2)}$ commute pairwise. That means that this quotient group is an Abelian group. Applying the same argument as before, we conclude that  $\beta_{n,m}^{(1)}\subseteq \beta_{n,m}^{(2)}$. By the definition of the derived series, we know that $\beta_{n,m}^{(2)}\subseteq \beta_{n,m}^{(1)}$, and therefore $\beta_{n,m}^{(2)}=\beta_{n,m}^{(1)}$, which completes the proof. 
\end{proof}

\section{The splitting problem for the cases $n\geq3$}\label{S5}
In this section, for $n\geq 3$ and $m>1$, we will provide necessary conditions, for the short exact sequence \eqref{ms} to split. Let $N$ be a normal subgroup of $B_{n,m}(\mathbb{R}P^2)$ that is also contained in $\beta_{n,m}$. Consider the following diagram of short exact sequences:
\begin{equation}\label{cd}
\begin{tikzcd}
1 \ar[r] & \beta_{n,m} \ar[d,"pr|_{\beta_{n,m}}"] \ar[r] & B_{n,m}(\mathbb{R}P^2) \ar[d,"pr"] \ar[r, "\bar{q}"] & B_n(\mathbb{R}P^2) \ar[-,double line with arrow={-,-}]{d} \ar[r] & 1 \\
1 \ar[r] & \beta_{n,m}/N \ar[r] & B_{n,m}(\mathbb{R}P^2)/N \ar[r, "q_*"] & B_n(\mathbb{R}P^2) \ar[r] & 1,
\end{tikzcd}
\end{equation}
where $q_*:B_{n,m}(\mathbb{R}P^2)/N\to B_n(\mathbb{R}P^2)$ denotes the homomorphism induced by $\bar{q}$ and $pr$ is the canonical projection from $B_n(\mathbb{R}P^2)$ to $B_{n,m}(\mathbb{R}P^2)/N$.
Suppose that the homomorphism $\bar{q}:B_{n,m}(\mathbb{R}P^2)\to B_n(\mathbb{R}P^2)$ admits a section $\bar{s}:B_n(\mathbb{R}P^2)\to B_{n,m}(\mathbb{R}P^2)$. Then $q_*:B_{n,m}(\mathbb{R}P^2)/N\to B_n(\mathbb{R}P^2)$ admits a section $s_*=pr\circ\bar{s}:B_n(\mathbb{R}P^2)\to B_{n,m}(\mathbb{R}P^2)/N$.
Since the subgroups of the lower central series of a group are characteristic, $\Gamma:=\Gamma_2(\beta_{n,m})$ is a characteristic subgroup of $\beta_{n,m}$. Further $\beta_{n,m}$ is a normal subgroup of $B_{n,m}(\mathbb{R}P^2)$, so $\Gamma$ is a normal subgroup of $B_{n,m}(\mathbb{R}P^2)$ contained in $\beta_{n,m}$. Thus, we obtain the following short exact sequence from \eqref{cd}:
\begin{equation}\label{qfinal}
\begin{tikzcd}
1 \ar[r]& \beta_{n,m}/\Gamma \ar[r] & B_{n,m}(\mathbb{R}P^2)/\Gamma \ar[r, "q"] & B_n(\mathbb{R}P^2)\ar[r] & 1.
\end{tikzcd}
\end{equation}

In order to study the short exact sequence \eqref{qfinal} we need to know a presentation of the quotients $\beta_{n,m}/\Gamma$ and $B_{n,m}(\mathbb{R}P^2)/\Gamma$.
We already  have a presentation of the Abelian group $\beta_{n,m}/\Gamma$, given by Corollary \ref{presofquotient}. We recall that $\beta_{n,m}/\Gamma\cong \mathbb{Z}^n\times \mathbb{Z}_2$, where $\rho$, $\beta_1,\dots ,\beta_{n-1}$ generate the $\mathbb{Z}^n$-component and $\sigma$ the $\mathbb{Z}_2$-component. One can obtain a presentation of $B_{n,m}(\mathbb{R}P^2)/\Gamma$ using the presentation of $\beta_{n,m}/\Gamma$, given by Corollary \ref{presofquotient}, the presentation of $B_n(\mathbb{R}P^2)$ given in Theorem \ref{vb}, and applying standard techniques for obtaining a presentation of group extensions as described in [\cite{johnsonpresentation}, p. 139] and presented in detail in the proofs of Propositions \ref{pre1}, \ref{presentationpunctures}, \ref{pres3} and Theorem \ref{pre4}. 

However, for $N=\Gamma$ the group $\beta_{n,m}/\Gamma$ is the Abelianisation of $\beta_{n,m}$. Thus, from commutative diagram of short exact sequences \eqref{cd}, a presentation of $B_{n,m}(\mathbb{R}P^2)/\Gamma$ may also be obtained straightforwardly by considering the union of the generators of $\beta_{n,m}/\Gamma$ and the coset representatives of the generators of $B_n(\mathbb{R}P^2)$ as a set of generators, and the relations of $B_{n,m}(\mathbb{R}P^2)$, given in Theorem \ref{pre4}, projected into $B_{n,m}(\mathbb{R}P^2)/\Gamma$ as a set of relations, as described in the following proposition.
\begin{proposition}\label{restrictions}
	For $n\geq 2, \ m\geq 1$, the following constitutes a presentation of $B_{n,m}(\mathbb{R}P^2)/\Gamma$, where $\Gamma=\Gamma_2(\beta_{n,m})$.
	\\
	\\ \textbf{Generators}: $\rho$, $\sigma$, $\beta_1,\dots ,\beta_n$, $\tau_1,\dots ,\tau_{n-1}$, $q_1,\dots ,q_n$.
	\\\textbf{Relations:} \begin{enumerate}[label=(\Roman*)]	
		\item The relations ($i$)-($iii$) from Corollary \ref{presofquotient},
		\item For $1\leq i, k \leq n-1,$ \ $1\leq j \leq n,$\\
		$\tau_i\tau_k=\tau_k\tau_i,$ \ for \ $|i-k|>1$, \\
		$\tau_i\tau_{i+1}\tau_i=\tau_{i+1}\tau_i\tau_{i+1},$ for $1\leq i < n-1,$\\
		$\tau_iq_j=q_j\tau_i,$ \ for \ $j \neq i, i+1,$\\
		$q_i=\tau_iq_{i+1}\tau_i,$\\
		$\tau_i^2=q_{i+1}^{-1}q_{i}^{-1}q_{i+1}q_i,$\\
		$q_1^2=(\tau_1\tau_2\dots\tau_{n-2}\tau_{n-1})\beta_n^m(\tau_{n-1}\tau_{n-2}\dots\tau_2\tau_1).$
		\item \begin{enumerate}
			\item For \ $1\leq i\leq n-1$, \ $1\leq j\leq n$,\\
			$\sigma\tau_i=\tau_i\sigma,$\\
			$\sigma q_j=q_j\sigma$.
			\item For \ $1\leq i\leq n-1$, \ $1\leq j\leq n$,\\
			$\rho\tau_i=\tau_i\rho,$\\
			$q_j\rho = \beta_j\rho q_j$.		
			\item For \ $1\leq i\leq n-1$, \ $1\leq k\leq n$,\\
			\begin{equation*}
			\tau_i\beta_k\tau_i^{-1}=
			\begin{cases}
			\beta_k, & \text{for}\ i\neq k-1,k.\\
			\beta_{k-1}, & \text{for} \ i=k-1.\\
			\beta_{k+1}, & \text{for} \ i=k.
			\end{cases}
			\end{equation*}
			\item For \ $1\leq j, k\leq n$,\\
			\begin{equation*}
			q_j\beta_kq_j^{-1}=
			\begin{cases}
			\beta_k, & \text{for}\ j\neq k.\\
			\beta_k^{-1}, & \text{for} \ j=k.
			\end{cases}
			\end{equation*}\\\\
		\end{enumerate}
	\end{enumerate}
	
\end{proposition}

\begin{remark}
	In the statement of Proposition \ref{restrictions} we denote by $\tau_{i}$ and $q_{j}$ the coset representatives of the generators $\sigma_i$ and $\rho_j$ of $B_n(\mathbb{R}P^2)$, given by Theorem \ref{vb}, for $1\leq i\leq n-1$ and $1\leq j\leq n$.
\end{remark}          

Now, suppose that there exists a section $\bar{s} : B_n(\mathbb{R}P^2) \to B_{n,m}(\mathbb{R}P^2)$ for the homomorphism $\bar{q}: B_{n,m}(\mathbb{R}P^2) \to B_n(\mathbb{R}P^2)$. As we already saw, it follows that there exists a section $s_*:B_n(\mathbb{R}P^2) \to B_{n,m}(\mathbb{R}P^2)/\Gamma$ for $q_*:B_{n,m}(\mathbb{R}P^2)/\Gamma\to B_n(\mathbb{R}P^2)$. From Corollary \ref{presofquotient}, the set $\{\beta_1,\dots ,\beta_{n-1}, \rho, \sigma\}$ generates $\ker(q_*)$, which is the group $\beta_{n,m}/\Gamma\cong \mathbb{Z}^n\times \mathbb{Z}_2$. This allows us to describe the image of the elements of the generating set of $B_n(\mathbb{R}P^2)$, under the section $s_*$, as follows:\\
\begin{equation}\label{imageoft}
s_*(\tau_i) =\tau_i\cdot\beta_1^{k_{i,1}}\beta_2^{k_{i,2}}\cdots\beta_{n-1}^{k_{i,n-1}}\rho^{l_i}\sigma^{m_i},\ \text{for}\ 1\leq i \leq n-1,\\
\end{equation}

\begin{equation}\label{imageofq}
s_*(q_j) =
q_j\cdot\beta_1^{\bar{k}_{j,1}}\beta_2^{\bar{k}_{j,2}}\cdots\beta_{n-1}^{\bar{k}_{j,n-1}}\rho^{\bar{l}_j}\sigma^{\bar{m}_j},\ \text{for}\ 1\leq j \leq n,\\
\end{equation} 
where $k_{i,1}, k_{i,2}, \dots, k_{i,n-1}, l_i, \bar{k}_{j,1}, \bar{k}_{j,2}, \bar{k}_{j,n-1}, \bar{l}_j \in \mathbb{Z}\ \text{and}\ m_i, \bar{m}_j\in\{0,1\}.$ Note that these integers are unique for every $s_*(\tau_i)$ and $s_*(q_j)$, where $1\leq i \leq n-1$ and $1\leq j \leq n$.
Under the assumption that there exists a section $s_*:B_n(\mathbb{R}P^2) \to B_{n,m}(\mathbb{R}P^2)/\Gamma$, the images under $s_*$ of the relations in $B_n(\mathbb{R}P^2)$ are also relations in $B_{n,m}(\mathbb{R}P^2)/\Gamma$. In this way, we will obtain further information regarding the exponents in the formulas \eqref{imageoft}, \eqref{imageofq} and possible restrictions for the value of $m$, under the assumption that the short exact sequence \eqref{ms} splits. 

Based on the presentation of $B_n(\mathbb{R}P^2)$ for $n\geq 3$ given by Theorem \ref{vb}, we have the following six relations, which hold in $B_{n,m}(\mathbb{R}P^2)/\Gamma$:
\begin{itemize}
	\item \textbf{R1.} $s_*(\tau_iq_j) = s_*(q_j\tau_i)$, for $j\neq i, i+1,\ 1\leq i\leq n-1,\ 1\leq j\leq n.$
	\item \textbf{R2.} $s_*(q_i) = s_*(\tau_iq_{i+1}\tau_i)$, for $1\leq i\leq n-1.$
	\item \textbf{R3.} $s_*(\tau_i\tau_{i+1}\tau_i) = s_*(\tau_{i+1}\tau_i\tau_{i+1})$, for $1\leq i< n-1.$
	\item \textbf{R4.} $s_*(\tau_i^2) = s_*(q_{i+1}^{-1}q_{i}^{-1}q_{i+1}q_{i})$, for $1\leq i\leq n-1.$
	\item \textbf{R5.} $s_*(\tau_i\tau_j) = s_*(\tau_j\tau_i)$, for \ $|i-j|>1$ \ and \ $1\leq i,j\leq n-1.$
	\item \textbf{R6.} $s_*(q_1^2)=s_*(\tau_1\tau_2\cdots\tau_{n-2}\tau_{n-1}^2\tau_{n-2}\cdots\tau_2\tau_1).$
\end{itemize}

In order to prove Proposition \ref{rquotient}, we will make use of the following relations in $B_{n,m}(\mathbb{R}P^2)/\Gamma$ that appear in Proposition \ref{restrictions}: 
\begin{enumerate}
	\item $\sigma^2=1$, relation ($ii$) of Corollary \ref{presofquotient}\label{enum1},
	\item $\beta_n = \beta_{1}^{-1}\cdots \beta_i^{-1}\cdots\beta_{n-1}^{-1}\rho^{-2}$, relation ($iii$) of Corollary \ref{presofquotient},
	\item $\beta_j, \sigma, \rho$ commute pairwise, relation ($i$) of Corollary \ref{presofquotient},
	\item $\tau_i $ commutes with $\sigma, \rho$, relations $(III)(a)$, $(III)(b)$ of Proposition \ref{restrictions},
	\item $\beta_i\tau_i = \tau_i \beta_{i+1}$ and $\beta_{i+1}\tau_i = \tau_i \beta_i$, relation $(III)(c)$ of Proposition \ref{restrictions},
	\item $q_i = \tau_iq_{i+1}\tau_i$ and $\tau_iq_j = q_j\tau_i$, for $j\neq i, i+1$, relation $(II)$ of Proposition \ref{restrictions},
	\item $q_j$ commutes with $\sigma$, relation $(III)(a)$ of Proposition \ref{restrictions},
	\item $\beta_jq_j = q_j\beta_j^{-1}$ and $\beta_iq_j = q_j\beta_i$, for $j\neq i$, relation $(III)(d)$ of Proposition \ref{restrictions},
	\item $\rho q_j = q_j\beta_j\rho$, which is deduced from $q_j\rho=\beta_j\rho q_j$ and $\beta_jq_j = q_j\beta_j^{-1}$ given by $(III)(d)$ and $(III)(b)$ of Proposition \ref{restrictions} respectively. Thus,  $\rho^k q_j = q_j\beta_j^k\rho^k,\ \text{for}\ k\in \mathbb{Z},\label{enum2}$
\end{enumerate}
where $1\leq i\leq n-1$ and $1\leq j \leq n$.
\begin{proposition}\label{rquotient}
	Let $m\geq1$ and $n\geq 3$.
	If the following short exact sequence
	\begin{equation}\label{new}
	\begin{tikzcd}
	1 \ar[r]& B_m(\mathbb{R}P^2\setminus\{x_1,\dots,x_n\}) \ar[r] & B_{n,m}(\mathbb{R}P^2) \ar[r, "\bar{q}"] & B_n(\mathbb{R}P^2)\ar[r] & 1
	\end{tikzcd}
	\end{equation}
	splits, then $m=k(n-1)$, where $k\geq1$.
\end{proposition}

\begin{proof}
	Based on the above discussion and on the assumption that the short exact sequence \eqref{new} splits, we will examine the relations \textbf{R1.-R6.}, which hold in $B_{n,m}(\mathbb{R}P^2)/\Gamma$, from which we will deduce that $m$ is a multiple of $(n-1)$.
	
	We start with the relation \textbf{R1.} where $s_*(\tau_iq_j) = s_*(q_j\tau_i)$, for $j\neq i, i+1,\ 1\leq i\leq n-1$ and $1\leq j\leq n$. We focus on the case where $i = n-1$ and $1\leq j \leq n-2$. Using the formulas \eqref{imageoft}, \eqref{imageofq} and the above set of relations \ref{enum1}-\ref{enum2} we have the following:
	
	\begin{flalign*}
	s_*(\tau_{n-1}q_j) &= \tau_{n-1}\beta_1^{k_{n-1,1}}\cdots \beta_j^{k_{n-1,j}}\cdots\beta_{n-1}^{k_{n-1,n-1}}\rho^{l_{n-1}}\sigma^{m_{n-1}}q_j\beta_1^{\bar{k}_{j,1}}\cdots \beta_j^{\bar{k}_{j,j}}\cdots\beta_{n-1}^{\bar{k}_{j,n-1}}\rho^{\bar{l}_j}\sigma^{\bar{m}_j} &\\ 
	&=\tau_{n-1}q_j\beta_1^{k_{n-1,1}}\cdots \beta_j^{-k_{n-1,j}}\cdots\beta_{n-1}^{k_{n-1,n-1}}\beta_j^{l_{n-1}}\rho^{l_{n-1}}\sigma^{m_{n-1}}\beta_1^{\bar{k}_{j,1}}\cdots \beta_j^{\bar{k}_{j,j}}\cdots\beta_{n-1}^{\bar{k}_{j,n-1}}\rho^{\bar{l}_j}\sigma^{\bar{m}_j}&\\
	&= \tau_{n-1}q_j\beta_1^{k_{n-1,1} + \bar{k}_{j,1}}\cdots\beta_j^{-k_{n-1,j}+ l_{n-1} + \bar{k}_{j,j}}\cdots\beta_{n-1}^{k_{n-1,n-1} + \bar{k}_{j,n-1}}\rho^{l_{n-1} + \bar{l_j}}\sigma^{m_{n-1} + \bar{m_j}}
	\end{flalign*}
	and
	\begin{flalign*}
	s_*(q_j\tau_{n-1}) &= q_j\beta_1^{\bar{k}_{j,1}}\cdots \beta_j^{\bar{k}_{j,j}}\cdots\beta_{n-1}^{\bar{k}_{j,n-1}}\rho^{\bar{l}_j}\sigma^{\bar{m}_j} \tau_{n-1}\beta_1^{k_{n-1,1}}\cdots \beta_j^{k_{n-1,j}}\cdots\beta_{n-1}^{k_{n-1,n-1}}\rho^{l_{n-1}}\sigma^{m_{n-1}}&\\
	&= q_j\tau_{n-1}\beta_1^{\bar{k}_{j,1}}\cdots \beta_j^{\bar{k}_{j,j}}\cdots\beta_{n-2}^{\bar{k}_{j,n-2}}\beta_{n}^{\bar{k}_{j,n-1}}\rho^{\bar{l}_j}\sigma^{\bar{m}_j}\beta_1^{k_{n-1,1}}\cdots \beta_j^{k_{n-1,j}}\cdots\beta_{n-1}^{k_{n-1,n-1}}\rho^{l_{n-1}}\sigma^{m_{n-1}}&\\
	&= q_j\tau_{n-1}\beta_1^{\bar{k}_{j,1}}\cdots \beta_j^{\bar{k}_{j,j}}\cdots\beta_{n-2}^{\bar{k}_{j,n-2}}(\beta_{1}^{-\bar{k}_{j,n-1}}\cdots \beta_j^{-\bar{k}_{j,n-1}}\cdots\beta_{n-1}^{-\bar{k}_{j,n-1}}\rho^{-2\bar{k}_{j,n-1}})\rho^{\bar{l}_j}\sigma^{\bar{m}_j}&\\
	&\quad \beta_1^{k_{n-1,1}}\cdots \beta_j^{k_{n-1,j}}\cdots\beta_{n-1}^{k_{n-1,n-1}}\rho^{l_{n-1}}\sigma^{m_{n-1}}&\\
	&= q_j\tau_{n-1}\beta_1^{\bar{k}_{j,1}-\bar{k}_{j,n-1}+k_{n-1,1}} \cdots\beta_j^{\bar{k}_{j,j}-\bar{k}_{j,n-1}+k_{n-1,j}}\cdots\beta_{n-2}^{\bar{k}_{j,n-2}-\bar{k}_{j,n-1}+k_{n-1,n-2}}\\
	&\quad\beta_{n-1}^{-\bar{k}_{j,n-1}+k_{n-1,n-1}}
	\rho^{\bar{l}_j-2\bar{k}_{j,n-1}+l_{n-1}}\sigma^{\bar{m_j}+ m_{n-1}}.
	\end{flalign*}
	Comparing the coefficients of the elements $\rho$ and $\beta_{j}$ of $B_{n,m}(\mathbb{R}P^2)/\Gamma$ in these two equations, we obtain the following:
	\begin{equation}\label{8}
	\bar{k}_{j,n-1} = 0,\ \text{for}\ 1\leq j \leq n-2.
	\end{equation}
	\begin{equation}\label{9}
	l_{n-1} = 2k_{n-1,j}, \ \text{for}\ 1\leq j \leq n-2.
	\end{equation}
	In the case where $j = n$ and $1\leq i < n-1$ we have the following:
	\begin{flalign*}
	s_*(\tau_iq_n) &= \tau_i\beta_1^{k_{i,1}}\cdots \beta_i^{k_{i,i}}\beta_{i+1}^{k_{i,i+1}}\cdots\beta_{n-1}^{k_{i,n-1}}\rho^{l_i}\sigma^{m_i}q_n\beta_1^{\bar{k}_{n,1}}\cdots\beta_i^{\bar{k}_{n,i}}\beta_{i+1}^{\bar{k}_{n,i+1}}\cdots\beta_{n-1}^{\bar{k}_{n,n-1}}\rho^{\bar{l}_n}\sigma^{\bar{m}_n} &\\ 
	&=\tau_iq_n\beta_1^{k_{i,1}}\cdots \beta_i^{k_{i,i}}\beta_{i+1}^{k_{i,i+1}}\cdots\beta_{n-1}^{k_{i,n-1}}\beta_n^{l_i}\rho^{l_i}\sigma^{m_i}\beta_1^{\bar{k}_{n,1}}\cdots\beta_i^{\bar{k}_{n,i}}\beta_{i+1}^{\bar{k}_{n,i+1}}\cdots\beta_{n-1}^{\bar{k}_{n,n-1}}\rho^{\bar{l}_n}\sigma^{\bar{m}_n}&\\
	&= \tau_iq_n\beta_1^{k_{i,1}-l_i + \bar{k}_{n,1}}\cdots\beta_i^{k_{i,i}-l_i + \bar{k}_{n,i}}\beta_{i+1}^{k_{i,i+1}-l_i + \bar{k}_{n,i+1}}\cdots\beta_{n-1}^{k_{i,n-1}-l_i + \bar{k}_{n,n-1}}\rho^{l_i-2l_i + \bar{l}_n}\sigma^{m_i + \bar{m}_n}
	\end{flalign*}
	and
	\begin{flalign*}
	s_*(q_n\tau_i) &= q_n\beta_1^{\bar{k}_{n,1}}\cdots\beta_i^{\bar{k}_{n,i}}\beta_{i+1}^{\bar{k}_{n,i+1}}\cdots\beta_{n-1}^{\bar{k}_{n,n-1}}\rho^{\bar{l}_n}\sigma^{\bar{m}_n}  \tau_i\beta_1^{k_{i,1}}\cdots \beta_i^{k_{i,i}}\beta_{i+1}^{k_{i,i+1}}\cdots\beta_{n-1}^{k_{i,n-1}}\rho^{l_i}\sigma^{m_i}&\\
	&= q_n\tau_i\beta_1^{\bar{k}_{n,1}}\cdots\beta_{i+1}^{\bar{k}_{n,i}}\beta_i^{\bar{k}_{n,i+1}}\cdots\beta_{n-1}^{\bar{k}_{n,n-1}}\rho^{\bar{l}_n}\sigma^{\bar{m}_n}\beta_1^{k_{i,1}}\cdots \beta_i^{k_{i,i}}\beta_{i+1}^{k_{i,i+1}}\cdots\beta_{n-1}^{k_{i,n-1}}\rho^{l_i}\sigma^{m_i}&\\
	&= q_n\tau_i\beta_1^{\bar{k}_{n,1}+k_{i,1}}\cdots\beta_i^{\bar{k}_{n,i+1}+k_{i,i}}\beta_{i+1}^{\bar{k}_{n,i}+k_{i,i+1}}\cdots\beta_{n-1}^{\bar{k}_{n,n-1}+k_{i,n-1}}\rho^{\bar{l}_n+l_i}\sigma^{\bar{m}_n+ m_i}.
	\end{flalign*}
	Comparing the coefficients of the element $\beta_{1}$ of $B_{n,m}(\mathbb{R}P^2)/\Gamma$ in these two equations, we obtain the following:
	\begin{equation}\label{10}
	l_i = 0,\ \text{for}\ 1\leq i \leq n-2.
	\end{equation}
	
	We continue with \textbf{R2.} for which $s_*(q_i) = s_*(\tau_iq_{i+1}\tau_i)$, for $1\leq i\leq n-1$.
	We will discuss the case $i = n-1$ separately. So, for $1\leq i \leq n-2$ we have from \eqref{10} that $l_i=0$, and once again, based on the formulas \eqref{imageoft}, \eqref{imageofq} and the set of relations \ref{enum1}-\ref{enum2} we have the following:\\
	$ s_*(q_i)= q_i\beta_1^{\bar{k}_{i,1}}\cdots\beta_i^{\bar{k}_{i,i}}\beta_{i+1}^{\bar{k}_{i,i+1}}\cdots \beta_{n-1}^{\bar{k}_{i,n-1}}\rho^{\bar{l}_i}\sigma^{\bar{m}_i}$
	\\\\and
	\begin{align*}
	s_*(\tau_iq_{i+1}\tau_i)&=\tau_iq_{i+1}\tau_i\beta_1^{2k_{i,1}+\bar{k}_{i+1,1}}\cdots\beta_i^{-k_{i,i+1}+\bar{k}_{i+1,i+1}+k_{i,i}}\beta_{i+1}^{k_{i,i}+\bar{k}_{i+1,i}+k_{i,i+1}}\cdots\beta_{n-1}^{2k_{i,n-1}+\bar{k}_{i+1,n-1}}\rho^{\bar{l}_{i+1}}&\\
	&\quad\sigma^{\bar{m}_{i+1}}.
	\end{align*}
	Note that $\sigma^{2m_i}=1$, since $\sigma$ is an element of order two. 
	Comparing the coefficients of the elements $\sigma$ and $\rho$ of $B_{n,m}(\mathbb{R}P^2)/\Gamma$ in these two equations, we obtain the following:
	\begin{equation}\label{12}
	\bar{m}_i = \bar{m}_{i+1},\ \text{for}\ 1\leq i \leq n-2.
	\end{equation}
	\begin{equation}\label{13}
	\bar{l}_i = \bar{l}_{i+1},\ \text{for}\ 1\leq i \leq n-2.
	\end{equation}
	In the case where $i=n-1$ we have the following:\\
	$ s_*(q_{n-1})= q_{n-1}\beta_1^{\bar{k}_{n-1,1}}\cdots\beta_i^{\bar{k}_{n-1,i}}\cdots \beta_{n-1}^{\bar{k}_{n-1,n-1}}\rho^{\bar{l}_{n-1}}\sigma^{\bar{m}_{n-1}}$
	\\\\and
	\begin{flalign*}
	s_*(\tau_{n-1}q_n\tau_{n-1}) &=\tau_{n-1}\beta_{1}^{k_{{n-1},1}}\cdots\beta_{n-1}^{k_{{n-1},n-1}}
	\rho^{l_{n-1}}
	\sigma^{m_{n-1}}q_n\beta_1^{\bar{k}_{n,1}}\cdots\beta_{n-1}^{\bar{k}_{n,n-1}}\rho^{\bar{l}_{n}}\sigma^{\bar{m}_{n}}&\\
	& \quad \tau_{n-1}\beta_{1}^{k_{{n-1},1}}\cdots\beta_{n-1}^{k_{{n-1},n-1}}
	\rho^{l_{n-1}}
	\sigma^{m_{n-1}}.
	%&=\tau_{n-1}q_n\beta_{1}^{k_{{n-1},1}}\cdots\beta_{n-1}^{k_{{n-1},n-1}}\beta_n^{l_{n-1}}
	%\rho^{l_{n-1}}
	%\sigma^{m_{n-1}}\tau_{n-1}\beta_1^{\bar{k}_{n,1}}\cdots\beta_{n}^{\bar{k}_{n,n-1}}\rho^{\bar{l}_{n}}\sigma^{\bar{m}_{n}}&\\
	%&\quad \beta_{1}^{k_{{n-1},1}}\cdots\beta_{n-1}^{k_{{n-1},n-1}}
	%\rho^{l_{n-1}}
	%\sigma^{m_{n-1}}&\\
	%&=\tau_{n-1}q_n\tau_{n-1}\beta_{1}^{k_{{n-1},1}}\cdots\beta_n^{k_{{n-1},n-1}}\beta_{n-1}^{l_{n-1}}
	%\rho^{l_{n-1}}
	%\sigma^{m_{n-1}}\beta_1^{\bar{k}_{n,1}}\cdots\beta_{n}^{\bar{k}_{n,n-1}}\rho^{\bar{l}_{n}}\sigma^{\bar{m}_{n}}&\\
	%&\quad \beta_{1}^{k_{{n-1},1}}\cdots\beta_{n-1}^{k_{{n-1},n-1}}
	%\rho^{l_{n-1}}\sigma^{m_{n-1}}&\\
	%&=\tau_{n-1}q_n\tau_{n-1}\beta_{1}^{2k_{{n-1},1}+\bar{k}_{n,1}-k_{n-1,n-1}-\bar{k}_{n,n-1}}
	%\cdots\beta_{n-1}^{l_{n-1}-\bar{k}_{n,n-1}}
	%\rho^{2l_{n-1}+\bar{l}_n-2k_{n-1,n-1}-2\bar{k}_{n,n-1}}\sigma^{\bar{m}_n}.
	\end{flalign*}
	We know from the set of relations \ref{enum1}-\ref{enum2} that the element $\sigma$ commutes with the other elements that appear in the second equation. This implies that the coefficient of $\sigma$ in the second equation is $\bar{m}_n$, since $\sigma^2=1$. Thus, comparing the coefficients of the element $\sigma$ of $B_{n,m}(\mathbb{R}P^2)/\Gamma$ in these two equations, we obtain:
	\begin{equation}\label{17}
	\bar{m}_n = \bar{m}_{n-1}.
	\end{equation}
	
	We continue with relation \textbf{R3.} where $s_*(\tau_i\tau_{i+1}\tau_i) = s_*(\tau_{i+1}\tau_i\tau_{i+1})$, for $1\leq i< n-1.$
	We will discuss the case $i = n-2$ separately. So, for $1\leq i < n-2$ we have from \eqref{10} that $l_i=0$, and once again, based on the formulas \eqref{imageoft}, \eqref{imageofq} and the set of relations \ref{enum1}-\ref{enum2} we have the following:
	\begin{flalign*}
	s_*(\tau_i\tau_{i+1}\tau_i)
	&=\tau_i\tau_{i+1}\tau_i\beta_1^{2k_{i,1}+k_{i+1,1}}\cdots\beta_{i-1}^{2k_{i,i-1}+k_{i+1,i-1}}\beta_i^{k_{i,i+2}+k_{i+1,i+1}+k_{i,i}}
	\beta_{i+1}^{k_{i,i}+k_{i+1,i}+k_{i,i+1}}&\\
	&\quad\beta_{i+2}^{k_{i,i+1}+k_{i+1,i+2}+{k_{i,i+2}}}
	\beta_{i+3}^{2k_{i,i+3}+k_{i+1,i+3}} \cdots\beta_{n-1}^{2k_{i,n-1}+k_{i+1,n-1}}
	\sigma^{m_{i+1}}
	\end{flalign*}
	and
	\begin{flalign*}
	s_*(\tau_{i+1}\tau_i\tau_{i+1})&=\tau_{i+1}\tau_i\tau_{i+1}\beta_1^{2k_{i+1,1}+k_{i,1}}\cdots\beta_{i-1}^{2k_{i+1,i-1}+k_{i,i-1}}
	\beta_i^{k_{i+1,i+1}+k_{i,i}+k_{i+1,i}}\beta_{i+1}^{k_{i+1,i+2}+k_{i,i+2}+k_{i+1,i+1}}\\
&\quad	\beta_{i+2}^{k_{i+1,i}+k_{i,i+1}+k_{i+1,i+2}}
	\beta_{i+3}^{2k_{i+1,i+3}+k_{i,i+3}}\cdots\beta_{n-1}^{2k_{i+1,n-1}+k_{i,n-1}}\sigma^{m_i}.
	\end{flalign*}
	Comparing the coefficients of the element $\beta_{i}$ of $B_{n,m}(\mathbb{R}P^2)/\Gamma$ in these two equations, it follows that $k_{i+1, i}=k_{i,i+2}$, for $1\leq i < n-2$. Using this relation and comparing the coefficients of the element $\beta_{i+1}$ of $B_{n,m}(\mathbb{R}P^2)/\Gamma$ in these two equations, it follows that $k_{i, i}+k_{i, i+1} = k_{i+1,i+2} + k_{i+1,i+1}$, for $1\leq i < n-2$. Moreover, comparing the coefficients of $\sigma$ we obtain $m_i = m_{i+1}$, for $1\leq i < n-2$. Thus, we have the following:
	\begin{equation}\label{21}
	m_i = m_{i+1},\ \text{for}\ 1\leq i < n-2.
	\end{equation}
	\begin{equation}\label{24}
	k_{i, i}+k_{i, i+1} = k_{i+1,i+2} + k_{i+1,i+1},\ \text{for}\ 1\leq i < n-2.
	\end{equation}
	In the case where $i=n-2$ we have the following:
	\begin{flalign*}
	s_*(\tau_{n-2}\tau_{n-1}\tau_{n-2}) = &\tau_{n-2}\beta_1^{k_{n-2,1}}\cdots\beta_{n-2}^{k_{n-2,n-2}}\beta_{n-1}^{k_{n-2,n-1}}\sigma^{m_{n-2}}
	\tau_{n-1}\beta_1^{k_{n-1,1}}\cdots\beta_{n-2}^{k_{n-1,n-2}}\beta_{n-1}^{k_{n-1,n-1}}\rho^{l_{n-1}}&\\
	&\sigma^{m_{n-1}}\tau_{n-2}\beta_1^{k_{n-2,1}}\cdots\beta_{n-2}^{k_{n-2,n-2}}\beta_{n-1}^{k_{n-2,n-1}}\sigma^{m_{n-2}}
	%&\\
	%&=\tau_{n-2}\tau_{n-1}\beta_1^{k_{n-2,1}}\cdots\beta_{n-2}^{k_{n-2,n-2}}\beta_n^{k_{n-2,n-1}}\sigma^{m_{n-2}}\tau_{n-2}
	%\beta_1^{k_{n-1,1}}\cdots\beta_{n-1}^{k_{n-1,n-2}}\beta_{n-2}^{k_{n-1,n-1}}\rho^{l-1}\sigma^{m_{n-1}}&\\
	%&\beta_1^{k_{n-2,1}}\cdots\beta_{n-2}^{k_{n-2,n-2}}\beta_{n-1}^{k_{n-2,n-1}}\sigma^{m_{n-2}}&\\
	%&=\tau_{n-2}\tau_{n-1}\tau_{n-2}\beta_1^{k_{n-2,1}}\cdots\beta_{n-1}^{k_{n-2,n-2}}\beta_n^{k_{n-2,n-1}}\sigma^{m_{n-2}}
	%\beta_1^{k_{n-1,1}}\cdots\beta_{n-1}^{k_{n-1,n-2}}\beta_{n-2}^{k_{n-1,n-1}}\rho^{l-1}\sigma^{m_{n-1}}&\\
	%&\beta_1^{k_{n-2,1}}\cdots\beta_{n-2}^{k_{n-2,n-2}}\beta_{n-1}^{k_{n-2,n-1}}\sigma^{m_{n-2}}&\\
	%&=\tau_{n-2}\tau_{n-1}\tau_{n-2}\beta_1^{2k_{n-2,1}+k_{n-1,1}-k_{n-2, n-1}}\cdots\beta_{n-2}^{k_{n-1,n-1}+k_{n-2,n-2}-k_{n-2,n-1}}&\\
	%&\beta_{n-1}^{k_{n-2,n-2}+k_{n-1,n-2}}\rho^{l_{n-1}-2k_{n-2,n-1}}
	%\sigma^{m_{n-1}}.
	\end{flalign*}
	and
	\begin{flalign*}
	s_*(\tau_{n-1}\tau_{n-2}\tau_{n-1}) = 
	&\tau_{n-1}\beta_1^{k_{n-1,1}}\cdots\beta_{n-2}^{k_{n-1,n-2}}\beta_{n-1}^{k_{n-1,n-1}}\rho^{l_{n-1}}\sigma^{m_{n-1}}\tau_{n-2}\beta_1^{k_{n-2,1}}\cdots\beta_{n-2}^{k_{n-2,n-2}}
	\beta_{n-1}^{k_{n-2,n-1}}&\\
	&\sigma^{m_{n-2}}\tau_{n-1}\beta_1^{k_{n-1,1}}\cdots\beta_{n-2}^{k_{n-1,n-2}}\beta_{n-1}^{k_{n-1,n-1}}
	\rho^{l_{n-1}}\sigma^{m_{n-1}}.
	%&\\
	%&=\tau_{n-1}\tau_{n-2}\tau_{n-1}\beta_1^{2k_{n-1,1}+k_{n-2,1}-k_{n-2, n-1}-k_{n-1, n-2}}\cdots\beta_{n-2}^{k_{n-1,n-1}+k_{n-2,n-2}-k_{n-2,n-1}}&\\
	%&\beta_{n-1}^{k_{n-1,n-1}-k_{n-2,n-1}-k_{n-1,n-2}}\rho^{2l_{n-1}-2k_{n-2,n-1}-2k_{n-1,n-2}}
	%\sigma^{m_{n-2}}.
	\end{flalign*}
	Since $\sigma$ commutes with the other elements that appear in these two equations, it follows that the coefficient of $\sigma$ in the first equation is $m_{n-1}$ and in second equation is $m_{n-2}$, because $\sigma^2=1$. Thus, we obtain the following:
	\begin{equation}\label{25}
	m_{n-1} = m_{n-2}.
	\end{equation}
	
	We continue with the relation \textbf{R4.} where $s_*(\tau_i^2) = s_*(q_{i+1}^{-1}q_{i}^{-1}q_{i+1}q_{i})$, for $1\leq i\leq n-1$. From \eqref{12} and \eqref{17}, we have $\bar{m}_1=\bar{m}_2=\dots=\bar{m}_{n-1}=\bar{m}_n=: \bar{M}$ and from \eqref{21} and \eqref{25} $m_1=m_2=\dots=m_{n-2}=m_{n-1}:=M$. From the set of relations \ref{enum1}-\ref{enum2} we have $\rho^k q_j = q_j\beta_j^k\rho^k$, and thus we obtain $\rho^kq_j^{-1}=q^{-1}_j\beta_j^{k}\rho^k$ and $\rho^{-k}q_j^{-1}=q^{-1}_j\beta_j^{-k}\rho^{-k}$, for $k\in \mathbb{Z},$ and $1\leq j \leq n$.  Moreover, using \eqref{10} and \eqref{13}, we obtain the following, for $1\leq i < n-1$:
	$$s_*(\tau_i^2)=\tau_i^2\beta_1^{2k_{i,1}}\cdots\beta_{i-1}^{2k_{i,i-1}}\beta_{i}^{k_{i,i+1}+k_{i,i}}\beta_{i+1}^{k_{i,i}+k_{i,i+1}}\beta_{i+2}^{2k_{i,i+2}}\cdots\beta_{n-1}^{2k_{i,n-1}}$$
	and
	$$s_*(q_{i+1}^{-1}q_{i}^{-1}q_{i+1}q_{i})=q_{i+1}^{-1}q_{i}^{-1}q_{i+1}q_{i}\beta_{1}^0\cdots \beta_{i-1}^{0}\beta_{i}^{-2\bar{k}_{i+1,i}+\bar{l}_{i}}\beta_{i+1}^{2\bar{k}_{i,i+1}-\bar{l}_{i}}\beta_{i+2}^{0}\cdots
	\beta_{n-1}^0\rho^0.$$
	Comparing the coefficients of $\beta_{1},\dots,\beta_{i-1},\beta_{i+2},\dots\beta_{n-1}$ and $\beta_{i+1}$ of $B_{n,m}(\mathbb{R}P^2)/\Gamma$ in these two equations, we obtain:
	\begin{equation}\label{29}
	k_{i,s}=0,\ \text{for}\ s=1, 2, \dots, n-1, \ s\neq i, i+1,\ 1\leq i\leq n-2.
	\end{equation}
	\begin{equation}\label{31}
	k_{i,i}+k_{i, i+1} = -\bar{l}_i+2\bar{k}_{i, i+1} ,\ 1\leq i \leq n-2.
	\end{equation}
	
	 In the last but one relation \textbf{R5.}, where $s_*(\tau_i\tau_j) = s_*(\tau_j\tau_i)$, for $|i-j|>1$ and $1\leq i,j\leq n-1,$ we will focus on the case where $i = n-1$ and $1\leq j\leq n-3$. From now on we suppose that $n\geq 4$, so that the relation that we will use exists.
	Thus, for these values, using \eqref{10} and \eqref{21}, we have the following:
	$$s_*(\tau_{n-1}\tau_j)=\tau_{n-1}\tau_j\beta_1^{k_{n-1,1}+k_{j,1}}\cdots\beta_{j}^{k_{n-1,j+1}+k_{j,j}}
	\beta_{j+1}^{k_{n-1,j}+k_{j,j+1}}\cdots\beta_{n-1}^{k_{n-1,n-1}+k_{j,n-1}}\rho^{l_{n-1}}$$
and
	\begin{flalign*}
	s_*(\tau_j\tau_{n-1})
	&=\tau_j\tau_{n-1}\beta_1^{k_{j,1}+k_{n-1,1}-k_{j,n-1}}\cdots\beta_{j}^{k_{j,j}+k_{n-1,j}-k_{j,n-1}}
	&\\
	&\beta_{j+1}^{k_{j,j+1}+k_{n-1,j+1}-k_{j,n-1}}\cdots \beta_{n-2}^{k_{j,n-2}+k_{n-1,n-2}-k_{j,n-1}}\quad \beta_{n-1}^{k_{n-1,n-1}-k_{j,n-1}}
	\rho^{l_{n-1}-2k_{j,n-1}}.
	\end{flalign*}
	Comparing the coefficients of the element $\rho$ of $B_{n,m}(\mathbb{R}P^2)/\Gamma$ in these two equations, it follows that $k_{j,n-1}=0$, for $1\leq j \leq n-3$. Using this result and comparing the coefficients of $\beta_{i}$ in these two equations, we obtain the following:
	\begin{equation}\label{37}
	k_{n-1, j} = k_{n-1, j+1},\ \text{for}\ 1\leq j \leq n-3.
	\end{equation}
	
	We now analyse the results that we obtained so far. From \eqref{21}, \eqref{25} we obtain $m_1=m_2=\dots=m_{n-2}=m_{n-1}=:M$. From \eqref{10} we get $l_1=l_2=\dots=l_{n-2}=0$. From \eqref{9} we get $l_{n-1}=2k_{n-1,j}$, for $1\leq j\leq n-2$. From \eqref{13} it follows that $\bar{l}_1=\bar{l}_2=\dots=\bar{l}_{n-2}=\bar{l}_{n-1}=:\bar{l}$. From \eqref{29} we get $k_{i,s}=0$, for $1\leq i \leq n-2$ and $s=1,2,\dots, n-1, s\neq i, i+1$. That means that for $1\leq i \leq n-2$ and $s=1,2,\dots, n-1$, the coefficients $k_{i,s}$ are zero except for $k_{i,i}$ and $k_{i,i+1}$. Finally, from \eqref{37} we have $k_{n-1,1}=k_{n-1,2}=\dots=k_{n-1,n-3}=k_{n-1,n-2}:=\pi$. Gathering together these results, the image of the elements $\tau_i$, for $1\leq i\leq n-1$ under the section $s$ is as follows:
	$$s_*(\tau_i)=\tau_i\beta_i^{k_{i,i}}\beta_{i+1}^{k_{i,i+1}}\sigma^M,\ \text{for}\ 1\leq i \leq n-2$$
	and
	$$s_*(\tau_{n-1})=\tau_{n-1}\beta_1^\pi\beta_2^\pi\cdots\beta_{n-2}^\pi\beta_{n-1}^{k_{n-1,n-1}}\rho^{l_{n-1}}\sigma^M.$$ 
	
	Finally, we examine the relation \textbf{R6.}, where $s_*(q_1^2)=s_*(\tau_1\tau_2\cdots\tau_{n-2}\tau_{n-1}^2\tau_{n-2}\cdots\tau_2\tau_1)$.
	\begin{flalign*}
	s_*(\tau_1\tau_2\cdots\tau_{n-2}\tau_{n-1}^2\tau_{n-2}\cdots\tau_2\tau_1) &=
	s_*(\tau_1\tau_2\cdots\tau_{n-2}\tau_{n-1}) 
	s_*(\tau_{n-1}\tau_{n-2}\cdots\tau_2\tau_1)&
	\end{flalign*}
	We will compute $s_*(\tau_1\tau_2\cdots\tau_{n-2}\tau_{n-1})$ and $s_*(\tau_{n-1}\tau_{n-2}\cdots\tau_2\tau_1)$ separately.
	\begin{flalign*}
	s_*(\tau_1\tau_2\cdots\tau_{n-2}\tau_{n-1}) 
	&=\tau_1\beta_1^{k_{1,1}}\beta_2^{k_{1,2}}
	\tau_2\beta_2^{k_{2,2}}\beta_3^{k_{2,3}}\cdots
	\tau_i\beta_i^{k_{i,i}}\beta_{i+1}^{k_{i,i+1}}\cdots
	\tau_{n-2}\beta_{n-2}^{k_{n-2,n-2}}\beta_{n-1}^{k_{n-2,n-1}}&\\
	&\quad\tau_{n-1}\beta_1^\pi\beta_2^\pi\cdots\beta_{n-2}^\pi\beta_{n-1}^{k_{n-1,n-1}}\rho^{l_{n-1}}\sigma^{(n-1)M}&\\
	&=\tau_1\tau_2\beta_1^{k_{1,1}}\beta_3^{k_{1,2}+k_{2,3}}
	\beta_2^{k_{2,2}}\cdots
	\tau_i\beta_i^{k_{i,i}}\beta_{i+1}^{k_{i,i+1}}\cdots
	\tau_{n-2}\beta_{n-2}^{k_{n-2,n-2}}\beta_{n-1}^{k_{n-2,n-1}}&\\
	&\quad\tau_{n-1}\beta_1^\pi\beta_2^\pi\cdots\beta_{n-2}^\pi\beta_{n-1}^{k_{n-1,n-1}}\rho^{l_{n-1}}\sigma^{(n-1)M}&\\
	&=\tau_1\tau_2\cdots\tau_i\beta_1^{k_{1,1}}\beta_{i+1}^{k_{1,2}+k_{2,3}+\cdots+k_{i,i+1}}
	(\beta_2^{k_{2,2}}\cdots
	\beta_i^{k_{i,i}})&\\
	&\quad\tau_{i+1}\beta_{i+1}^{k_{i+1,i+1}}\beta_{i+2}^{k_{i+1,i+2}}\cdots
	\tau_{n-2}\beta_{n-2}^{k_{n-2,n-2}}\beta_{n-1}^{k_{n-2,n-1}}&\\
	&\quad\tau_{n-1}\beta_1^\pi\beta_2^\pi\cdots\beta_{n-2}^\pi\beta_{n-1}^{k_{n-1,n-1}}\rho^{l_{n-1}}\sigma^{(n-1)M}&\\
	&=\tau_1\tau_2\cdots\tau_i\cdots\tau_{n-2}\beta_1^{k_{1,1}}\beta_{n-1}^{k_{1,2}+k_{2,3}+\cdots+k_{i,i+1}+\cdots+k_{n-2,n-1}}
	(\beta_2^{k_{2,2}}\cdots
	\beta_i^{k_{i,i}}\cdots \beta_{n-2}^{k_{n-2,n-2}})&\\
	&\quad\tau_{n-1}\beta_1^\pi\beta_2^\pi\cdots\beta_{n-2}^\pi\beta_{n-1}^{k_{n-1,n-1}}\rho^{l_{n-1}}\sigma^{(n-1)M}&\\
	&=\tau_1\tau_2\cdots\tau_i\cdots\tau_{n-2}\tau_{n-1}
	\beta_1^{k_{1,1}}\beta_n^{k_{1,2}+k_{2,3}+\cdots+k_{i,i+1}+\cdots+k_{n-2,n-1}}&\\
	&\quad(\beta_2^{k_{2,2}}\cdots
	\beta_i^{k_{i,i}}\cdots \beta_{n-2}^{k_{n-2,n-2}}\beta_{n-1}^{k_{n-1,n-1}})
	\beta_1^\pi\beta_2^\pi\cdots\beta_{n-2}^\pi\rho^{l_{n-1}}\sigma^{(n-1)M}.
	\end{flalign*}
	Setting $\lambda:=k_{1,2}+k_{2,3}+\cdots+k_{i,i+1}+\cdots+k_{n-2,n-1}$, we have:
	\begin{flalign*}
	s_*(\tau_1\tau_2\cdots\tau_{n-2}\tau_{n-1}) &=
	\tau_1\tau_2\cdots\tau_i\cdots\tau_{n-2}\tau_{n-1}&\\
	&\quad\beta_1^{k_{1,1}+\pi}\beta_n^\lambda\beta_2^{k_{2,2}+\pi}\cdots
	\beta_i^{k_{i,i}+\pi}\cdots \beta_{n-2}^{k_{n-2,n-2}+\pi}\beta_{n-1}^{k_{n-1,n-1}}\rho^{l_{n-1}}\sigma^{(n-1)M}&\\
	&=\tau_1\tau_2\cdots\tau_{n-2}\tau_{n-1}\beta_1^{k_{1,1}+\pi-\lambda}&\\
	&\quad\beta_2^{k_{2,2}+\pi-\lambda}\cdots
	\beta_i^{k_{i,i}+\pi-\lambda}\cdots \beta_{n-2}^{k_{n-2,n-2}+\pi-\lambda}\beta_{n-1}^{k_{n-1,n-1}-\lambda}\rho^{l_{n-1}-2\lambda}
	\sigma^{(n-1)M}.
	\end{flalign*}
	We continue by computing $s_*(\tau_{n-1}\tau_{n-2}\cdots\tau_2\tau_1)$.
	\begin{flalign*}
	s_*(\tau_{n-1}\tau_{n-2}\cdots\tau_2\tau_1)
	&=\tau_{n-1}\beta_1^\pi\beta_2^\pi\cdots\beta_{n-2}^\pi\beta_{n-1}^{k_{n-1,n-1}}
	\tau_{n-2}\beta_{n-2}^{k_{n-2,n-2}}\beta_{n-1}^{k_{n-2,n-1}}\cdots
	\tau_i\beta_i^{k_{i,i}}\beta_{i+1}^{k_{i,i+1}}\cdots&\\
	&\quad\tau_2\beta_2^{k_{2,2}}\beta_3^{k_{2,3}}
	\tau_1\beta_1^{k_{1,1}}\beta_2^{k_{1,2}}\rho^{l_{n-1}}\sigma^{(n-1)M}&\\
	&=\tau_{n-1}\tau_{n-2}\beta_1^\pi\beta_2^\pi\cdots\beta_{n-1}^{\pi+k_{n-2,n-1}}\beta_{n-2}^{k_{n-1,n-1}+k_{n-2,n-2}}\cdots
	\tau_i\beta_i^{k_{i,i}}\beta_{i+1}^{k_{i,i+1}}\cdots&\\
	&\quad\tau_2\beta_2^{k_{2,2}}\beta_3^{k_{2,3}}
	\tau_1\beta_1^{k_{1,1}}\beta_2^{k_{1,2}}\rho^{l_{n-1}}\sigma^{(n-1)M}&\\
	&=\tau_{n-1}\tau_{n-2}\cdots\tau_i\beta_1^\pi\beta_2^\pi\cdots\beta_{i+1}^{k_{i,i+1}+\pi}\cdots\beta_{n-1}^{\pi+k_{n-2,n-1}}
	\beta_i^{k_{n-1,n-1}+k_{n-2,n-2}+\cdots+k_{i,i}}\cdots&\\
	&\quad\tau_2\beta_2^{k_{2,2}}\beta_3^{k_{2,3}}
	\tau_1\beta_1^{k_{1,1}}\beta_2^{k_{1,2}}\rho^{l_{n-1}}\sigma^{(n-1)M}&\\
	&=\tau_{n-1}\tau_{n-2}\cdots\tau_i\cdots\tau_2\beta_1^\pi\beta_3^{\pi+k_{2,3}}\cdots
	\beta_{i+1}^{\pi+k_{i,i+1}}\cdots\beta_{n-1}^{\pi+k_{n-2,n-1}}
	&\\
	&\quad\beta_2^{k_{n-1,n-1}+k_{n-2,n-2}+\cdots+k_{i,i}\cdots+k_{2,2}}\tau_1\beta_1^{k_{1,1}}\beta_2^{k_{1,2}}\rho^{l_{n-1}}\sigma^{(n-1)M}&\\
	&=\tau_{n-1}\tau_{n-2}\cdots\tau_i\cdots\tau_2\tau_1\beta_2^{\pi+k_{1,2}}\beta_3^{\pi+k_{2,3}}\cdots
	\beta_{i+1}^{\pi+k_{i,i+1}}\cdots\beta_{n-1}^{\pi+k_{n-2,n-1}}&\\
	&\quad
	\beta_1^{k_{n-1,n-1}+k_{n-2,n-2}+\cdots+k_{i,i}\cdots+k_{2,2}+k_{1,1}}\rho^{l_{n-1}}\sigma^{(n-1)M}.
	\end{flalign*}
	We set $\omega:=k_{1,1}+k_{2,2}+\cdots+k_{i,i}+\cdots+k_{n-2,n-2}+k_{n-1,n-1}$. We are now able to compute $s_*(\tau_1\tau_2\cdots\tau_{n-2}\tau_{n-1}^2\tau_{n-2}\cdots\tau_2\tau_1)$:
	\begin{flalign*}
	s_*(\tau_1\tau_2\cdots\tau_{n-2}\tau_{n-1}^2\tau_{n-2}\cdots\tau_2\tau_1) &=
	s_*(\tau_1\tau_2\cdots\tau_{n-2}\tau_{n-1}) 
	s_*(\tau_{n-1}\tau_{n-2}\cdots\tau_2\tau_1)&\\
	&=\tau_1\tau_2\cdots\tau_{n-2}\tau_{n-1}&\\
	&\quad\beta_1^{k_{1,1}+\pi-\lambda}\beta_2^{k_{2,2}+\pi-\lambda}\cdots
	\beta_i^{k_{i,i}+\pi-\lambda}\cdots \beta_{n-2}^{k_{n-2,n-2}+\pi-\lambda}\beta_{n-1}^{k_{n-1,n-1}-\lambda}&\\
	&\quad\tau_{n-1}\tau_{n-2}\cdots\tau_i\cdots\tau_2\tau_1&\\
	&\quad\beta_1^\omega\beta_2^{\pi+k_{1,2}}\beta_3^{\pi+k_{2,3}}\cdots
	\beta_{i}^{\pi+k_{i-1,i}}\cdots\beta_{n-2}^{\pi+k_{n-3,n-2}}\beta_{n-1}^{\pi+k_{n-2,n-1}}
	\rho^{2l_{n-1}-2\lambda}\\
	&=\tau_1\tau_2\cdots\tau_{n-2}\tau_{n-1}^2\tau_{n-2}\cdots\tau_2\tau_1\beta_{1}^{\omega-k_{n-1,n-1}+\lambda} &\\
	&\quad\beta_2^{k_{1,1}+2\pi+k_{1,2}-k_{n-1,n-1}}\cdots
	\beta_{i}^{k_{i-1,i-1}+2\pi+k_{i-1,i}-k_{n-1,n-1}}\cdots&\\ &\quad\beta_{n-2}^{k_{n-3,n-3}+2\pi+k_{n-3,n-2}-k_{n-1,n-1}}
	\beta_{n-1}^{k_{n-2,n-2}+2\pi+k_{n-2,n-1}-k_{n-1,n-1}}&\\
	&\quad\rho^{2l_{n-1}-2k_{n-1,n-1}}.
	\end{flalign*}
	Now:
	\begin{flalign*}
	s_*(q_1^2) &= q_1\beta_1^{\bar{k}_{1,1}}\cdots\beta_{n-1}^{\bar{k}_{1,n-1}}\rho^{\bar{l}}\sigma^{\bar{m}_1}q_1
	\beta_1^{\bar{k}_{1,1}}\cdots\beta_{n-1}^{\bar{k}_{1,n-1}}\rho^{\bar{l}}\sigma^{\bar{m}_1}=q_1^2\beta_{1}^{\bar{l}}\beta_2^{2\bar{k}_{1,2}}\cdots\beta_i^{2\bar{k}_{1,i}}\cdots\beta_{n-1}^{2\bar{k}_{1,n-1}}\rho^{2\bar{l}}\\
	&=(\tau_1\tau_2\cdots\tau_{n-2}\tau_{n-1}^2\tau_{n-2}\cdots\tau_2\tau_1)\beta_{1}^{m+\bar{l}}\beta_2^{2\bar{k}_{1,2}}\cdots\beta_i^{2\bar{k}_{1,i}}\cdots\beta_{n-1}^{2\bar{k}_{1,n-1}}\rho^{2\bar{l}}.
	\end{flalign*}
	We recall that $q_1^2=(\tau_1\tau_2\cdots\tau_{n-2}\tau_{n-1})\beta_n^m(\tau_{n-1}\tau_{n-2}\cdots\tau_2\tau_1)$, given in Proposition \ref{restrictions}, ($II$). The last equality comes from the relation: $$q_1^2=(\tau_1\tau_2\cdots\tau_{n-2}\tau_{n-1})\beta_n^m(\tau_{n-1}\tau_{n-2}\cdots\tau_2\tau_1)=(\tau_1\tau_2\cdots\tau_{n-2}\tau_{n-1})(\tau_{n-1}\tau_{n-2}\cdots\tau_2\tau_1)\beta_1^m$$ that holds in $B_{n,m}(\mathbb{R}P^2)/\Gamma$.
	Comparing the coefficients of the element $\beta_{1}$ of $B_{n,m}(\mathbb{R}P^2)/\Gamma$ in these two equations, we obtain:
	\begin{equation}\label{39}
	m+\bar{l}=\omega+\lambda-k_{n-1,n-1}.
	\end{equation}
	
	From \eqref{39} we have that $m+\bar{l}=\omega+\lambda-k_{n-1,n-1}$. From the definitions of $\omega$ and $\lambda$ we obtain:
	$$m+\bar{l}=(k_{1,1}+k_{1,2})+(k_{2,2}+k_{2,3})+\cdots+(k_{i,i}+k_{i,i+1})+\cdots+(k_{n-2,n-2}+k_{n-2,n-1}).$$
	In addition, from \eqref{24}, we have $k_{i, i}+k_{i, i+1} = k_{i+1,i+1}+k_{i+1,i+2},\ \text{for}\ 1\leq i < n-2$,
	which gives:
	\begin{equation}\label{41}
	k_{1, 1}+k_{1, 2}=k_{2,2} + k_{2,3}=\cdots=k_{n-3, n-3}+k_{n-3, n-2}=k_{n-2, n-2} + k_{n-2,n-1}.
	\end{equation}
	We conclude that:
	\begin{equation}\label{42}
	m+\bar{l}=(n-2)(k_{1,1}+k_{1,2}).
	\end{equation}
	Moreover, from \eqref{31} we have $k_{i,i}+k_{i, i+1} = -\bar{l}_i+2\bar{k}_{i, i+1}$, for $1\leq i \leq n-2$. Letting $i=n-2$, and using \eqref{13} we get:$$k_{n-2,n-2}+k_{n-2, n-1} = -\bar{l}+2\bar{k}_{n-2, n-1}.$$
	However, from \eqref{8} we know that $\bar{k}_{n-2, n-1}=0$, which in turn implies that: 
	\begin{equation}\label{43}
	k_{n-2,n-2}+k_{n-2, n-1} = -\bar{l}.
	\end{equation}
	
	In conclusion, from \eqref{41}, \eqref{42} and \eqref{43} it follows that $m=(n-1)(k_{1,1}+k_{1,2})$ for $n\geq 4$, and therefore that $m=k(n-1)$ for $k\geq1$, since $m$ is a natural number. The conclusion is valid when $n\geq4$ since we have used the relation \textbf{R5.}, $s_*(\tau_i\tau_j) = s_*(\tau_j\tau_i)$ for $|i-j|>1$ and $1\leq i,j\leq n-1,$ which only holds when $n\geq 4$. 
	
	Nevertheless, we can obtain the same result for $n=3$. Indeed, using the results of relations \eqref{9}, \eqref{10}, \eqref{25}, which also hold in the case $n=3$, we obtain $l_2=2k_{2,1},\ l_1=0$ and that $m_1=m_2:=M$, respectively. Thus, $s_*(\tau_{1})=\tau_{1}\beta_{1}^{k_{1,1}}\beta_{2}^{k_{1,2}}\sigma^M$ and 
	$s_*(\tau_{2})=\tau_{2}\beta_1^{k_{2,1}}\beta_{2}^{k_{2,2}}\rho^{2k_{2,1}}\sigma^M$. It holds that $s_*(\tau_{1})$ and $s_*(\tau_{2})$ have the same structure as the equations
	$s_*(\tau_i)=\tau_i\beta_i^{k_{i,i}}\beta_{i+1}^{k_{i,i+1}}\sigma^M$
	and
	$s_*(\tau_{n-1})=\tau_{n-1}\beta_1^\pi\beta_2^\pi\cdots\beta_{n-2}^\pi\beta_{n-1}^{k_{n-1,n-1}}\rho^{l_{n-1}}\sigma^M$, since $\pi=k_{2,1}$ for $n=3$. Therefore, we can indeed apply \eqref{39}, which follows from relation \textbf{R6}, $s_*(q_1^2)=s_*(\tau_1\tau_2\cdots\tau_{n-2}\tau_{n-1}^2\tau_{n-2}\cdots\tau_2\tau_1)$, that also holds for $n=3$. Thus, from \eqref{39} we obtain that $m+\bar{l}=k_{1,1}+k_{1,2}$, which is \eqref{42}, and from \eqref{43}, it follows that $m=2(k_{1,1}+k_{1,2})$, which is the result in the case $n=3$. So if the short exact sequence \eqref{ms} splits then $m=k(n-1)$, for $n\geq 3$ and $k\geq1$.
\end{proof}

We can obtain further restrictions for the value of $m$ by examining the torsion elements of the group $B_n(\mathbb{R}P^2)$ and $B_{n,m}(\mathbb{R}P^2)$. 

\begin{proposition}\label{rtorsion}
	Let $m>1$ and $n\geq 3$.
	If the following short exact sequence
	\begin{equation*}
	\begin{tikzcd}
	1 \ar[r]& B_m(\mathbb{R}P^2\setminus\{x_1,\dots,x_n\}) \ar[r] & B_{n,m}(\mathbb{R}P^2) \ar[r, "\bar{q}"] & B_n(\mathbb{R}P^2)\ar[r] & 1
	\end{tikzcd}
	\end{equation*}
	splits, then either $m\equiv 0\ \textrm{mod}\ n$ or $m\equiv 1\ \textrm{mod}\ n$.
\end{proposition}

\begin{proof}
	From [\cite{gonccalves2004braid}, Theorem 4], the group $B_n(\mathbb{R}P^2)$ contains a torsion element of order $k$ if and only if $k$ divides $4n$ or $4(n-1)$. Similarly, the group $B_{n+m}(\mathbb{R}P^2)$ contains a torsion element of order $l$ if and only if $l$ divides $4(n+m)$ or $4(n-1+m)$.
	
	Suppose that the homomorphism $\bar{q}$ admits a section $\bar{s}: B_n(\mathbb{R}P^2)\to B_{n,m}(\mathbb{R}P^2)$, where $B_{n,m}(\mathbb{R}P^2)\subset B_{n+m}(\mathbb{R}P^2)$. Since $\bar{s}$ is a section, it follows that any torsion element of $B_n(\mathbb{R}P^2)$ is mapped under $\bar{s}$ to a torsion element of $B_{n,m}(\mathbb{R}P^2)$ of the same order. In particular, since $B_n(\mathbb{R}P^2)$ contains a torsion element of order $4n$, whose image under $\bar{s}$ is a torsion element of order $4n$ in $B_{n,m}(\mathbb{R}P^2)$. Therefore, $4n$ divides $4(n+m)$ or $4(n-1+m)$ by [\cite{gonccalves2004braid}, Theorem 4]. Thus, $$4n\ | \ (4n+4m)\ \text{or}\ 4n\ |\ [4(n-1)+4m].$$ 
	As a result $4n+4m=p\cdot 4n$ or $4(n-1)+4m=\bar{p}\cdot 4n$, for $p,\bar{p}\in \mathbb{Z}$. It follows that
	$m=(p-1)\cdot n$ or $m=(\bar{p}-1)\cdot n+1$, or equivalently that either $m\equiv 0\ \textrm{mod}\ n$ or $m\equiv 1\ \textrm{mod}\ n$. 
\end{proof}

\begin{remark}
	Arguing in the same manner for the torsion element of order $4(n-1)$ in $B_n(\mathbb{R}P^2
	)$, it follows that $(n-1)$ divides $(n+m)$ or $[(n-1)+m]$, so $m=l(n-1)+n$ or $m=\bar{l}(n-1)$, for some $l,\bar{l}\in\mathbb{Z}$, but in fact from Proposition \ref{rquotient}, we know that $m=k(n-1)$, which is a stronger result.
\end{remark}

\begin{proof}[Proof of Theorem \ref{th2}]
The proof follows from Propositions \ref{rquotient} and \ref{rtorsion}.
\end{proof}

\section{Construction of a geometric section for the case $n\geq3$}\label{S6}
In this section, we will prove Theorem \ref{th3}. In \cite{MR4185935}, Chen--Salter constructed geometric cross-sections using M\"{o}bius transformations in the case of the 2-sphere, $\mathbb{S}^2$, for some of the cases given by Gon\c{c}alves--Guaschi in \cite{gonccalves2005braid}. We will carry out a similar construction for the case of the projective plane $\mathbb{R}P^2$.

\begin{proposition}\label{construction1}
	Let $n\geq 3$ and $m>1$. The fibration $q: UF_{n,m}(\mathbb{R}P^2)\to UF_n(\mathbb{R}P^2)$ admits a cross-section for $m=kn(2n-1)(2n-2),\ \text{where}\ k\geq1.$
\end{proposition}

\begin{proof}
	Let $n\geq 3$. We consider the two-sheeted covering map $p:\mathbb{S}^2\to\mathbb{R}P^2$. Let $\{y_1,\dots,y_n\}$ be an element of $UF_n(\mathbb{R}P^2)$. By $A_n(\mathbb{S}^2)$ we denote the set of $n$ pairs of antipodal points on $\mathbb{S}^2$, which is a subset of $UF_{2n}(\mathbb{S}^2)$.  Using $p$, $\bigcup\limits_{i=1}^np^{-1}(\{y_i\})=\{x_1,x_2,\dots,x_{2n}\}$, which is an element in $A_n(\mathbb{S}^2)$. So $\{x_1,x_2,\dots,x_{2n}\}$ consists of $n$ pairs of antipodal points of $S^2$. As we will see, we can transform the problem of the construction of a cross-section on the configuration spaces of the projective plane to the construction of a cross-section on the configuration spaces of the 2-sphere. We will generalise a particular construction, using M\"{o}bius transformations, of a cross-section on the configuration space of the 2-sphere, given by Chen--Salter in [\cite{MR4185935}, Remark 3.6], which induces the cabling map $B_n(\mathbb{S}^2)\to B_{n,m}(\mathbb{S}^2)$, given in [\cite{MR4185935}, Theorem B].
	Considering the element $\{x_1,x_2,\dots,x_{2n}\}\in A_n(\mathbb{S}^2)\subset UF_{2n}(\mathbb{S}^2)$, we will construct new distinct points on $\mathbb{S}^2$ around each point $x_i$, for $i=1,\dots,2n$.
	
	Let $(x_p,x_q,x_r)$ be an ordered triple of distinct points of the set of points $\{x_1,x_2,\dots,x_{2n}\}$ in $A_n(\mathbb{S}^2)$. We consider the unique M\"{o}bius transformation of the Riemann sphere, defined by $M_{x_p,x_q,x_r}(x)=\frac{(x-x_p)(x_q-x_r)}{(x-x_r)(x_q-x_p)}$, for $x\in \mathbb{S}^2$, which sends $(x_p,x_q,x_r)$ to $(0,1,\infty)$. For every $x_i\in \{x_1,x_2,\dots,x_{2n}\}$ we define the following map: 
	$$ R_{x_i}(x):= \prod_{j\neq k\in \{1,\dots,2n\}\setminus\{i\}}M_{x_i,x_j,x_k}(x)=\prod_{j\neq k\in \{1,\dots,2n\}\setminus\{i\}}\frac{(x-x_i)(x_j-x_k)}{(x-x_k)(x_j-x_i)}.$$
	
	For every $x_i\in \{x_1,x_2,\dots,x_{2n}\}$, the map $R_{x_i}:\mathbb{S}^2\to\mathbb{S}^2$ is a rational map and in particular it is a product of $(2n-1)(2n-2)$ rational fractions. Thus by the fundamental Theorem of Algebra, the map $R_{x_i}$ is of degree $(2n-1)(2n-2)$. It follows that the preimage of any regular value of $R_{x_i}$ is a set of $(2n-1)(2n-2)$ distinct points. Moreover, by the definition of $R_{x_i}$, $x_i$ is the only zero of $R_{x_i}$.
	
	By the definition of $M_{x_p,x_q,x_r}(x)$ we have $M_{x_p,x_q,x_r}(x)=M_{-x_p,-x_q,-x_r}(-x)$, for every $x\in \mathbb{S}^2$. For simplicity, in what follows we will denote by $i^\prime$, the index $k$ for which $x_k=-x_i$. Moreover, note that for any pair of $j,k$, where $j,k\neq i^\prime$, and for any factor $M_{x_i,x_j,x_k}(x)$ included in the product $\prod\limits_{j\neq k\in \{1,\dots,2n\}\setminus\{i\}}M_{x_i,x_j,x_k}(x)$, the factor $M_{x_i,-x_j,-x_k}(x)$ is included as well, since $x_j, x_k$ belong in a set of exactly $n-1$ pairs of antipodal points. Furthermore, for $j=i^\prime$ (respectively for $k=i^\prime$), for any factor $M_{x_i,-x_i,x_k}(x)$, where $k\in \{1,\dots,2n\}\setminus\{i,i^\prime\}$ (respectively $M_{x_i,x_j,-x_i}(x)$, where $j\in \{1,\dots,2n\}\setminus\{i,i^\prime\}$), included in the product $\prod\limits_{j\neq k\in \{1,\dots,2n\}\setminus\{i\}}M_{x_i,x_j,x_k}(x)$, the factor $M_{x_i,-x_i,-x_k}(x)$ (respectively $M_{x_i,-x_j,-x_i}(x)$), is included as well.
	Similarly, for any pair of $j,k$, where $j,k\neq i$, and for any factor $M_{-x_i,x_j,x_k}(-x)$ included in $\prod\limits_{j\neq k\in \{1,\dots,2n\}\setminus\{i^\prime\}}M_{-x_i,x_j,x_k}(-x)$, the factor $M_{-x_i,-x_j,-x_k}(-x)$ is included as well, since the source of $x_j, x_k$ is a set of exactly $n-1$ pairs of antipodal points. Furthermore, for $j=i$ (respectively for $k=i$), for any factor $M_{-x_i,x_i,x_k}(-x)$, where $k\in \{1,\dots,2n\}\setminus\{i,i^\prime\}$ (respectively $M_{-x_i,x_j,x_i}(-x)$, where $j\in \{1,\dots,2n\}\setminus\{i,i^\prime\}$), included in the product $\prod\limits_{j\neq k\in \{1,\dots,2n\}\setminus\{i\}}M_{-x_i,x_j,x_k}(-x)$, the factor $M_{-x_i,x_i,-x_k}(-x)$ (respectively $M_{-x_i,-x_j,x_i}(-x)$), is included as well.
	 Therefore, 
	\begin{align*}
    R_{x_i}(x)=&\prod\limits_{j\neq k\in \{1,\dots,2n\}\setminus\{i,i^\prime\}}M_{x_i,x_j,x_k}(x)
	\prod\limits_{k\in \{1,\dots,2n\}\setminus\{i,i^\prime\}}M_{x_i,-x_i,x_k}(x)\\
	&\cdot\prod\limits_{j\in \{1,\dots,2n\}\setminus\{i,i^\prime\}}M_{x_i,x_j,-x_i}(x)
	\end{align*}
	and
	\begin{align*}
	R_{-x_i}(-x)=&
	\prod\limits_{j\neq k\in \{1,\dots,2n\}\setminus\{i,i^\prime\}}M_{-x_i,x_j,x_k}(-x)
	\prod\limits_{k\in \{1,\dots,2n\}\setminus\{i,i^\prime\}}M_{-x_i,x_i,x_k}(-x)\\
	&\cdot\prod\limits_{j\in \{1,\dots,2n\}\setminus\{i,i^\prime\}}M_{-x_i,x_j,x_i}(-x).
	\end{align*}
	Thus, from $M_{x_p,x_q,x_r}(x)=M_{-x_p,-x_q,-x_r}(-x)$, we obtain:
	\begin{equation}\label{antipod}
	R_{x_i}(x)=R_{-x_i}(-x),\ \text{for every}\ x\in \mathbb{S}^2.
	\end{equation}
	
	Near each point $x_i\in \{x_1, x_2,\dots,x_{2n}\}$, we will now construct $(2n-1)(2n-2)$ new distinct unordered points. To do so, we consider a small regular value $\epsilon_i$ of $R_{x_i}$ close to zero and we take its preimage. The preimage, $R^{-1}_{x_i}(\epsilon_i)$, will be a set of $(2n-1)(2n-2)$ distinct points, and by continuity close to $x_i$ and different from it. One needs to check that the $(2n-1)(2n-2)$ new distinct points close to the point $x_i$ are also distinct from the $(2n-1)(2n-2)$ new distinct points close to any other point $x_j$, for $j\neq i$. In other words, we have to chose each regular value in such a way that $R^{-1}_{x_i}(\epsilon_i)\cap R^{-1}_{x_j}(\epsilon_j)=\emptyset$, for every $i\neq j\in \{1,\dots, 2n\}$.
	In order to achieve this, we have to refine the choice of the regular value of every $R_{x_i}$, as we now describe.
	
	To begin with, we denote the spherical metric of $\mathbb{S}^2$ by $d:\mathbb{S}^2\times\mathbb{S}^2\to\mathbb{R}^+$, which is the distance function on the sphere based on the central angle between two points, and we set:
	\begin{equation*}
	m:=\frac{1}{3}\min_{i\neq j\in \{1,\dots, 2n\}}d(x_i,x_j).
	\end{equation*}
	Moreover, we can define a map $f: \{x_1, x_2,\dots, x_{2n}\}\to \mathbb{R}^+$, that assigns a positive small number $r_i$ to every element $x_i\in\{x_1, x_2,\dots, x_{2n}\}$, which is a regular value of the map $R_{x_i}$, such that any $x$ with $0<|x|\leq r_i$ is a regular value of $R_{x_i}$ and for which $$R^{-1}_{x_i}(\bar{D}(0,r_i)):=\{x\in\mathbb{S}^2\ |\ R_{x_i}(x)\in \bar{D}(0,r_i)\}\subset \bar{D}(x_i, m),$$ where $\bar{D}(z,a)=\{y\in \mathbb{S}^2\ |\ d(z, y)\leq a \}$. The inclusion $R^{-1}_{x_i}(\bar{D}(0,r_i))\subset \bar{D}(x_i,m)$ follows by continuity of $R_{x_i}$ and by the fact that $x_i$ is the only zero of $R_{x_i}$. Note that we can always find such a regular value $r_i$ of $R_{x_i}$, since the set of critical values of $R_{x_i}$ is finite.
	Finally, we set $r$ to be the smallest value among the $r_i$, where $1\leq i\leq 2n$, that is:
	\begin{equation}\label{radius}
	r:=\min(r_1,\dots,r_{2n})\in\mathbb{R}^{+}.
	\end{equation}
	Therefore, for every point $x_i\in \{x_1,x_2,\dots,x_{2n}\}$ we have defined a rational map $R_{x_i}$ and a small positive number $r$, which is a regular value of $R_{x_i}$ and for which $$R^{-1}_{x_i}(r)=\{\zeta_1^i,\dots,\zeta_{(2n-1)(2n-2)}^i\}\subset \bar{D}(x_i,m),\ \text{for every}\ i\in\{1,\dots,2n\},$$
	where $\zeta_1^i,\dots,\zeta_{(2n-1)(2n-2)}^i$ are pairwise distinct.
	This implies that $R^{-1}_{x_i}(r)\cap R^{-1}_{x_j}(r)=\emptyset$, for every $i\neq j\in \{1,\dots,2n\}.$ In addition, from $\eqref{antipod}$, we have:
	 $$R^{-1}_{-x_i}(r)= \{-\zeta_1^i,\dots,-\zeta_{(2n-1)(2n-2)}^i\}=-R^{-1}_{x_i}(r)= \{\zeta_1^i,\dots,\zeta_{(2n-1)(2n-2)}^i\}.$$ In other words, the $(2n-1)(2n-2)$ new distinct points around $x_i$ are antipodal to the $(2n-1)(2n-2)$ new distinct points around $-x_i$.
	Now, set:
	 $$S_r:=\bigcup\limits_{1\leq i\leq {2n}}R^{-1}_{x_i}(r),$$
	which is the set of the $2n(2n-1)(2n-2)$ new distinct points of $\mathbb{S}^2$ that we constructed, based on the points of the set $\{x_1,x_2,\dots,x_{2n}\}\in A_n(\mathbb{S}^2)$. Note that $S_r\in A_{n(2n-1)(2n-2)}(\mathbb{S}^2)$. 
	
	Note that if $r^{\prime}, r^{\prime\prime}\in \mathbb{R}^+$ satisfy $0<r^{\prime}<r^{\prime\prime}<r$, then by construction, $S_{r^{\prime}}\cap S_{r^{\prime\prime}}=\emptyset$. So if $k\in\mathbb{N}$, in the same way, we may obtain $k\big((2n(2n-1)(2n-2))\big)$ new points on $\mathbb{S}^2$, taking for example $S_r\cup S_{\frac{1}{2}r}\cup\dots\cup S_{\frac{1}{k}r}$, where once again the $k(2n-1)(2n-2)$ new points close to $x_i$ will be antipodal to the $k(2n-1)(2n-2)$ new points close to $-x_i$.
	
	Projecting now the set of points $S_r\cup S_{\frac{1}{2}r}\cup\dots\cup S_{\frac{1}{k}r}$ into $\mathbb{R}P^2$ by the covering map $p$, around each point $y_1,\dots,y_n$ of $\mathbb{R}P^2$ we obtain  $k((2n-1)(2n-2))$ new distinct unordered points. In total, we have constructed $kn((2n-1)(2n-2))$ new distinct unordered points, that depend continuously on the element $\{y_1,\dots,y_n\}\in UF_n(\mathbb{R}P^2)$. This yields a cross-section $s:UF_n(\mathbb{R}P^2)\to UF_{n,m}(\mathbb{R}P^2)$ of the fibration $q: UF_{n,m}(\mathbb{R}P^2)\to UF_n(\mathbb{R}P^2)$, for $m=kn(2n-1)(2n-2)$, where $k\geq 1$. 
\end{proof}

We now present another construction of a cross-section, which covers smaller values of $m$.

\begin{proposition}\label{constraction2}
	Let $n\geq 3$ and $m>1$. For $m=2n(n-1)$ the fibration $$q:UF_{n,m}(\mathbb{R}P^2)\to UF_n(\mathbb{R}P^2)$$ admits a cross-section.
\end{proposition}

\begin{proof}
	Let $n\geq 3$. As in the proof of Proposition \ref{construction1}, we consider the two-sheeted covering map map $p:\mathbb{S}^2\to\mathbb{R}P^2$. Let $\{y_1,\dots,y_n\}$ be an element of $UF_n(\mathbb{R}P^2)$. Using $p$, $\bigcup\limits_{i=1}^np^{-1}(\{y_i\})=\{x_1,x_2,\dots,x_{2n}\}$, which is an element in $A_n(\mathbb{S}^2)$, where by $A_n(\mathbb{S}^2)$ we denote the set of $n$ pairs of antipodal points, which is a subset of $UF_{2n}(\mathbb{S}^2)$. So, $\{x_1,x_2,\dots,x_{2n}\}$ consists of $n$ pairs of antipodal points of $S^2$. We denote the distance function on the sphere based on the central angle between two points by $d:\mathbb{S}^2\times\mathbb{S}^2\to\mathbb{R}^2$. Clearly, we have $0\leq d(x_i,x_j)\leq \pi$, with $d(x_i,x_i)=0$ and $d(x_i,-x_i)=\pi$, for any $x_i, x_j\in \mathbb{S}^2$. We define the following positive numbers $m$ and $M_{x_i}$:
	$$m:=\frac{1}{3}\min_{i\neq j}d(x_i,x_j)\ \text{and}\ M_{x_i}:=\max_{x_i\neq -x_j}d(x_i,x_j)<\pi,\  \text{for every}\ 1\leq i\neq j\leq 2n.$$
	
	We consider the closed discs $\bar{D}(x_i,m)=\{y\in \mathbb{S}^2\ |\ d(x_i, y)\leq m \}$, for all $1\leq i\leq 2n$. By the definition of $m$, these discs are pairwise disjoint. Our aim is to construct $2n-2$ new distinct unordered points inside each closed disc $\bar{D}(x_i,m)$. Without loss of generality, we will describe this construction for the point $x_1$. The construction will be similar for the other points. We want to move all of the points of the set $\{x_1,\dots x_{2n}\}\setminus\{-x_1\}$ to pairwise distinct points inside the closed disc $\bar{D}(x_1,m)$. 
	The distance from $x_1$ to any point of the set $\{x_1,\dots x_{2n}\}\setminus\{-x_1\}$ is at most $M_{x_1}$, and so $\{x_1,\dots x_{2n}\}\setminus\{-x_1\}\subset P_1$, 
	where $P_1=\{z\in\mathbb{S}^2\ |\ d(x_1, z)\leq M_{x_1} \}$. Let $C_{x_1}=\{z\in\mathbb{S}^2\ |\ d(x_1, z)=M_{x_1} \}$. Moving now this circle $C_{x_1}$ and shrinking the space $P_1$ toward the point $x_1$, along the longitudes with respect to the axis that passes through the centre of the sphere and $x_1$, we move the points of the set $\{x_1,\dots,x_{2n}\}\setminus\{-x_1\}$ to pairwise distinct points inside the closed disc $\bar{D}(x_1,m)$. To do so, we shrink the circle $C_{x_1}$ until it coincides with the boundary of the closed disc $\bar{D}(x_1,m)$. In other words, at the end of the process, the angle $M_{x_1}$ becomes $m$. In this way, we have constructed $2n-2$ new distinct unordered points close to the point $x_1$.	
	We apply the same method for every point of $\{x_2,\dots,x_{2n}\}$ . We thus obtain $2n(2n-2)$ new distinct unordered points on the sphere. Note that these $2n(2n-2)$ new points are distinct, since we constructed $(2n-2)$ distinct points inside $2n$ pairwise disjoint closed discs. Moreover, choosing the same shrinking process for $x_i$ and $-x_i$, the $(2n-2)$ new points that we obtain inside the closed disc $\bar{D}(x_i,m)$ are the antipodal points of the $(2n-2)$ new points that we obtain inside the closed disc $\bar{D}(-x_i,m)$. Thus, by projecting these $2n(2n-2)$ new distinct points into $\mathbb{R}P^2$, via the covering map $p$, we obtain $2n(n-1)$ new distinct unordered points. These points depend continuously on the element $\{y_1,\dots,y_n\}\in UF_n(\mathbb{R}P^2) $. We conclude that there exists a cross-section $s:UF_n(\mathbb{R}P^2) \to UF_{n,m}(\mathbb{R}P^2)$ for the fibration $q: UF_{n,m}(\mathbb{R}P^2)/(S_n\times S_m)\to UF_n(\mathbb{R}P^2)$, for $m=2n(n-1)$. 	
\end{proof}	

\begin{proof}[Proof of Theorem \ref{th3}]
The result follows from Propositions \ref{construction1} and \ref{constraction2}.
\end{proof}

Applying the same argument as in [\cite{gonccalves2004braid}, Proposition 7], the fibration $q: UF_{n,m}(\mathbb{R}P^2)\to UF_n(\mathbb{R}P^2)$ admits a cross-section if and only if the short exact sequence \eqref{ms} splits. Thus, we obtain the following corollary:
\begin{corollary}\label{imply}
Let $n\geq 3$. The short exact sequence \eqref{ms} splits for $m=2n(n-1)$ and $m=kn(2n-1)(2n-2)$, where $k\geq1$.
\end{corollary}
\begin{remark}
For $n=3$, based on Theorem \ref{th2} and Corollary \ref{imply}, the smallest value of $m$, for which the homomorphism $\bar{q}_{3+m,3}:B_{3,m}(\mathbb{R}P^2)\to B_3(\mathbb{R}P^2)$ could possibly admit a section is $m=6$, and moreover, $m=12$ is the smallest known value for which the homomorphism $\bar{q}_{3+m,3}:B_{3,m}(\mathbb{R}P^2)\to B_3(\mathbb{R}P^2)$ admits a section.
\end{remark}

\begin{remark}
	To the best of our knowledge the remaining cases that are not covered by Theorems \ref{th2} and \ref{th3} are open.
\end{remark}
	
\section*{Acknowledgement}	
The author is thankful to John Guaschi for reading through this work and for his very valuable comments.

	\bibliographystyle{plain}
	\bibliography{bibliography.bib}

\begin{thebibliography}{10}

\bibitem{artin1925theorie}
E.~Artin.
\newblock Theorie der {Z}{\"o}pfe.
\newblock {\em Abh. Math. Sem. Univ. Hamburg}, 4:47--72, (1925).

\bibitem{BardakovOn}
V.~G. Bardakov.
\newblock On the theory of braid groups.
\newblock {\em Mat. Sb.}, 183:3--42, (1992).

\bibitem{https://doi.org/10.1002/cpa.3160220104}
J.~S. Birman.
\newblock On braid groups.
\newblock {\em Comm. Pure Appl. Math.}, 22:41--72, (1969).

\bibitem{MR4185935}
L.~Chen and N.~Salter.
\newblock Section problems for configurations of points on the {R}iemann
  sphere.
\newblock {\em Algebr. Geom. Topol.}, 20:3047--3082, (2020).

\bibitem{fadellHomotopy}
E.~Fadell.
\newblock Homotopy groups of configuration spaces and the string problem of
  {D}irac.
\newblock {\em Duke Math. J.}, 29:231--242, (1962).

\bibitem{fadell1962configuration}
E.~Fadell and L.~Neuwirth.
\newblock Configuration spaces.
\newblock {\em Math. Scand.}, 10:111--118, (1962).

\bibitem{fadell1962braid}
E.~Fadell and J.~Van~Buskirk.
\newblock The braid groups of {$E^{2}$} and {$S^{2}$}.
\newblock {\em Duke Math. J.}, 29:243--257, (1962).

\bibitem{falk}
M.~Falk and R.~Randell.
\newblock The lower central series of a fiber-type arrangement.
\newblock {\em Invent. Math.}, 82:77--88, (1985).

\bibitem{fox1962braid}
R.~Fox and L.~Neuwirth.
\newblock The braid groups.
\newblock {\em Math. Scand.}, 10:119--126, (1962).

\bibitem{gonccalves2004braid}
D.~L. Gon{\c{c}}alves and J.~Guaschi.
\newblock The braid groups of the projective plane.
\newblock {\em Algebr. Geom. Topol.}, 4:757--780, (2004).

\bibitem{gonccalves2007braid}
D.~L. Gon{\c{c}}alves and J.~Guaschi.
\newblock The braid groups of the projective plane and the {F}adell--{N}euwirth
  short exact sequence.
\newblock {\em Geom. Dedicata}, 130:93--107, (2007).

\bibitem{gonccalves2010braid}
D.~L. Gon{\c{c}}alves and J.~Guaschi.
\newblock Braid groups of non-orientable surfaces and the {F}adell--{N}euwirth
  short exact sequence.
\newblock {\em J. Pure Appl. Algebra}, 214:667--677, (2010).

\bibitem{gonccalves2012surface}
D.~L. Gon{\c{c}}alves and J.~Guaschi.
\newblock Surface braid groups and coverings.
\newblock {\em J. Lond. Math. Soc.}, 85:855--868, (2012).

\bibitem{gonccalves2005braid}
D.~L. Gon\c{c}alves and J.~Guaschi.
\newblock The braid group {$B_{n,m}({\Bbb S}^2)$} and a generalisation of the
  {F}adell-{N}euwirth short exact sequence.
\newblock {\em J. Knot Theory Ramifications}, 14:375--403, (2005).

\bibitem{guaschiIncl}
D.~L. Gon\c{c}alves and J.~Guaschi.
\newblock Inclusion of configuration spaces in {C}artesian products, and the
  virtual cohomological dimension of the braid groups of {$\Bbb{S}^2$} and
  {$\Bbb{R}P^2$}.
\newblock {\em Pacific J. Math.}, 287(1):71--99, (2017).

\bibitem{guaschi2013survey}
J.~Guaschi and D.~Juan-Pineda.
\newblock A survey of surface braid groups and the lower algebraic {$K$}-theory
  of their group rings.
\newblock In {\em Handbook of group actions. {V}ol. {II}}, Adv. Lect. Math.,
  Vol. 32, pages 23--75. Int. Press, Somerville, MA, (2015).

\bibitem{hansen}
V.~L. Hansen.
\newblock {\em Braids and coverings: selected topics}.
\newblock London Mathematical Society Student Texts, Vol. 18. Cambridge
  University Press, (1989).

\bibitem{johnsonpresentation}
D.~L. Johnson.
\newblock {\em Presentations of groups}.
\newblock London Mathematical Society Student Texts, Vol. 15. Cambridge
  University Press, Cambridge, (1997).

\bibitem{RolfsenKim}
D.~M. Kim and D.~Rolfsen.
\newblock An ordering for groups of pure braids and fibre-type hyperplane
  arrangements.
\newblock {\em Canad. J. Math.}, 55:822--838, (2003).

\bibitem{van1966braid}
J.~Van~Buskirk.
\newblock Braid groups of compact 2-manifolds with elements of finite order.
\newblock {\em Trans. Amer. Math. Soc.}, 122:81--97, (1966).

\bibitem{zariski1936poincare}
O.~Zariski.
\newblock On the {P}oincar{\'e} group of rational plane curves.
\newblock {\em Amer. J. Math.}, 58:607--619, (1936).

\bibitem{zariski1937topological}
O.~Zariski.
\newblock The topological discriminant group of a {R}iemann surface of genus
  {$p$}.
\newblock {\em Amer. J. Math.}, 59:335--358, (1937).

\end{thebibliography}
	
\end{document}